\documentclass[12pt,reqno]{amsart}
\usepackage{amssymb}
\usepackage{cases}
\usepackage{bbm}
\usepackage{mathrsfs}
\usepackage{graphicx}
\usepackage{amsfonts}
\usepackage{enumerate}
\usepackage{hyperref}

 \usepackage[dvipsnames]{xcolor}

\usepackage{amssymb, amsmath, amsfonts, latexsym}
\usepackage{enumerate}
\usepackage{color}

\newtheorem{thm}{Theorem} 
\newtheorem{cor}{Corollary} 
\newtheorem{lem}{Lemma} 
\newtheorem{prop}{Proposition}
\newtheorem{example}{Example} 
\newtheorem{remarks}{Remark} 
\newtheorem{note}{Note} 
\newtheorem{defn}{Definition}
\newtheorem{hyp}{Hypothesis} 
\numberwithin{equation}{section}

\date{}

\def\BB{\mathcal B}

\def\MM{\mathcal M}

\def\<{\langle}
\def\>{\rangle}

\def\d"{^{\prime\prime}}

\def\bhyp{\begin{hyp}}
\def\nhyp{\end{hyp}}

\def\bbeq{\begin{equation}}
\def\nneq{\end{equation}}

\def\bdef{\begin{defn}}
\def\ndef{\end{defn}}

\def\bthm{\begin{thm}}
\def\nthm{\end{thm}}

\def\bprop{\begin{prop}}
\def\nprop{\end{prop}}

\def\brmk{\begin{remarks}}
\def\nrmk{\end{remarks}}

\def\bexa{\begin{example}}
\def\nexa{\end{example}}

\def\blem{\begin{lem}}
\def\nlem{\end{lem}}

\def\bcor{\begin{cor}}
\def\ncor{\end{cor}}

\def\bexe{\begin{exe}}
\def\nexe{\end{exe}}

\def\bprf{\begin{proof}}
\def\nprf{\end{proof}}

\def\bdes{\begin{description}}
\def\ndes{\end{description}}

\def\benu{\begin{enumerate}}
\def\nenu{\end{enumerate}}

\usepackage{geometry}
\begin{document}

 \title[Killed Feynman-Kac semigroups]
 {Long time behavior of killed Feynman-Kac semigroups   with singular  Schrödinger potentials}

\author[A. Guillin]{\textbf{\quad {Arnaud} Guillin$^{\dag}$    }}
\address{{\bf Arnaud Guillin}. Universit\'e Clermont Auvergne, CNRS, LMBP, F-63000 CLERMONT-FERRAND, FRANCE}
 \email{arnaud.guillin@uca.fr}
 
 \author[D. Lu]{\textbf{\quad Di Lu$^{\dag}$  }}
 \address{{\bf Di Lu}.  School of Mathematics and Statistics, Yunnan University, Kunming 650091, China, and School of Mathematical Sciences, Dalian University of Technology, Dalian 116024, China}
 \email{diluMath@hotmail.com}
  
\author[B. Nectoux]{\textbf{\quad Boris Nectoux$^{\dag}$  }}
\address{{\bf Boris Nectoux}.Universit\'e Clermont Auvergne, CNRS, LMBP, F-63000 CLERMONT-FERRAND, FRANCE}
 \email{boris.nectoux@uca.fr}

\author[L. Wu]{\textbf{\quad Liming Wu$^{\dag}$ \, \, }}
\address{{\bf Liming Wu}. Universit\'e Clermont Auvergne, CNRS, LMBP, F-63000  CLERMONT-FERRAND, FRANCE, and, 
Institute for Advanced Study in Mathematics, Harbin Institute of Technology, Harbin 150001, China}
\email{Li-Ming.Wu@uca.fr}

\date{\today}

\begin{abstract} In this work, we investigate the compactness and the long time behavior of killed  Feynman-Kac semigroups of various processes arising from statistical physics  with very general singular Schrödinger potentials.  The processes we consider  cover a large class of processes used in statistical physics, with strong links with quantum mechanics and (local or not) Schrödinger operators (including e.g. fractional Laplacians).  For instance we consider  solutions to elliptic differential equations,      Lévy processes,  the kinetic Langevin process with locally Lipschitz gradient fields, and  systems of interacting Lévy particles. Our analysis relies on a Perron-Frobenius type theorem derived in [A. Guillin, B. Nectoux, L. Wu, 2020 J. Eur. Math. Soc.] for Feller kernels and on  the tools introduced in [L. Wu, 2004, Probab. Theory Relat. Fields] to compute bounds on the essential spectral radius of a bounded nonnegative  kernel. 
\end{abstract}

%%%

\maketitle
\noindent {\it AMS 2010 Subject classifications.}   47D08, 60G51,	37A30,  37A60,     60J60. \\
 \noindent {\it Key words.} Feynman-Kac, singular potentials, Schrödinger,   quasi-stationary distributions, Brownian particle, Lévy processes, kinetic Langevin.  
% 
 %Schrödinger
    %probabilistic representation of the semigroup 
 
 %%%%

\section{Setting and results}

%%%

\subsection{Feynman-Kac semigroups and models}
\label{sec.FK}
%%%
 
 \subsubsection{Setting and purpose of this work}
 The purpose of this work is to study the basic properties and above all,  the  existence and the uniqueness of a quasi-stationary distribution as well as the exponential convergence towards this quasi-stationary distribution (see more precisely Definition \ref{de.2}) of killed Feynman-Kac semigroups of several models with very general singular Schrödinger potentials   (see Section \ref{sec.Sc} and more precisely  \textbf{[S1]}, \textbf{[S2]},  and  \eqref{eq.Ui}).    
  The killing occurs when the process exits  a (position) subdomain\footnote{A domain is by definition a non-empty, open, and connected set.}. We have no assumption on the regularity of the boundary of this subdomain which can be bounded or not and  whose boundary can intersect the points where  the Schrödinger potential  is infinite (which is the case of interest). 
 For the first three considered processes, the singular potential arises from a fixed  point located at $0$ whereas  for the fourth one, it is created by $n$ interacting Lévy  particles (see the fourth  model below in Section~\ref{sec.IVB}). As explained in Section~\ref{sec.Ss}, more general singular potentials can be considered.  
 This work is also motivated by the strong link between Feynman-Kac semigroups  and  solution of evolution equations associated with a  Schrödinger type  operator, see~\eqref{eq.Sch}.

 The introduction and the paper are organized as follows. We first introduce the class of singular Schrödinger potentials we will treat in details  in this work and we then give the  notation which will be used throughout  this work. Next, we introduce the four     models and we give preliminary results on their behaviors. These are the purposes of Section~\ref{sec.FK}. The main results, which are Theorems~\ref{th.oL}, \ref{th.L},  \ref{th.kL},  and~\ref{th.Bn} (see also Theorem \ref{th.Fil}),   are given in Section \ref{sec.MR}. We provide also extensions of these results  in Section~\ref{sec.Ss}. The related literature is given in Section \ref{sec.RR}. Section \ref{sec.proof} is dedicated to the proofs of the main results.

In all this work, the set $(\Omega, \mathcal F, (\mathcal F_t)_{t\ge 0}, \mathbb P)$ is   a  filtered probability space, where the filtration satisfies the usual condition,  $(B_s,s\ge 0)$ is a $\mathscr R^d$-standard Brownian motion,  and\footnote{Except in Section  \ref{sec.sec-ns} where, only in this section,    $d\ge 1$.} 
$$d\ge 2.$$

 \subsubsection{Singular Schrödinger potential}
 \label{sec.Sc}
 As already mentioned above, we will consider the killed Feynman-Kac semigroups of four models used  in statistical mechanics (which are also related to quantum mechanics through their generators) though many other processes  can also be treated with our techniques (see Section \ref{sec.Ss}). For the first three models we will work with a    very general singular (at $0$) Schrödinger potential $\mathbf V_{\mathbf S}$. More precisely,   $\mathbf V_{\mathbf S}:\mathscr R^d\setminus \{0\}\to \mathscr R$ is assumed to be continuous, lower bounded  (say by $-k_{\mathbf S}$, $k_{\mathbf S}\ge 0$) and satisfies one of the two conditions:
\begin{enumerate}
\item[]\textbf{[S1]}  $\mathbf V_{\mathbf S} \to +\infty$ if and only if $|x|\to 0^+$.
 \item[] \textbf{[S2]} $\mathbf V_{\mathbf S} \to +\infty$  if (and only if) $|x|\to 0^+$ or  $|x|\to +\infty$.
 \end{enumerate}
 Assumption \textbf{[S2]} differs from \textbf{[S1]} because the potentials satisfying  \textbf{[S2]} confine also at $+\infty$. 
Notice that we can consider any kind of singularities at $0$, e.g. the Coulomb potential, the Riesz potential,   the Lennard-Jones potential, and the  log-potential.   
 Note that near $0$ the potential $ \mathbf V_{\mathbf S}$ is repulsive. In the fourth model below, the singularities in the Schrödinger potentials  are created  when the interacting (moving) Lévy particles collide, see more precisely \eqref{eq.Ui}. We also mention that our techniques allow to consider more general singularities, see   Section~\ref{sec.Ss} for examples.
 %%%

\subsubsection{Notation}
 Let $\mathscr S$ be a polish space and denote by   $\mathcal B(\mathscr S)$ the   Borel $\sigma$-algebra over  $\mathscr S$.  
In the following, $b\mathcal B(\mathscr S)$ (resp. $\mathcal C_{b}(\mathscr S)$) is the set of  bounded   measurable (resp. bounded and continuous) functions $f:\mathscr S\to \mathscr R$. 
For a measurable function $\mathbf W: \mathscr S\to [1,+\infty]$,  we denote by $b\mathcal B_{\mathbf W}(\mathscr S)$  (resp. $\mathcal C_{b\mathbf W}(\mathscr S)$) the space of measurable (resp. continuous) functions $f:\mathscr S\to \mathscr R$ such that $f/\mathbf W$ is bounded over $\mathscr S$. These spaces are endowed with the norm $\Vert f \Vert _{b\mathcal B_{\mathbf W}(\mathscr S)}:= \sup_{\mathscr S} |f/\mathbf W|$. The function   $\mathbf 1$ denotes the constant function over $\mathscr S$ which equals $1$.   The set $\mathcal P(\mathscr S)$ is the space of probability measures over $\mathscr S$ and $\mathcal P_{\mathbf W}(\mathscr S):=\{\nu \in \mathcal P(\mathscr S), \nu(\mathbf W)<+\infty\}$.  Note that $b\mathcal B_{\mathbf 1}(\mathscr S)=b\mathcal B(\mathscr S)$ and $\mathcal P_{\mathbf 1}(\mathscr S)=\mathcal P(\mathscr S)$.  The space  $\mathcal C([0,T],\mathscr S)$ denotes the  space of continuous    $\mathscr S$-valued functions over $[0,T]$ and  the (Skorokhod) space of   $\mathscr S$-valued functions 
 that are right-continuous and have left-hand limits (say \textit{càdlàg}) over $[0,T]$ is denoted by $\mathcal D([0,T],\mathscr S)$.   
For a bounded linear operator $T$ over a Banach space, we denote by $\mathsf r_{sp}(T)$ its spectral radius, and by $\mathsf r_{ess}(T)$ its essential spectral radius (see e.g.~\cite[Section 3.1]{guillinqsd} for a definition). 
 For a stochastic process $(\mathfrak y_t,t\ge 0)$ and $T\ge 0$, we denote by   $\mathfrak y_{[0,T]}=(\mathfrak y_t,t\in [0,T])$ the trajectory of the process up to time~$T$.

\subsubsection{First model, a laboratory Brownian elliptic model} In this section, we introduce    the overdamped Langevin process \eqref{eq.Lsur}, which is  an  elliptic diffusion driven by a Brownian noise. This is our first model. The  analysis  of its killed Feynman-Kac semigroup will turn out  to be  very instructive, which explains our choice to start with this prototypical model. 
 \medskip
 
 \noindent
\textbf{Model 1.}
For $x\in \mathscr R^d$, consider a single Brownian particle $(X_t,t\ge 0)$  solution to the elliptic stochastic differential equation in $\mathscr R^d$ %\textcolor{blue}{(Extensions:  (a) two brownian particles with $V_s(x,y)=1/|x-y|^a$ [the probability that two independent brownian motions starting at $x\neq y$ meets is known and depends on the dimension $d$], (b) kinetic Langevin process, (c) Lévy $\alpha$-stable process)}:
\begin{equation}\label{eq.Lsur}
dX_t= \mathbf b_{\mathbf c}(X_t)dt + d B_t, \ X_0=x.
\end{equation}
In all this work,  $\mathbf b_{\mathbf c}:\mathscr R^d\to \mathscr R^d$ is   a   locally Lipschitz vector field. We define the following two assumptions: 
   \begin{enumerate}
\item[]\textbf{[c1]}  $\lim_{|x|\to +\infty}  \mathbf b_{\mathbf c}(x)\cdot \frac{x}{|x|}=-\infty$.
 \item[] \textbf{[c2]} $| \mathbf b_{\mathbf c}|$ has at most linear growth over $\mathscr R^d$.
 \end{enumerate} 
 We mention that the case when the potential $\mathbf b_{\mathbf c}$ is singular has already been treated in~\cite{guillinqsd2,guillinqsd3}\footnote{For both elliptic and hypoelliptic processes.} and thus will not  be considered in this work.

\begin{prop}\label{pr.Pre}
Assume  {\rm \textbf{[c1]}} or {\rm \textbf{[c2]}}.  
For all $x\in \mathscr R^d$  there exists a unique strong solution $(X_t(x),t\ge 0)$ to \eqref{eq.Lsur} over $\mathscr R^d$ whose sample paths are a.s. continuous.  
In addition, for all $T>0$,  the following  Girsanov's formula holds:  
\begin{equation}\label{eq.Girsanov}
\frac{d \mathbf P_x}{d\mathbf P_{x}^0}{  \Big|_{\mathcal F_T}}= m_T^0(x)  ,
\end{equation}
where $ \mathbf P_x$ (resp. $ \mathbf P^0_x$)  is the law of  $(X_t(x),t\ge 0)$ (resp. of $(B_s(x)=x+B_s,s\ge 0)$), and for $t\ge 0$, $
     m_t^0(x)=\exp\big[   \int_0^t \mathbf b_{\mathbf c}(B_s(x)) \cdot d B_s-\frac 12 \int_0^t  |\mathbf b_{\mathbf c}(B_s(x))  |^2 ds   \big]$ 
 is the   Doléans-Dade exponential (true) martingale associated with the process $(X_t,t\ge 0)$. 
\end{prop}
 %%%
 
 \begin{proof}
 Assume  {\rm \textbf{[c2]}}. Then, there exists a unique pathwise solution to  \eqref{eq.Lsur}, see e.g.~\cite[Theorem 2.2 in Section 5]{friedman1975}. The Girsanov formula then follows from~\cite[Theorem~3.1 in Section~7]{friedman1975} and~\cite[Theorem~1.1  in Section~7]{friedman1975}. Note also that in this case there exists $C>0$ such that for all $T>0$, it holds  a.s. $|X_t(x)|\le Ce^{T}(|x|+T+ \sup_{u\in [0,T]}|B_u|)$ and therefore
\begin{equation}\label{eq.energy2} 
  \mathbb P_x[\sigma^{oL}_{B(0,R)}\le t]\to 0 \text{ as $R\to +\infty$ uniformly in $x$ in the compact sets}, 
  \end{equation}
   where   $\sigma^{oL}_{B(x,R)}$ is the first exit time from the open ball $B(x,R)$ for the process $(X_t,t\ge 0)$.

 Assume  now {\rm \textbf{[c1]}}. 
 For all $x\in \mathscr R^d$, define $\mathbf L(x)= |x| \times (1-\chi(x))$ where $\chi\in \mathcal C_c^\infty(B(0,1), [0,1])$ and $\chi=1$ on $B(0,1/2)$, and 
\begin{equation}\label{eq.W}
\mathbf W(x)=e^{\epsilon \mathbf L(x)}, \ \epsilon >0.
\end{equation} 
 Denote by  $\mathscr L^{oL}= \mathbf b_{\mathbf c} \cdot \nabla + \frac 12 \Delta$
  the infinitesimal generator of the diffusion \eqref{eq.Lsur}. 
We  have for every  $x\in \mathscr R^d$, 
 \begin{align}
 \nonumber
 \mathscr L^{oL} \mathbf W(x)&= \Big[   \epsilon\mathbf b_{\mathbf c}(x)\cdot  \nabla \mathbf L(x) +\frac 12\epsilon^2| \nabla \mathbf L(x)|^2 +\frac 12 \epsilon \Delta  \mathbf L(x)\Big] \, \mathbf W(x).
 \end{align}
 Since $\nabla \mathbf L$ and $\Delta \mathbf L$ are bounded over $\mathscr R^d$, we have  $\mathscr L^{oL} \mathbf W(x)/\mathbf W(x)\to -\infty$ thanks to  {\rm \textbf{[c1]}}. In particular $\mathscr L^{oL} \mathbf W\le  c \mathbf W$ over $\mathscr R^d$. Together with the fact that  $\mathbf b_{\mathbf c}$ is locally Lipschitz and using well known arguments, this implies the first statement in Proposition~\ref{pr.Pre} and also that for some $c>0$ and all $x\in \mathscr W_R$, 
\begin{equation}\label{eq.energy}
\mathbb P_x[\sigma^{oL}_{\mathscr W_R}\le t]\le \frac{e^{ct}}{R} \mathbf W(x), 
\end{equation}
where for $R>0$, $\mathscr W_R:= \{y\in \mathscr R^d, \mathbf W(y)< R\}$ is an open bounded subset of $\mathscr R^d$ and $\sigma^{oL}_{\mathscr W_R}:=\inf\{t\ge 0, X_t\notin \mathscr W_R\}$. Using the fact  that for all $x\in \mathscr R^d$, a.s. both $\sigma^{oL}_{\mathscr W_R}(x)\nearrow^{+\infty}$ and $\sigma^0_{\mathscr W_R}\nearrow^{+\infty}$ (where $\sigma^0_{\mathscr W_R}:=\inf\{t\ge 0, B_t\notin \mathscr W_R\}$) as $R\to +\infty$, one proves the Girsanov formula  with the same arguments as those used in the proof of~\cite[Proposition 2.2]{cattiaux09} (see also the proof of~\cite[Lemma 1.1]{Wu2001}).    
 \end{proof}
  %%%
  
  % The Markov property follows from   the (pathwise) uniqueness by standard considerations (cf. Cepa) 
%The strong Markov property follows e.g. from the fact that the semigroup has the Feller property (cf. the proof of~\cite[Theorem 6.17]{Legall Brownian motion}, E is assumed to be metrizable locally compact and countable at infinity).  
%  Cf. aussi E. Cepa, Ergodicité des équations stochastiques variationnelles (pas besoin de la continuité en t, cadlag suffit
% Sinon Kallenberg: unicité de la solution faible (car unicité forte) THM 18.11. [C'est pareil avec un Lévy]

Define the space 
\begin{equation}\label{eq.RD0}
\mathscr R^d_0:=\mathscr R^d\setminus \{0\}.
\end{equation}
The space  $\mathscr R^d_0$ is Polish (it is equipped with   a metric  generating the original topology and making $\mathscr R^d_0$ a complete space\footnote{In our setting, it is the metric  
\begin{equation}\label{eq.DD}
\mathsf d_{\mathscr R^d_0}: (x,y)\in (\mathscr R^d_0)^2\mapsto |x-y|+||x|^{-1}-|y|^{-1}|
\end{equation}
 which makes the boundary of $\mathscr R^d_0$ (namely $\{0\}$) looks like $+\infty$. Note that for $x\in \mathscr R^d_0$ and a sequence $(x_n)\subset  \mathscr R^d_0$,  $x_n\to x$ in $\mathscr R^d_0$ if and only if $x_n\to x$ in the original space $\mathscr R^d$, i.e. in $\mathscr R^d_0$ the topology 
 %,  see~\cite[Definition 1.1.10 and its note]{bogachev2007measure}) 
 induced by $\mathsf d_{\mathscr R^d_0}$ coincides with the one of $\mathscr R^d$.}).

 \begin{lem}\label{le.finiteLsur}
 Assume  {\rm \textbf{[c1]}} or  {\rm \textbf{[c2]}}. 
Then, for all $x\in \mathscr R^d_0$ and $T>0$,  $\mathbb P_x[ X_{[0,T]} \in \mathcal C([0,T],\mathscr R^d_0)]=1$, and,  if   {\rm\textbf{[S1]}} or {\rm\textbf{[S2]}} holds,    for all $t\ge 0$, $ \int_0^t  \mathbf V_{\mathbf S}(X_s(x))ds$ is  a.s.  finite. 
\end{lem}

  \begin{proof}
Since $d\ge 2$, it is well known, see e.g.~\cite[Theorem 4.1 in Section 11]{friedman1976} that any point  is   \textit{nonattainable}   for a  $\mathscr R^d$-standard Brownian motion $(W_s,s\ge 0)$  and therefore  it holds for all $y\in \mathscr R^d$, $\mathbb P\big [ W_s=y, \text{ for some } s>0 \big ]=0$. 
Thanks to the Girsanov formula \eqref{eq.Girsanov}, we have for all $T>0$, 
 $
\mathbb P\big [ |X_s(x)|=0, \text{ for some } s\in [0,T] \big ]=0, \forall x\in \mathscr R^d_0
 $.
Thus,  $\mathbb P_{ x}[\forall s\ge 0, X_s\in \mathscr R^d_0]=1$, $\forall x\in \mathscr R^d_0$. By continuity of the trajectories of $(X_s(x),s\ge 0)$ in $\mathscr R^{d}$, for all $x\in \mathscr R^d_0$, we deduce that 
\begin{align}
 \label{eq.Cc}
 &\mathbb P_x[ X_{[0,t]} \in \mathcal C([0,t],\mathscr R^d_0)] =1.
 \end{align} 
 This ends the proof of the Lemma. 
 \end{proof}

 %%%%

  In view of Lemma \ref{le.finiteLsur}, $\mathscr R^d_0$ is  the natural state  space  to study  the Feynman-Kac semigroup   associated  with $\mathbf V_{\mathbf S}$ for the process \eqref{eq.Lsur} (see also Remark \ref{re.D}).
In all this work, $\mathscr O$ is a subdomain of  $\mathscr R^d_0$, and
$$\sigma^{oL}_{\mathscr O}:=\inf\{t\ge 0, X_t\notin \mathscr O\}$$
is the first exit time from $\mathscr O$ for the process $(X_t,t\ge 0)$. 
The  killed (outside $\mathscr O$)  Feynman-Kac semigroup $(Q^{oL}_t, t\ge 0)$  over    $\mathscr R^d_0$ associated  with $\mathbf V_{\mathbf S}$ and  the process \eqref{eq.Lsur}    is then defined by   
\begin{equation}\label{eq.FK-sur1}
 Q^{oL}_tf(x)=\mathbb E_x\Big[f(X_t)   \, e^{  - \int_0^t  \mathbf V_{\mathbf S}(X_s)ds} \mathbf 1_{t<\sigma^{oL}_{\mathscr O}}  \Big], \ \text{$t\ge 0$, $x\in \mathscr O$, and $f\in b\mathcal B(\mathscr O)$}.
 \end{equation} 
 Note that  the confinement at $+\infty$   comes from   the dynamics itself (i.e. from $\mathbf b_{\mathbf c}$) or from  the Schrödinger potential $\mathbf V_{\mathbf S}$. The confinement at $0$ comes from   $\mathbf V_{\mathbf S}$. 
 Note also that there is no killing when  $\mathscr O=\mathscr R^d$. 
 The superscript oL in the previous notation stands for the fact that we consider   the overdamped Langevin process \eqref{eq.Lsur}.

For $t\ge 0$, the transition kernel $Q^{oL}_t(x,dy)$ at time $t\ge 0$ defines, through a natural pairing,  an (adjoint) operator on  the set $\mathcal M_b(\mathscr O)$ of all $\sigma$-additive  measures $\nu$ of bounded variations over $\mathscr O$:
$$\nu  Q_t^{oL} (\mathcal A)=\int_{\mathscr R^d_0} Q_t^{oL}(x, \mathcal A)\,  \nu(dx)= \mathbb E_{\nu}\Big[\mathbf 1_{\mathcal  A}(X_t)   \, e^{  - \int_0^t  \mathbf V_{\mathbf S}(X_s)ds}  \mathbf 1_{t<\sigma^{oL}_{\mathscr O}}  \Big], \ \mathcal A \in \mathcal B(\mathscr O).$$
%Note that in view of Lemma \ref{le.finiteLsur}, for $t\ge 0$ and such measures $\nu$ over $\mathscr O$, $\mathbb E_{\nu} [ e^{  - \int_0^t  \mathbf V_{\mathbf S}(X_s)ds}   ]>0$. 
The killed renormalized Feynman–Kac semigroup is defined by:
\begin{equation}\label{eq.FK-sur2}
\nu  P_t^{oL} (\mathcal A)=\frac{\nu Q_t^{oL} (\mathcal A)}{\mathscr \nu Q_t^{oL}(\mathscr O)}= \frac{ \mathbb E_{\nu}\Big[\mathbf 1_{\mathcal  A}(X_t)   \, e^{  - \int_0^t  \mathbf V_{\mathbf S}(X_s)ds} \mathbf 1_{t<\sigma^{oL}_{\mathscr O}}   \Big]  }{ \mathbb E_{\nu}\Big[ e^{  - \int_0^t  \mathbf V_{\mathbf S}(X_s)ds} \mathbf 1_{t<\sigma^{oL}_{\mathscr O}}   \Big]  }.
\end{equation} 
The main result for this model is Theorem \ref{th.oL}. 

\begin{remarks}\label{re.D} 
The first thing to establish was obviously the space to work in. 

 Lemma \ref{le.finiteLsur} suggests indeed to work outside the singularity of $\mathbf V_{\mathbf S}$, i.e. on $\mathscr R^d_0$. Anticipating a little bit, another reason to work on $\mathscr R^d_0$  is the following. Compactness  properties of the non-conservative semigroup $(Q^{oL}_t, t\ge 0)$ is  (with our techniques) obtained by   constructing a Lyapunov function $\mathbf W$ over the chosen space $\mathscr S$ such that  $(\mathscr L^{oL}-\mathbf V_{\mathbf S}) \mathbf W(x)/\mathbf W(x)\to -\infty$  when $x\to \{\infty\}\cup\{\partial \mathscr S\}$ ($\mathscr L^{oL}-\mathbf V_{\mathbf S}$ being the infinitesimal generator of the   Feynman-Kac semigroup). 
 This definitely  suggests     to work outside the singularity of $\mathbf V_{\mathbf S}$  since we have for free that $\mathbf V_{\mathbf S}(x)\to+\infty$ as $|x|\to 0^+$.  
 Apart from that, note that there exist a large class of singular potentials\footnote{E.g. $|x|^{-b}$, $b>2$. Use indeed the L.I.L $\mathbb P[\limsup_{s\to 0}|B_s|/(2s \log_2(1/s))^{1/2}=1]=1$, see e.g.~\cite[Chapter~II]{revuz2013continuous}.} $\mathbf V_{\mathbf S}$ such that  $\mathbb P_0[ \int_0^t  \mathbf V_{\mathbf S}(B_s)ds=+\infty]=1$ and in this case the associated killed Feynman-Kac semigroup of the Brownian motion will not be topologically irreducible over $\mathscr R^d$. 
 %\footnote{E.g. when $\mathbf V_{\mathbf C}=0$, the law of the iterated logarithm  ($\mathbb P[\limsup_{s\to 0}|B_s|/(2s \log_2(1/s))^{1/2}=1]=1$, see e.g.~\cite[Chapter~II]{revuz2013continuous}) implies that if $b>2$, $\mathbb P[ \int_0^t |B_s|^{-b}ds=+\infty]=1$ for all $t> 0$.}.

 Let us mention that working with Schrödinger potentials in the Kato class\footnote{This is the class of functions $\mathbf V_{\mathbf S}\in L^1_{loc}(\mathscr R^d)$ such that (see e.g.~\cite{carmona1990relativistic}),
\begin{equation}\label{eq.Kato}
 \lim_{t\to 0^+}\sup_{x\in \mathscr R^d}\mathbb E_x\Big [\int_0^t  |\mathbf V_{\mathbf S}|(X_s)ds\Big]=0.
\end{equation}
}
 allows, with different techniques, to work in a state space including the singularity, see e.g.~\cite[Section 3.2]{chung2001brownian} and~\cite{carmona1990relativistic,kaleta2015pointwise} where the spectral analysis of some $L^2(\mathscr R^d)$-symmetric Feynman-Kac semigroups is carried out  in $L^p(\mathscr R^d)$ when the Schrödinger potential belongs to the (local) Kato class (which however allows some singular attractive potentials). Note that the Kato class depends on the underlying process and does not include all kinds of singularities as it is a subset of $L^1_{loc}(\mathscr R^d)$. 
 %(a subset of locally absolutely integrable functions)
 %Apart from that, note that the computation of $\mathbb P[ \int_0^t  \mathbf V_{\mathbf S}(X_s(0))ds<+\infty]$, for $t>0$, is   a tedious task and  many situations can    {\rm a priori} occur 
 %Note also that with our techniques,  to work on $\mathscr R^d$, a necessary (but    not sufficient) condition would be to have  $\mathbb P_0[ \int_0^t  \mathbf V_{\mathbf S}(X_s(0))ds<+\infty]=1$ for all $t> 0$
 %\footnote{E.g. when $\mathbf V_{\mathbf C}=0$, the law of the iterated logarithm  ($\mathbb P[\limsup_{s\to 0}|B_s|/(2s \log_2(1/s))^{1/2}=1]=1$, see e.g.~\cite[Chapter~II]{revuz2013continuous}) implies that if $b>2$, $\mathbb P[ \int_0^t |B_s|^{-b}ds=+\infty]=1$ for all $t> 0$.}.
\end{remarks}

\subsubsection{Second model: the case of Lévy processes}
\label{sec.Lee}
The second model we consider is the killed Feynman-Kac semigroup of   Lévy processes with a singular Schrödinger potential.
\medskip

\noindent
\textbf{Model 2.}  
Let   $(L_s,s\ge 0)$ be Lévy process  over $\mathscr R^d$ and   
$$\sigma^{Le}_{\mathscr O}:=\inf\{t\ge 0, X_t\notin \mathscr O\},$$
where we recall $\mathscr O$ is any subdomain of $\mathscr R^d_0$ (see \eqref{eq.RD0}). 
We make the following assumptions over the Lévy process $(L_s,s\ge 0)$:
\begin{enumerate}
\item[] \textbf{[L1]}  For all $t>0$, $L_t$ admits a density   w.r.t. the Lebesgue measure over $\mathscr R^d$.
\item[] \textbf{[L2]}  The   killed semigroup $(L_t,t\in [0, \sigma^{Le}_{\mathscr O}) )$     is topologically irreducible  over $\mathscr O$, i.e.  for all $t>0$,   all $x\in \mathscr O$,  and all non-empty open subset $O$ of $\mathscr O$,  $\mathbb P_x(L_t\in O, t<\sigma^{Le}_{\mathscr O})>0$. 
\item[] \textbf{[L3]}  If $\mathscr R^d_0\setminus \overline{\mathscr O}$ is nonempty, then $\mathbb P_x[\sigma^{Le}_{\mathscr O}<+\infty]>0$ for some $x\in \mathscr O$. 
\end{enumerate}
We give in the appendix below a  non exhaustive list of   important examples of  Lévy processes over $\mathscr R^d$ ($d\ge 2$) satisfy \textbf{[L1]}, and also  \textbf{[L2]}-\textbf{[L3]} for any subdomain $\mathscr O$ of $\mathscr R^d$.  
For $x\in \mathscr R^d$, we will denote by $(L_s(x),s\ge 0)$ the process $(x+L_s,s\ge 0)$.  
   We start with the following lemma.

%%%
 
\begin{lem}\label{le.finitealpha-s}
Assume {\rm \textbf{[L1]}}. 
Then, for all $z\in \mathscr R^d_0$,  
\begin{equation}\label{eq.No0-alpha}
\mathbb P_{z}[\forall t\ge 0, L_t=0 \text{ or } L_{t^-}=0]=0.
\end{equation}
In addition, for all $x\in \mathscr R^d_0$ and $t\ge 0$,  $\mathbb P_x[ L_{[0,t]} \in \mathcal D([0,t],\mathscr R^d_0)]=1$, and,  if   {\rm\textbf{[S1]}} or {\rm\textbf{[S2]}} holds,    $ \int_0^t  \mathbf V_{\mathbf S}(L_s(x))ds$ is  a.s.  finite.
\end{lem}
%%%

\begin{proof}
%We recall that by~\cite[Theorem 1]{knopova2013note}, 
%if the characteristic exponent  $\Psi$ of a Lévy process without Gaussian exponent satisfies 
%\begin{equation}
%\label{eq.cond-densite}
%\lim_{|\xi|\to +\infty}\frac{\Re (\Psi(\xi))}{\ln(1+|\xi|)}=+\infty,
%\end{equation}
%it  admits a (smooth) density w.r.t. the Lebesgue measure. 
Since for all $s>0$, $L_s$ has a density w.r.t. the Lebesgue measure,
%  (and more precisely,  $\mathbb E_z[f(L_s)]=\int p_s(z-y)f(y)dy$, for every $z\in \mathscr R^d$ and $f\in b\mathcal B(\mathscr R^d$)). 
~\eqref{eq.No0-alpha}  is   a consequence of~\cite[Theorem 3]{kesten} (recall that we assume $d\ge 2$).  Moreover,~\eqref{eq.No0-alpha} implies that for all $z\in \mathscr R^d_0$ and $T>0$,   
\begin{equation}\label{eq.Ty}
\text{there exists a.s. $\epsilon>0$ such that $|L_t(z)|\ge \epsilon$ for all $t\in [0,T]$.}
\end{equation}
 Indeed, assume that there exist two sequences $\epsilon_n\to 0^+$ and $t_n\in [0,T]$ such that $|L_{t_n}(z)|\le \epsilon_n$. Up to extracting a subsequence,  we assume that $t_n\to t_*\in [0,T]$. 
 One of the two sets $S^*_-=\{n\ge 1, t_n<t_*\}$ and $S^*_+=\{n\ge 1, t_*\le t_n\}$  has an infinite cardinality.   Since the trajectories of $(L_s(x),s\ge 0)$ are  right-continuous and have left limits everywhere, if $S^*_-$ has an infinite cardinality, then, $L_{t_*^-}(z)=0$  whereas if $S^*_+$ has an infinite cardinality, $L_{t_*^-}(z)=0$, but this occurs with null probability. This proves \eqref{eq.Ty}. The proof of the lemma is complete. 
\end{proof}

%%%

%\begin{remarks}\label{re.Att}
%In view of~\cite[Theorems 1, 2, and 3]{kesten} more general Lévy processes can be considered as soon as: lemma \ref{le.finitealpha-s} is still valid and the semigroup is strong Feller.    In this work, for sake of simplicity, we stick to  a $\alpha$-stable processe since it is   a  sufficiently rich example. 
%\end{remarks}

%Consider over $\mathscr R^d$ a potential $\mathbf U_{\mathbf S} = \mathbf V_{\infty} + \mathbf V_{\mathbf S}$ where $\mathbf V_{\mathbf S}$ is given by \eqref{eq.HS} and    
%\begin{equation}\label{eq.VCC}
%\text{$\mathbf V_{\infty}:\mathscr R^d\to \mathscr R$ is a continuous coercive potential.}
%\end{equation}
%By coercive we mean here: $\mathbf V_{\infty}(x)\to +\infty$ when $|x|\to +\infty$. Note that $-\mathbf U_{\mathbf S}$ is upper bounded over $\mathscr R^d_0$. 
The killed Feynman-Kac semigroup $(Q^{Le}_t, t\ge 0)$ over the   space  $\mathscr O$ associated  with the potential $ \mathbf V_{\mathbf S}$  and the  process    $(L_s,s\ge 0)$      is defined by 
\begin{equation}\label{eq.FK-B1}
Q_t^{Le}f(x)=\mathbb E_x\Big[f(L_t)   \, e^{  - \int_0^t  \mathbf V_{\mathbf S} (L_s)ds}  \mathbf 1_{t<\sigma^{Le}_{\mathscr O}}   \Big], \ \text{$t\ge 0$, $x\in \mathscr O$, and $f\in b\mathcal B(\mathscr O)$}.
 \end{equation}
%Note that here, both the confinement at $+\infty$  and at $0$ comes from the   Schrödinger potential $\mathbf U_{\mathbf S}$. 
Note also that there is no killing when $\mathscr O=\mathscr R^d_0$ since in this case, due to Lemma \ref{le.finitealpha-s}, for all $x\in \mathscr R^d_0$, a.s. $\inf\{t\ge 0, X_t\notin \mathscr R^d_0\}= +\infty$.  
 The associated killed renormalized Feynman–Kac semigroup is then defined by, for $\nu\in \mathcal M_b(\mathscr O)$ and $\mathcal A \in \mathcal B(\mathscr O)$, $ \nu  P^{Le}_t (\mathcal A)= {\nu Q^{Le}_t (\mathcal A)}/{\mathscr \nu Q^{Le}_t(\mathscr O)}$. The main result for this killed semigroup  is Theorem \ref{th.L}.

%%%%%
%%%%%

%%%%%
%%%%%
%%%%%

\subsubsection{The third model: a prototypical  hypoelliptic process} 
The third model we consider is the Feynman-Kac semigroup associated with  the kinetic Langevin process \eqref{eq.Lcin} and with the singular Schrödinger potential $\mathbf V_{\mathbf S}$.
\medskip

\noindent
\textbf{Model 3.}  Let  $\gamma>0$ and $(Y_t=(x_t,v_t),t\ge 0)$ be the solution to the stochastic differential equation in $\mathscr R^d\times \mathscr R^d$:
\begin{equation}\label{eq.Lcin}
dx_t=v_tdt, \ dv_t=-\nabla \mathbf V_{\mathbf c}(x_t)dt-\gamma v_tdt + dB_t,
\end{equation} 
 where $\mathbf V_{\mathbf c}:\mathscr R^d\to [1,+\infty)$ is   a  differentiable function such that   $\nabla \mathbf V_{\mathbf c}$ is locally Lipschitz. 
The kinetic Langevin process (also  called the underdamped Langevin process) is a prototypical kinetic   diffusion which is widely used in statistical physics and in molecular dynamics~\cite{lelievre2016partial}. The long time behavior of such a process as well as its ergodic properties are now well known, see for instance~\cite{talay1,Wu2001,mattingly2002ergodicity,herau2004isotropic,herzog,lu2019geometric,Stoltz-2}. In all this work, we write $\mathsf y=(x,v)$ for  an element in $\mathscr R^d\times \mathscr R^d$. The Hamiltonian of the process~\eqref{eq.Lcin} is, for  $\mathsf y=(x,v)\in \mathscr R^{2d}$,  $\mathbf H(x,v)=\mathbf V_{\mathbf c}(x)+\frac12|v|^2$. 
We introduce the process  $(Y^0_t=(x^0_t,v^0_t),t\ge 0)$   solution to the stochastic differential equation in $\mathscr R^d\times \mathscr R^d$:
\begin{equation}\label{eq.Lcin0}
dx^0_t=v^0_tdt, \ dv^0_t=  dB_t.
\end{equation} 
Since $\mathbf V_{\mathbf c}$ is lower bounded,  we have the following result from~\cite[Lemma 1.1]{Wu2001}. 
%%%

\begin{prop}\label{pr.PreKin} 
Then, for all $\mathsf y\in \mathscr R^{2d}$,   there exists a unique strong solution $(Y_t(\mathsf y),t\ge 0)$ to \eqref{eq.Lcin} over $\mathscr R^{2d}$ whose sample paths are a.s. continuous.  
Finally, for all $T>0$,   the following  Girsanov's formula holds for all $T>0$,  
\begin{equation}\label{eq.Girsanov2}
\frac{d \mathbf P_{\mathsf y}}{d\mathbf P_{\mathsf y}^0}{  \Big|_{\mathcal F_T}}= \mathbf  m^0_T(\mathsf y)  ,
\end{equation}
where $ \mathbf P_{\mathsf y}$ (resp. $ \mathbf P^0_{\mathsf y}$)  is the law of  $(Y_t(\mathsf y),t\ge 0)$ (resp. of $(Y_t^0(\mathsf y),t\ge 0)$, see \eqref{eq.Lcin0}), and $(\mathbf m_t^0,t\ge 0)$ is the   Doléans-Dade exponential (true) martingale defined by $\mathbf m_t^0=\exp\big[  - \int_0^t   \big(\gamma v_s^0+\nabla \mathbf V_{\mathbf c}(x_s^0)\big)dB_s -\frac 12 \int_0^t  \big\vert   \gamma v_s^0+\!\nabla \mathbf V_{\mathbf c}(x_s^0)\big\vert^2 ds   \big]$.
 \end{prop} 
 In all this work, when we will deal with the process \eqref{eq.Lcin}, 
 the state space we will consider is 
\begin{equation}\label{eq.r2d0}
\mathscr R^{2d}_0:=(\mathscr R^d\setminus \{0\})\times \mathscr R^d= \mathscr R^d_0\times \mathscr R^d\  \text{ (see \eqref{eq.RD0})}.
\end{equation}
Consider a metastable set for the dynamics \eqref{eq.Lcin}, where we recall that such sets are of the form  $\mathscr D=\mathscr O\times \mathscr R^d$ ($\mathscr O$ being a subdomain of $\mathscr R^d_0$), see~\cite{guillinqsd,lelievre2022eyring}. Set  
$$\sigma^{kL}_{\mathscr D}:=\inf\{t\ge0, Y_t\notin  \mathscr D\}=\inf\{t\ge 0, x_t\notin  \mathscr O\}.$$ 
% We will consider the killed (outside $\mathscr D$) Feynman-Kac semigroup over  $\mathscr D$ associated  with $\mathbf V_{\mathbf S}$ for the process $(Y_t(\mathsf y),t\ge 0)$ solution to \eqref{eq.Lcin} (see \eqref{eq.FK-cin1} below). 
 We have the following result.  
 %%%
  \begin{lem}\label{le.finiteLcin}
For all $\mathsf y\in \mathscr R^{2d}_0$,  
\begin{equation}\label{eq.nn}
\mathbb P_{\mathsf y}[\forall t\ge 0, |x_t|>0]=1.
\end{equation}
In particular, for all $\mathsf y\in \mathscr R^{2d}_0$ and $t\ge 0$, $\mathbb P_{\mathsf y}[Y_{[0,t]} \in \mathcal C([0,t], \mathscr R^{2d}_0)]=1$, and,  if   {\rm\textbf{[S1]}} or {\rm\textbf{[S2]}} holds,   $ \int_0^t  \mathbf V_{\mathbf S}(x_s(\mathsf y))ds$ is  a.s.  finite. 
\end{lem}
Lemma \ref{le.finiteLcin} shows that the unbounded  closed set $M=\{0\}\times \mathscr R^d$ is, using the terminology introduced in \cite{friedman1976}, nonattainable.  The full degeneracy of the noise on the orthogonal of $M$ as well as the fact that $M$ is not the boundary of a $\mathcal C^3$ bounded domain  prevent us from using directly the results of  \cite{friedman1976}. We will rather rely on direct arguments as well as nonattainability results for Gaussian processes. 
%We mention that an alternative strategy could be to use the general result~\cite[Lemma 2.1 in Chapter 11]{friedman1976}. 
The proof of Lemma~\ref{le.finiteLcin} is given in Section~\ref{sec.prp}.

The killed (outside $\mathscr D$) Feynman-Kac semigroup $( Q_t^{kL}, t\ge 0)$ over  $\mathscr D$ associated  with $\mathbf V_{\mathbf S}$ and  the process $(Y_t,t\ge 0)$ solution to \eqref{eq.Lcin}     is defined by:
\begin{equation}\label{eq.FK-cin1}
 Q_t^{kL}f(\mathsf y)=\mathbb E_{\mathsf y}\Big[f(Y_t)   \, e^{  - \int_0^t  \mathbf V_{\mathbf S}(x_s)ds}  \mathbf 1_{t<\sigma^{kL}_{\mathscr D}}   \Big], \ \text{$t\ge 0$, $\mathsf y\in \mathscr D$, and $f\in b\mathcal B(\mathscr D)$}.
 \end{equation}
Its associated killed renormalized Feynman–Kac semigroup is then defined by, for $\nu\in \mathcal M_b(\mathscr D)$ and $\mathcal A \in \mathcal B(\mathscr D)$, $ \nu  P^{kL}_t (\mathcal A)= {\nu Q^{kL}_t (\mathcal A)}/{\mathscr \nu Q^{kL}_t(\mathscr D)}$. The main result for this killed semigroup  is Theorem \ref{th.kL}.
%$ \nu  Q^B_t (\mathcal A)= {\nu Q^B_t (\mathcal A)}/{\mathscr \nu Q^B_t(\mathscr U)}$, $\mathcal A \in \mathcal B(\mathscr U)$ and  $\nu \in \mathcal M_b(\mathscr V)$. 
%%%

\subsubsection{The last model: interacting Lévy particles} 
%%%

\label{sec.IVB}
So far, in the three   previous models, the potential is created by a fixed point (located at $0$). We now consider the situation when it is no longer the case. More precisely, we consider the following archetypical  model of interacting particles.  
\medskip

\noindent
\textbf{Model 4.}  
Let      $(L^{ j}_t,t\ge 0)_{j=1,\ldots,n}$ be  $n\ge 2$    independent $\mathscr R^d$-copies ($d\ge 2$) of a Lévy process $(L_t,t\ge0)$ satisfying \textbf{[L1]}.   
Denote by 
   \begin{equation}\label{eq.eD}
   \mathscr E:=\{\mathsf x=(x_1,\ldots,x_n) \in (\mathscr  R^d)^n \text{ s.t. } \forall i\neq j, x_i\neq x_j\}.
         \end{equation} 
The set $\mathscr E$ is a non-empty open (unbounded) connected subset of $(\mathscr  R^d)^n$  (recall that $d\ge 2$) and its boundary is $\partial \mathscr E:=\cup_{i<j} \{\mathsf x \in (\mathscr  R^d)^n, x_i=x_j\}$. 
  Let $\mathbf U_{\mathbf S}$ be the Schrödinger potential defined by 
   \begin{equation}\label{eq.Ui}
   \mathbf U_{\mathbf S}: \mathsf x=(x_1,\ldots,x_n) \in  \mathscr E \mapsto \sum_{i=1}^n  \mathbf V_{\infty}(x_i) +   \sum_{i<j}^n  \mathbf v_{\mathbf S} (|x_i-x_j|),
     \end{equation} 
   where   
 \begin{equation}\label{eq.VCC}
 \text{$\mathbf V_{\infty}:\mathscr R^d\to \mathscr R$ is a continuous coercive potential},
  \end{equation} 
   and for some  $k_{\mathbf S}\ge 0$:  
 \begin{equation}\label{eq.HS2}
 \mathbf v_{\mathbf S}:\mathscr R_+^*\to \mathscr R  \text{ is continuous, $ \mathbf v_{\mathbf S}\ge -k_{\mathbf S}$, and }   \mathbf v_{\mathbf S}(u)\to+\infty \text{ iff } u \to 0^+.
 \end{equation} 
 By coercive we mean here: $\mathbf V_{\infty}(x)\to +\infty$ when $|x|\to +\infty$.  The form of the potential $\mathbf U_{\mathbf S}$  is typically the one used in molecular dynamics. 
   Note that $\mathbf U_{\mathbf S}$ is continuous and lower bounded over $\{\mathbf U_{\mathbf S}<+\infty\}=\mathscr E$, and $\mathbf U_{\mathbf S}(\mathsf x)\to +\infty$ if and only if $|\mathsf x|\to \infty$ or $\mathsf x\to \partial \mathscr E$.  For   $t\ge 0$ and $\mathsf x=(x_1,\ldots,x_n) \in (\mathscr  R^d)^n$, consider the process $( \Theta_t(\mathsf x),t\ge 0)$ defined by:
   \begin{equation}\label{eq.theta}
   \Theta_t(\mathsf x):= (x_1+L_t^{  1},\ldots,x_n+L_t^{  n}).
      \end{equation} 
Let $\mathscr U$ be a subdomain of  $\mathscr E$ and set  $\sigma^\Theta_{\mathscr U}:=\inf \{t\ge 0, \Theta_t\notin \mathscr U\}$. The killed (outside $\mathscr U$) Feynman-Kac semigroup $(  Q^\Theta_t, t\ge 0)$ over  $\mathscr U$ associated  with $\mathbf U_{\mathbf S}$ and  the Lévy process $(\Theta_t,t\ge 0)$  is defined by:
\begin{equation}\label{eq.FK-Ll}
 Q^\Theta_tf(\mathsf x)=\mathbb E_{\mathsf x}\Big[f(\Theta_t)   \, e^{  - \int_0^t  \mathbf U_{\mathbf S}(\Theta_s)ds}  \mathbf 1_{t<\sigma^\Theta_{\mathscr U}}   \Big], \ \text{$t\ge 0$, $\mathsf x \in \mathscr U$, and $f\in b\mathcal B(\mathscr U)$},
 \end{equation}
 and its killed renormalized  semigroup is 
$ \nu  P^\Theta_t (\mathcal A)= {\nu Q^\Theta_t (\mathcal A)}/{\mathscr \nu Q^\Theta_t(\mathscr U)}$, $\mathcal A \in \mathcal B(\mathscr U)$ and  $\nu \in \mathcal M_b(\mathscr U)$. 
\begin{lem}\label{le.Scc}
For all $\mathsf x\in  \mathscr E$ and $T\ge 0$,  $
\mathbb P_{\mathsf x}[ \Theta_{[0,T]} \in \mathcal D([0,T],\mathscr E)]=1$ and $ \int_0^T  \mathbf U_{\mathbf S}(\Theta_s(\mathsf x))ds$ is  a.s.  finite. 
\end{lem}
%{eq.Cc}
\begin{proof} For all $i\neq j$ and $t\ge 0$, set  $\Upsilon_t^{i,j}=L^{ j}_t-L^{ i}_t$. Then, $(\Upsilon_t^{i,j}, t\ge 0)$ is a Lévy process and for $t>0$, $\Upsilon_t^{i,j}$ has a density w.r.t. the Lebesgue measure. Therefore, from~\cite[Theorem~3]{kesten},  we have for all  $\mathsf x\in \mathscr E$ and for all $i\neq j$,  
$\mathbb P_{\mathsf x}[\exists t\ge  0, \Upsilon_{t^-}^{i,j}=x_i-x_j \text{ or } \Upsilon_t^{i,j}=x_i-x_j ]=0$.  Fix $\mathsf x\in \mathscr E$, $i\neq j$, and $t\ge 0$. 
Then   there exists a.s. $\epsilon>0$  such that  
${\rm dist}\, (\Theta_s(\mathsf x), \partial \mathscr E )\ge \epsilon$, for all $s\in [0,t]$. Else there exists $t_*\in [0,T]$ such that either ${\rm dist}\, (\Theta_{t_*^-}(\mathsf x), \partial \mathscr E )=0$ or  ${\rm dist}\, (\Theta_{t_*}(\mathsf x), \partial \mathscr E )=0$, but this event has null probability.  
%\mathscr E est fermé
%there exists $\epsilon_n\to 0^+$  and $t_n\in [0,T]$ such that ${\rm dist}\, (\Theta_{t_n}(\mathsf x), \partial \mathscr E )\le \epsilon_n$. Up to extracting a subsequence,  we assume that $t_n\to t^*\in [0,T]$. 
 %One of the two sets $S^*_-=\{n\ge 1, t_n<t^*\}$ and $S^*_+=\{n\ge 1, t^*\le t_n\}$  has an infinite cardinality.   If $S^*_-$ has an infinite cardinality, then, ${\rm dist}\, (\Theta_{t^-}(\mathsf x), \partial \mathscr E )=0$,   whereas if $S^*_+$ has an infinite cardinality, ${\rm dist}\, (\Theta_{t_*}(\mathsf x), \partial \mathscr E )=0$, but this as null probability.  
%\begin{align*}
%\mathbb P_{\mathsf x}[\exists s> 0, \Theta_s(x)\in \partial \mathscr E] &= \mathbb P_{\mathsf x}[\exists s\ge  0 \text{ and } i\neq j, L^{ j}_s-L^{ i}_s=x_i-x_j]\\
%&\le \sum_{i< j} \mathbb P_{\mathsf x}[\exists s\ge  0, L^{ j}_s-L^{ i}_t=x_i-x_j]=0,
%\end{align*}
%  which follows from~\cite[Theorem 3]{kesten}. 
  %there exist a.s. $\epsilon>0$ and $M>0$ such that $ |\Theta_s(\mathsf x)|\le M$  and ${\rm dist}\, (\Theta_s(\mathsf x), \partial \mathscr E )\ge \epsilon$, for all $s\in [0,t]$
  \end{proof}
This fourth  model is motivated, when considering e.g. Brownian motions,  by its strong  relation with  the Schrödinger  operator $-\Delta + \frac 12\sum_{i<j}|x_j-x_i|^{-1}$ with Dirichlet boundary conditions over $\mathscr U$. Given a subdomain $\mathscr U$ of $\mathscr E$, we define the assumptions
\begin{enumerate}
\item[] \textbf{[L4]}  $(\Theta_t({\mathsf x}),t\in [0, \sigma^\Theta_{\mathscr U}) )$     is topologically irreducible  over $\mathscr U$, i.e.  for all $t>0$,   all $\mathsf x\in \mathscr U$,  and all non-empty open subset $O$ of $\mathscr U$,  $\mathbb P_{\mathsf x}(L_t\in O, t<\sigma^\Theta_{\mathscr U})>0$.  
\item[] \textbf{[L5]}  If $\mathscr E\setminus \overline{\mathscr U}$ is nonempty, then $\mathbb P_{\mathsf x}[\sigma^\Theta_{\mathscr U}<+\infty]>0$ for some ${\mathsf x}\in \mathscr U$. 
\end{enumerate}
%%%
\textit{Inter alia},
%entre autres
 processes satisfying  \textbf{[L4]} and \textbf{[L5]} include the cases when $(L_t,t\ge0)$ is: (i) a standard Brownian motion,  (ii) a jump diffusion process, or (iii) rotationally invariant $\alpha$-stable processes, see the appendix. 
The main result for the killed semigroup $(  Q^\Theta_t, t\ge 0)$  is Theorem \ref{th.Bn}.

%%%%%

%%%%%

%%%%%

\subsection{Main results}
  \label{sec.MR}

 \subsubsection{Some definitions} 
\label{sec.Xx} 
 Let $\mathscr M$ be an open subset of $\mathscr R^k$, $k\ge 2$, and $(\mathfrak X_t, t\ge 0)$  be a   time homogeneous strong Markov process  with càdlàg sample paths in $\mathscr M$. Consider a lower bounded  and continuous Schrödinger potential  $\mathbf V:\mathscr M \to \mathscr R$ and a subdomain $\mathscr V$ of $\mathscr M$. Let  $\sigma^{\mathfrak X}_{\mathscr V}:=\inf\{t\ge 0, \mathfrak X_t\notin \mathscr V\}$ be the first exit time from $\mathscr V$ for the process $(\mathfrak X_t, t\ge 0)$. 
 Assume that for all $\mathfrak z\in \mathscr M$ and $t\ge 0$, 
\begin{equation}\label{eq.C--1}
\mathbb P_{\mathfrak z}\big [\int_0^t  \mathbf V (\mathfrak X_s) ds <+\infty\big ]=1.
\end{equation}
 The associated killed Feynman-Kac semigroup   is then defined by 
\begin{equation}\label{eq.C--2}
Q^{\mathfrak X}_tf(\mathfrak z)=\mathbb E_{\mathfrak z} \big [f(\mathfrak X_t)   \, e^{  - \int_0^t  \mathbf V (\mathfrak X_s)ds}  \mathbf 1_{t<\sigma^{\mathfrak X}_{\mathscr V}} \big ],
\end{equation}
  for $\mathfrak z\in \mathscr V$  and $f\in b\mathcal B(\mathscr V)$, and the killed renormalized Feynman-Kac semigroup is: 
  $$ \nu  P^{\mathfrak X}_t (\mathcal A)= \frac{\nu Q^{\mathfrak X}_t (\mathcal A)}{\mathscr \nu Q^{\mathfrak X}_t(\mathscr V)}, \text{   for   $\mathcal A \in \mathcal B(\mathscr V)$ and $\nu \in \mathcal M_b(\mathscr V)$}.$$ 
  Note that the killed Feynman-Kac semigroup 
  is strongly related to the solution of the evolution equation  
\begin{equation}\label{eq.Sch}
\partial_tg= \mathscr L g- \mathbf V g,
\end{equation}
with absorbing boundary conditions, 
  where $\mathscr L-\mathbf V$ is the so-called    Schrödinger operator and $\mathscr L$ the infinitesimal generator of the process $(\mathfrak X_t,t\ge 0)$.

 We  recall   the notion of quasi-stationary distribution  (q.s.d.)  of  absorbed Markov chains introduced    in the context of population  processes~\cite{collet2012quasi,meleard2012quasi,champagnat2021lyapunov}, an object which is also at the heart of the analysis of  metastable processes~\cite{di-gesu-lelievre-le-peutrec-nectoux-17,DLLN-saddle1,DLLN}.

\noindent
\begin{defn}
\label{def.QSD}
  A measure $ \mu \in \mathcal P(\mathscr V)$ is a  q.s.d. for the killed renormalized Feynman–Kac semigroup $(P_t,t\ge 0)$   over $ \mathscr V $ if $
\mu P^{\mathfrak X}_t(\mathcal A)=\mu(\mathcal A)$, $\forall t\ge 0$ and $\forall \mathcal A\in \mathcal B(\mathscr V)$. 
\end{defn}
 
 In order to easily state our result, we introduce the notion of   \textit{compact-ergodic} operator.

\begin{defn}\label{de.2}
Let  $\mathbf F: \mathscr V\to [1,+\infty]$ be a measurable function. 
   We say that  $(Q^{\mathfrak X}_t,t\ge 0)$ is    $\mathbf F$ {\rm compact-ergodic} over $\mathscr V$   if:
\begin{enumerate}
 \item[\textbf 1.] There exists a unique q.s.d.  $\rho$ for  $(P^{\mathfrak X}_t,t\ge 0)$ in   $  \mathcal P_{\mathbf F}(\mathscr V)$.  
 \item[\textbf 2.] 
  For all $t>0$, $Q^{\mathfrak X}_t:b\mathcal B_{\mathbf F}(\mathscr V)\to b\mathcal B_{\mathbf F}(\mathscr V)$ is compact and there exists $\lambda\ge  \inf_{\mathscr V} \mathbf V$   such that $\mathsf r_{sp}(Q^{\mathfrak X}_t|_{b\mathcal B_{\mathbf F}(\mathscr V)} ) =e^{-\lambda t}, \ \forall t>0$. 
  Furthermore,  $\rho    Q^{\mathfrak X}_t =e^{-\lambda t}\rho$, for all $t\ge 0$, and $\rho (O)>0$ for all nonempty open subsets $O$ of $ \mathscr V$. In addition,  there is a unique   function $\varphi \in \mathcal C_{b\mathbf F}( \mathscr V)$   such that  $\rho   (\varphi )=1$ and $
      Q^{\mathfrak X}_t \varphi= e^{-\lambda  t} \varphi  \text{ on } \mathscr V, \forall t\ge0$.  
 Moreover, $\varphi>0$   everywhere on $\mathscr V$.
\item[\textbf 3.] There exist   $\mathfrak m_1>0$, and   $\mathfrak m_2\ge 1$, s.t.  for all  $t>0$ and all   $\nu\in \mathcal P_{\mathbf F  }( \mathscr V)$:
$$
\sup_{\mathcal A\in\mathcal B(\mathscr V)}\big |  \nu P^{\mathfrak X}_t(\mathcal A)-\rho  (\mathcal A)\big |\le \mathfrak m_2 \, e^{-\mathfrak m_1 t} \frac{\nu(\mathbf F )}{\nu(\varphi)}.
$$
\end{enumerate}
 \end{defn}
 The real  number  $\lambda$ (resp. the function $\varphi$) in the above  definition is usually called the \textit{principal eigenvalue} (resp.  the \textit{principal eigenfunction}) of $(Q^{\mathfrak X}_t,t\ge 0)$ in   $b\mathcal B_{\mathbf F}(\mathscr V)$. When $\mathbf F=\mathbf W^{1/p}$,  $p\in (1,+\infty)$, we   write these two quantities   $\lambda_p$ and $\rho_p$. 
We will simply say that $(Q^{\mathfrak X}_t,t\ge 0)$ is \textit{compact-ergodic over $\mathscr V$} when $(Q^{\mathfrak X}_t,t\ge 0)$ is $\mathbf 1$ compact-ergodic over $\mathscr V$. 
%Finally, define $i_{\mathbf V,\mathscr V}:=\inf_{\mathscr V} \mathbf V$, we say that   \textbf{[S$_{\mathbf V,\mathscr V}$]} holds if 
% \begin{itemize}
% \item[] \textbf{[O1]} When $\mathscr V\$\mathbb P_{\mathfrak z}[\sigma_{\mathscr V}  <+\infty]>0$ for some $\mathfrak z\in \mathscr V$. Either $\mathbf V \neq i_{\mathbf V,\mathscr V}$ over $\mathscr V$, or, if it is not the case,  
% \item[] \textbf{[O2]} $\mathbb P_{\mathfrak z}[\sigma_{\mathscr V}  <+\infty]>0$ for some $\mathfrak z\in \mathscr V$.
% \end{itemize}

%%%
  \subsubsection{Main result for the first model}
   
 The first main result of this section concerns the long time behavior of the killed  semigroup   defined in \eqref{eq.FK-sur1}. Recall   \eqref{eq.RD0}, Lemma \ref{le.finiteLsur}, and that $\mathscr O$ is a subdomain of $\mathscr R^d_0$.  
 
 \begin{thm}\label{th.oL} Let   $(Q_t^{oL},t\ge 0)$ be defined in \eqref{eq.FK-sur1}. Then:
\begin{enumerate}
\item [] \textbf{Case 1.} Assume {\rm \textbf{[c1]}}. Assume also {\rm \textbf{[S1]}} or {\rm \textbf{[S2]}}.  Consider the function $\mathbf W$ defined in \eqref{eq.W}.
Then, 
for all  $p\in (1,+\infty)$,    $(Q_t^{oL},t\ge 0)$   is $\mathbf W^{1/p}$ compact-ergodic over $\mathscr O$. If moreover  $\mathscr R^d_0\setminus \overline{\mathscr O}$ is non-empty or  if  $\mathscr O=\mathscr R^d_0$, then $\lambda_p>  \inf_{\mathscr O} \mathbf V_{\mathbf S}$  where $\lambda_p$
 is the principal eigenvalue of $(Q^{oL}_t,t\ge 0)$ in   $ b\mathcal B_{\mathbf W^{1/p}}(\mathscr O)$.
 \end{enumerate}
Actually when {\rm \textbf{[S2]}} holds we have the following stronger result.
 \begin{enumerate}
 \item [] \textbf{Case 2.} Assume   {\rm \textbf{[S2]}}. Assume also {\rm \textbf{[c1]}} or {\rm \textbf{[c2]}}. Then,    $(Q_t^{oL},t\ge 0)$   is  compact-ergodic over $\mathscr O$. If moreover  $\mathscr R^d_0\setminus \overline{\mathscr O}$ is non-empty or  if  $\mathscr O=\mathscr R^d_0$, then $\lambda>  \inf_{\mathscr O} \mathbf V_{\mathbf S}$  where $\lambda$
 is the principal eigenvalue of $(Q^{oL}_t,t\ge 0)$ in   $ b\mathcal B (\mathscr O)$.
 \end{enumerate}
  \end{thm}
  
  In both cases the confinement at $0$ comes from  $\mathbf V_{\mathbf S}$. 
In the first   case the confinement at $+\infty$ comes from $\mathbf V_{\mathbf c}$ (and also  from $\mathbf V_{\mathbf S}$  if {\rm \textbf{[S2]}} holds). In the second case  the confinement at $+\infty$ comes from $\mathbf V_{\mathbf S}$ (and also  from $\mathbf V_{\mathbf c}$  if {\rm \textbf{[c1]}} holds).

Note that in the first case above,    any deterministic initial conditions $x\in \mathscr O$ belongs to $\mathcal P_{\mathbf W ^{1/p}}(\mathscr O)$ and thus,   the  q.s.d. $\rho  _p\in \mathcal P_{\mathbf W ^{1/p}}(\mathscr O)$ and the principal eigenvalue $\lambda_p $ of $(Q_t^{oL},t\ge 0)$ in $ b\mathcal B_{\mathbf W^{1/p}}(\mathscr O)$    are independent of $p>1$. We also mention that in the first case above, Theorem~\ref{th.oL}  also holds with  a smaller Lyapunov function than \eqref{eq.W} when $\mathbf V_{\mathbf c}$ grows sufficiently fast at $+\infty$, see \eqref{eq.alternativeW}. 
Theorem \ref{th.oL} is proved in Section~\ref{sec.Sec-th1}.

\subsubsection{Non singular Schrödinger  potential}  
 \label{sec.sec-ns}

Assume in this section (and    only in this section) that $d\ge 1$. Let  $\mathscr O$ be any subdomain of $\mathscr R^d$. Here,  we consider the   much simpler case when    the Schrödinger  potential    is not singular. Let $\mathbf J_{\mathbf S}:\mathscr R^d\to \mathscr R$ be a continuous   lower bounded and coercive potential. 
The killed  Feynman-Kac semigroup    over    $\mathscr O$ associated  with $\mathbf J_{\mathbf S}$ and   \eqref{eq.Lsur} is defined by:
\begin{equation}\label{eq.OL2}
 S_t^{oL}f(x)=\mathbb E_x\Big[f(X_t)   \, e^{  - \int_0^t  \mathbf J_{\mathbf S}(X_s)ds}    \mathbf 1_{t<\sigma^{oL}_{\mathscr O}}  \Big],   \ \text{$t\ge 0$, $x\in \mathscr O$, and $f\in b\mathcal B(\mathscr O)$}.
\end{equation}

 \begin{prop}\label{th.1-bis} 
Assume   {\rm \textbf{[S2]}}. Assume also {\rm \textbf{[c1]}} or {\rm \textbf{[c2]}}.  
Then, the semigroup $( S_t^{oL},t\ge 0)$ given by \eqref{eq.OL2}  is  compact-ergodic over $\mathscr R^d$  (see Definition \ref{de.2}). Moreover,  $\lambda>  \inf_{\mathscr R^d} \mathbf J_{\mathbf S}$ where $\lambda$
 is the principal eigenvalue of $( S_t^{oL},t\ge 0)$  in   $  b\mathcal B(\mathscr R^d)$.
 \end{prop} 

% Si $|\nabla \mathbf V_{\mathbf c}|$ has at most linear growth:
% - existence et unicité d'une solution forte : THM 2.2. p104 Friedman  vol 1
%  - Girsanov: appliquer Friedman + E[ exp  a \sup |B_s|^2 -> non attainiabilité de 0
% Pour t>0,  Growall-> P_z[\tau_{B(0,R)}\le t]\to 0 pour R\to +\infty et unifo en z \in K (tq  \sup_K|z|<R). Car pour tout t>0 fixé,  \sup_[0,t]|X_s| (z)\le C_{t,K}(1+\sup_[0,t]|B_s|)
% convergence en proba des trajectoires: facile (article generalized NH processes)
% le semi groupe de X_t est strong Feller: faire comme Liming en 2001 (cf. lapreuve dans le cas gamma>0 article generalized NH processes.
%  le semi groupe tué de X_t est Topo Irreductible. Facile avec girsanov: le brownien l'est
% on a bien que lim_s sup_K P_x(sigma_B\le s)=0
% $\lambda>  \inf_{\mathscr R^d} \mathbf J_{\mathbf S}$ car J_S non constant sur R^d et le processus visite tout R^d

The proof of  Proposition \ref{th.1-bis}  is a direct adaptation of Theorem \ref{th.oL} in a much simpler setting, and is therefore omitted. 
Note that  if  $\mathbf J_{\mathbf S}$  is only continuous and  lower bounded, then when   $\mathbf b_{\mathbf c}$  satisfies  \textbf{[c1]} (in this case the confinement at $\infty$ comes from the dynamics itself), the semigroup $(  S_t^{oL},t\ge 0)$ given by \eqref{eq.OL2} is $\mathbf W^{1/p}$ compact-ergodic over $\mathscr R^d$ (see  \eqref{eq.W} and \eqref{eq.alternativeW}), for all $p\in (1,+\infty)$.

%%%%

  \subsubsection{Main result for the second   model: a killed   Lévy  particle in a singular potential}

The second main result of this section concerns the  killed Feynman-Kac semigroup $(Q_t^{Le}, t\ge 0)$ defined in \eqref{eq.FK-B1}, where we recall    \eqref{eq.RD0}, \eqref{eq.No0-alpha}, \eqref{eq.Ty}, and that $\mathscr O$ is a subdomain of $\mathscr R^d_0$. 
 
 \begin{thm}\label{th.L} 
Assume  {\rm \textbf{[L1]}}, {\rm \textbf{[L2]}},  and {\rm \textbf{[S2]}}.  Then, $(Q_t^{Le},t\ge 0)$ is  compact-ergodic over $\mathscr O$ (see Definition \ref{de.2}). If $\mathscr R^d_0\setminus \overline{\mathscr O}$ is non-empty (in this case we assume in addition {\rm \textbf{[L3]}}) or  if  $\mathscr O=\mathscr R^d_0$, then $\lambda> \inf_{\mathscr O} \mathbf V_{\mathbf S}$ where $\lambda$
 is the principal eigenvalue of $(Q_t^{Le},t\ge 0)$   in   $ b\mathcal B(\mathscr O)$. 
 
  \end{thm}

 %%%
 The proof of Theorem \ref{th.L} is made in Section \ref{sec.alp}. 
 
 \subsubsection{Main result for the third model: a killed kinetic particle in a singular potential}

The third main result concerns the   killed Feynman-Kac semigroup \eqref{eq.FK-cin1} of  the kinetic Langevin process \eqref{eq.Lcin} with singular potentials, where we recall that $\mathscr D=\mathscr O\times \mathscr R^d$  and that $\mathscr O$ is a subdomain of $\mathscr R^d_0$ (see \eqref{eq.RD0}).   Recall also Lemma \ref{le.finiteLcin}.

   %%%
   
    \begin{thm}\label{th.kL} 
Assume that $-\nabla  \mathbf V_{\mathbf c}$ satisfies {\rm \textbf{[c1]}}. Assume also {\rm \textbf{[S1]}} or {\rm \textbf{[S2]}}.   Let $\mathbf W$ be  defined in \eqref{eq.Wkl} below. Then, for all $p\in (1,+\infty)$,  $(Q_t^{kL},t\ge 0)$ is $\mathbf W^{1/p}$ compact-ergodic over $\mathscr D$. If in addition   $\mathscr R^d_0\setminus \overline{\mathscr O}$ is non-empty or  if  $\mathscr O=\mathscr R^d_0$, then $\lambda_p>  \inf_{\mathscr O} \mathbf V_{\mathbf S}$ where $\lambda_p$
 is the principal eigenvalue of $(Q^{kL}_t,t\ge 0)$ in   $  b\mathcal B_{\mathbf W^{1/p}}(\mathscr D)$. 
   \end{thm}
  
The q.s.d. $\rho_p$ and the principal eigenvalue   $\lambda_p$  of $(Q^{kL}_t,t\ge 0)$ in   $  \mathcal P_{\mathbf W^{1/p}}(\mathscr D)$ do not depend on $p>1$. 
   The proof of Theorem \ref{th.kL} is made in Section \ref{sec.th.3}.

   \begin{note}\label{no.kL}
All the assertions of Theorem \ref{th.kL}  hold true without assuming that   $\nabla  \mathbf V_{\mathbf c}$  is locally Lipschitz, namely when $ \mathbf V_{\mathbf c}$ is only    continuously differentiable and $-\nabla  \mathbf V_{\mathbf c}$ satisfies {\rm \textbf{[c1]}}, see  Remark \ref{re.th3a}. 
   \end{note}
   
   %%%%

    \subsubsection{Main result for the last model: killed  Lévy  particles interacting through singular potential}
 Let       $(L^{ j}_t,t\ge 0)_{j=1,\ldots,n}$ be $n\ge 2$    independent $\mathscr R^d$-copies of a Lévy process satisfying \textbf{[L1]}. Let $\mathscr U$ be a subdomain of $\mathscr E$ and assume that the process \eqref{eq.theta} satisfies \textbf{[L4]}. Recall also Lemma \ref{le.Scc}. We then have the following result.

  \begin{thm}\label{th.Bn}
Assume    \eqref{eq.VCC} and \eqref{eq.HS2}.  Then, $(Q_t^\Theta,t\ge 0)$, see \eqref{eq.FK-Ll}, is  compact-ergodic over $\mathscr U$ (see Definition \ref{de.2}).  If $\mathscr E \setminus \overline{\mathscr U}$ is non-empty (in this case we assume {\rm \textbf{[L5]}}) or  if  $\mathscr U=\mathscr E$, then $\lambda> \inf_{\mathscr U} \mathbf U_{\mathbf S}$ where $\lambda$
 is the principal eigenvalue of $(Q_t^\Theta,t\ge 0)$   in   $ b\mathcal B(\mathscr U)$.
  \end{thm}
  
  The proof of  Theorem \ref{th.Bn}  is  made in Section \ref{sec.4}.

  \subsection{More general  Schrödinger  potentials}\label{sec.Ss}
  \begin{sloppypar}
   We chose for the   three first models    to deal with  a Schrödinger  potential  $\mathbf V_{\mathbf S}$ with a unique singularity   located at $0$. 
  As it is  clear from the proofs of the  main theorems above, one  can easily adapt our techniques to  Schrödinger  potential  $\mathbf V_{\mathbf S}$ having much more singularities (say a closed set $\mathscr C$)  as soon these singularities are  {nonattainable} for the \textit{càdlàg} process $(\mathfrak X_t,t\ge 0)$ we consider (namely $\mathbb P_{\mathfrak  z}[\exists t\ge 0, \mathfrak X_t \in \mathscr  C  \text{ or } \mathfrak X_{t^-} \in \mathscr  C]=0$ for all $\mathfrak z\notin \mathscr  C$), see e.g. \cite{friedman1976,kesten,morters} for examples of such sets $\mathscr C$.

   For instance,  consider      the vertical axis  $\mathscr C=\{(0,0,t), t\in \mathscr R\}$ in $\mathscr R^3$ and the potential (see \eqref{eq.HS2}):
     $$\mathbf C_{\mathbf S}: x \in   \mathscr  R^3 \setminus \mathscr C\mapsto   \mathbf v_{\mathbf S}({\rm dist}\, (x,\mathscr C)),$$ 
     where ${\rm dist}\, (x,\mathscr C)=(|x_1|^2+|x_2|^2)^{1/2}$, $x=(x_1,x_2,x_3)$.  
 When $\mathbf v_{\mathbf S}(y)=-\log y$ for  $0<y<r_0$, $\mathbf C_{\mathbf S}$ is the potential generated by   the infinite line $\mathscr C$ with a uniform charge density per unit length. 
 Let $(B^j_t,t\ge 0)_{j=1,2,3}$ be three independent  standard Brownian motions in $\mathscr R$ and set $B_t=(B^1_t,B^2_t,B^3_t)$. Note that $\mathbb P_x[\exists s\ge 0,  |x_1+B_s^1| +|x_2+B^2_s| =0]= \mathbb P_x[\exists s\ge 0,  -(B_s^1,B^2_s) =(x_1,x_2)]=0$ for all $x\in \mathscr  R^3 \setminus \mathscr C$. 
 Then, if moreover    \eqref{eq.VCC} holds,
  we have the following result, whose proof is an easy adaptation of the proof of Theorem \ref{th.oL}, and is therefore omitted.
 
\begin{thm}\label{th.Fil}
The (non-killed) semigroup 
     $$Q_t^{\mathscr C}f(x):= \mathbb E_x\Big[f(B_t)   \, e^{  - \int_0^t  (\mathbf C_{\mathbf S}+\mathbf V_{\infty})(B_s)ds}   \Big], \ \text{$t\ge 0$, $x\in \mathscr R^3 \setminus \mathscr C$, and $f\in b\mathcal B(\mathscr R^3 \setminus \mathscr C)$},$$
   is  compact-ergodic over $\mathscr R^3  \setminus \mathscr C$ and $\lambda >\inf_{\mathscr R^3  \setminus \mathscr C}(\mathbf C_{\mathbf S}+\mathbf V_{\infty})$.
   \end{thm}
  One can also add a killing (outside some subdomains  of $\mathscr R^3  \setminus \mathscr C$) in the previous theorem.    
We finally mention that the previous result is just an example and one  can also choose more complicated singular potentials. 
  \end{sloppypar}

   %%%%%%

\subsection{Related results}
\label{sec.RR}
 Feynman-Kac semigroups   appear in many fields of science ranging e.g. from statistical physics,  engineering science,  nonlinear filtering,  and genealogical tree  models, see the classical textbook~\cite{del2004feynman}, and the  related mathematical literature is very rich. Such semigroups   have also   a strong connection with   Schrödinger operators  and non linear Fokker–Planck equations~\cite{nagasawa2012stochastic,del2003particle,del2004feynman}. 
To discuss the related literature, we adopt the following notation. Let us consider  a potential $\mathbf V_\star$ and a process $(Z_t,t\ge 0)$, we denote by    $Q_tf(x)=\mathbb E_x [ f (Z_t)   \, \exp(  -\int_0^t  \mathbf V_\star (Z_s)ds)  ]$ its associated (non-killed) Feynman-Kac semigroup.  

 Related to our first model are the following contributions.
When  $(Z_s,s\ge 0)$    is a  Brownian motion (this is our second model when $\alpha=2$), the   generator of the non-killed Feynman-Kac semigroup     is the Schrödinger operator  $\Delta-\mathbf V_\star$, and we refer to the classical textbooks \cite{Kato,CFKS,nagasawa2012schrodinger,simon2015quantum,Yosida,chung2001brownian} for the spectral study of  this operator.  In~\cite{rousset2006control},  the  interacting particle model   introduced  in~\cite{del2000branching} is used  to approximate, via Feynman–Kac formulas,   the principal eigenvalue of Schrödinger type operators as well as its associated eigenfunction.  In \cite{ferre2}, the authors  studied the long time behavior of (non-killed) Feynman-Kac semigroups for  solutions to elliptic diffusions with smooth coefficients and with non singular  Schrödinger  potentials. We   mention that in \cite{ferre2}, the authors are also concerned by the study of   discretized (in time) Feynman-Kac semigroups.  In~\cite{collet2024branching}, the authors  studied in a one-dimensional setting the long time properties of a non-killed  Feynman-Kac  semigroup   related to  linear functionals of branching diffusion processes when  the potential $\mathbf V_\star$ is non singular.  Note that Proposition \ref{th.1-bis} provides  a similar result as~\cite[Theorem 2.1]{collet2024branching} in any dimension and without assumption  on the growth of the  Schrödinger  potential $\mathbf J_{\mathbf S}$ at infinity.

 Our second model fits into the following literature.
When  $(Z_s,s\ge 0)$    is a (Lévy) rotationally invariant $\alpha$-stable process with $\alpha\in (0,2)$, the   generator of the non-killed Feynman-Kac semigroup $Q$     is   $-(-\Delta)^{\alpha/2}-\mathbf V_\star  $ where  $(-\Delta)^{\alpha/2}$ is the fractional Laplacian. When $\textbf V_\star$ is non negative and locally bounded, the spectral properties in $L^2(\mathscr R^d)$ of  such a semigroup have been investigated   in~\cite{kaleta2010intrinsic}  (see also the references therein) and the same investigation has been carried out    in~\cite{kulczycki2006intrinsic}      in the spaces $L^p(\mathscr R^d)$  when  $(Z_s,s\ge 0)$  is  a relativistic stable process  (see also~\cite{chen2000intrinsic,chen2015intrinsic,chen2016intrinsic,ascione2024bulk}). We  refer to~\cite{kaleta2015pointwise} for a   spectral analysis in $L^p(\mathscr R^d)$ of $L^2(\mathscr R^d)$- (non-killed) symmetric Feynman-Kac semigroups associated with some symmetric Lévy processes\footnote{Satisfying  in particular  \textbf{[L1]}  and other regularity conditions on the Green function.} and   when the Schrödinger potential belongs to the Kato class (see \eqref{eq.Kato}), see also \cite{daubechies1983one,carmona1990relativistic,guneysu2011feynman}.  We mention that in the papers quoted just above, the authors are   also concerned with the  precise study of the eigenfunctions and to  contractivity-type properties. In this direction, we also refer to~\cite{vsikic2006potential,bogdan-book,chen2011heat} and references therein. We also mention the recent work \cite{kaleta2023quasi} for the study  of the quasi-ergodicity  in $L^2(\mathscr M)$ of ultracontractive semigroups on   locally compact Polish space $\mathscr M$  with various applications to (non-killed) Feynman-Kac semigroups with locally bounded, lower bounded, and confining Schrödinger potential.

 In conclusion, regarding our first and our second model, we mostly differ from the aforementioned works  in three ways. First, our approach is based on   the Lyapunov type  framework introduced in~\cite{guillinqsd} and our mathematical setting  is that of the space of bounded measurable functions. Second, the class of  singular  Schrödinger    potentials we consider is very large (e.g. we do    not restrict ourselves to Schrödinger  potentials in the Kato class).  Indeed, we recall that any singular potential can be considered as long as its singularity is not reachable by the  process (see   Section \ref{sec.Ss}). Lastly, the aforementioned works  do not consider hard killing. Let us mention that another advantage of our techniques is that they allow us    to make no assumptions about the regularity of the domain outside of which the process is killed. Finally, to the best of our knowledge, the long time behavior of the killed Feynman-Kac semigroups of the third and the fourth models had not   yet been considered in the literature.
 %, and  Theorems \ref{th.kL}  and \ref{th.Bn}     are new as well as Theorem \ref{th.Fil}.  }

 The stability and the contraction properties of  non-killed Feynman Kac semigroups as well as  their large time behavior of various models ranging from physics, biology, evolutionary computing, and nonlinear filtering   have also been investigated in~\cite{del2002stability,del2004feynman,chopin2011stability}  (see also \cite{del2018sharp}).    
  In \cite{ferre2}, the authors also provide several conditions   to obtain the long time convergence of non-killed Feynman-Kac  semigroups. Among these conditions is the Lyapunov criterion  \cite[Assumption 4: Eqs (31) and (32)]{ferre2}.  When working on the set  $ \mathscr M$ where the singular Schrödinger potential  is finite, \cite[Assumption 4]{ferre2} has now to be satisfied when $\mathfrak z\to  \partial \mathscr M \cup \{\infty\}$, $\mathfrak z\in \mathscr M$. It turns out that~\cite[Equation (32)]{ferre2} is actually too challenging for $\mathfrak z\to  \partial \mathscr M$. For that reason, and inspired by the work~\cite{guillinqsd,guillinqsd2}, we prove using the tools introduced in \cite{Wu2004}, that is enough to only work with the    Lyapunov condition~\cite[Equation (31)]{ferre2} adapted to the singular setting, see indeed Theorems \ref{th.Ress} and~\ref{th.G} below.    We also refer to~\cite{champagnatALEA,del2018exponential} for related results.  
%Less important,  we   also mention that we will not need in our proofs that the density of the Feynman-Kac semigroup  is everywhere positive as it is assumed in \cite[Assumption 5]{ferre2}.
In this context, we  mention~\cite{champagnat2017general,CloezHarris} for other  general criteria to study the long-time behavior of non-conservative semigroups with applications to quasi-stationary distributions and growth-fragmentation semigroups, see also \cite{MeynFK2,del2023stability}. These different techniques can also be used to study Feynman-Kac semigroups. 
 Finally,  let us mention that quasistationarity has applications in  Monte Carlo inference problems. We refer to~\cite{pollock2020quasi,wang2019theoretical} where quasi-limiting convergence is used as the basic ingredient to build efficient algorithms  to sample from a distribution of interest $\pi$ (typically  the Bayesian posterior distribution). More precisely, given a diffusion process $(\zeta_t,t\ge 0)$, the main idea there is to construct a suitable  (soft) killing time $\tau_\kappa$ with   rate $\kappa$ (see~\cite[Eq. (1.9)]{wang2019theoretical}) such that $\mathbb P[\zeta_t\in \cdot | t<\tau_\kappa]\to \pi(\cdot)$. 
%%%

% Feynman-Kac path distributions, interacting particle systems, and genealogical tree based models.

%wu1994feynman
%del2002stability

   %%%%

   \section{Proof of the main results}
 \label{sec.proof}

 \subsection{Preliminary regularity results on $Q_t^{oL}$ and proof of Theorem \ref{th.oL}}
 \label{sec.Sec-th1}

 %The second case is an easy adaptation where the Lyapunov condition is satisfied with the constant function $\mathbf 1$. 
 
 %%%%%

  In this section we prove Theorem \ref{th.oL}.   We start   by giving some regularity  and some spectral properties of the  transition kernel $Q^{oL}_t$, this is the purpose of Section \ref{sec.rp}. 
 To prove Theorem \ref{th.oL}, we will rely on  a Perron-Frobenius type theorem  \cite[Theorem 4.1]{guillinqsd} with the Lyapunov  function defined by, for $p>1$, 
\begin{equation}\label{eq.Ww}
\mathbf W_\star:=\mathbf W^{1/p},
\end{equation}
where $\mathbf W$ is defined in \eqref{eq.W}. For the reader's convenience, we recall \cite[Theorem 4.1]{guillinqsd}  in Section \ref{sec.thm41}. 
Moreover,  we will use the  tools developed in \cite{Wu2004}  to check the spectral gap condition \eqref{eq.SGap}.  These tools are introduced in   Section \ref{sec.bw}. In Section \ref{sec.L3}, we then show that the operator  $T_t^{oL}$ has a  spectral gap. Theorem \ref{th.oL} is finally proved in Section \ref{sec.thm1-p}.  In all this section, we recall that we consider the process $(X_t,t\ge 0)$ solution to \eqref{eq.Lsur} (see Proposition~\ref{pr.Pre}). 
Recall that  the infinitesimal generator of the diffusion \eqref{eq.Lsur} is $\mathscr L^{oL}= \mathbf b_{\mathbf c}\cdot \nabla + \frac 12 \Delta$.

%%%
 
 \subsubsection{Properties of the transition kernel $Q^{oL}_t$}
  \label{sec.rp}
 
 %%%
  \begin{prop}\label{pr.c2}
 Assume  {\rm \textbf{[c1]}} or {\rm \textbf{[c2]}}.  Then, for all   $t\ge 0$ and  all sequence $(x_n)_n$ in $\mathscr R^d$ such that $x_n\to x\in \mathscr R^d$ as $n\to +\infty$, it holds for all $\epsilon>0$, as $n\to +\infty$
  $$\mathbb P  [\sup_{s\in [0,t]}| X_s( x_n)-X_s( x)|\ge \epsilon   ]\to 0.$$ 
 \end{prop}
 
 %%%

 \begin{proof}
 Since the coefficients in \eqref{eq.Lsur} are locally Lipschitz, using in addition \eqref{eq.energy} and \eqref{eq.energy2}, the proof of the proposition is the same as the one made for Proposition 2.2 in \cite{guillinqsd3}.  
 \end{proof}
 
 %%%%%
 Let $\mathscr O$ be a subdomain of $\mathscr R^d_0$.

 %%%%%

  \begin{thm}\label{th.SF1}
 Assume  {\rm \textbf{[c1]}} or {\rm \textbf{[c2]}} and assume also    {\rm \textbf{[S1]}} or {\rm \textbf{[S2]}}.  Then,   $Q^{oL}_t$ is strongly Feller over $\mathscr O$ for all $t>0$, i.e. $Q^{oL}_t(b\mathcal B(\mathscr O))\subset \mathcal C_b(\mathscr O)$. 
   \end{thm}
 \begin{proof}
 Let $t>0$   and $f\in b\mathcal B(\mathscr O)$. The goal is to show that
  \begin{equation}\label{eq.C0-killed}
 z\in \mathscr O \mapsto Q^{oL}_tf(z)=   \mathbb E_z\Big[f(X_t)   \, e^{  - \int_0^t  \mathbf V_{\mathbf S}(X_s)ds} \mathbf 1_{t<\sigma^{oL}_{\mathscr O}}  \Big] \text{ is continuous}.
 \end{equation}
  We divide the proof of \eqref{eq.C0-killed} into two steps. The first step (\textbf{Step A}) consists in proving that the non-killed Feynman-Kac semigroup is strongly Feller (see \eqref{eq.C0} below). The second one (\textbf{Step B}) is dedicated to the proof of \eqref{eq.C0-killed}. The proof of \eqref{eq.C0-killed} relies on the following three ingredients:  \eqref{eq.C0}, the Markov property of the process, and the (uniform w.r.t. to initial conditions) continuity of the cumulative distribution function of  $\sigma^{oL}_{\mathscr O}$ at $0^+$, see \eqref{eq.contu} below.
 \medskip

 \noindent
 \textbf{Step A}. 
 The purpose of this step is to prove that the non-killed Feynman-Kac semigroup $(T^{oL}_t,t\ge 0)$  is strongly Feller\footnote{Note that, as the product of two strong Feller kernels,  \eqref{eq.C0} implies that $ T_t^{oL}$ is strong Feller in the strict sense~\cite{Revuz1976}.} over $\mathscr R_0^d$, i.e.   for all $t>0$   and $f\in b\mathcal B(\mathscr R_0^d)$,
  \begin{equation}\label{eq.C0}
 z\in \mathscr R^d_0  \mapsto  T^{oL}_tf(z):=  \mathbb E_z\Big[f(X_t)   \, e^{  - \int_0^t  \mathbf V_{\mathbf S}(X_s)ds}   \Big] \text{ is continuous}.
 \end{equation}
 To prove \eqref{eq.C0}, we first show in  \textbf{Step A.1} that the (non-killed) semigroup of $(X_t,t\ge 0)$ is strongly Feller (see \eqref{eq.C1} below). We then prove the continuity property of $x\mapsto f(X_t(x))$ in  \textbf{Step A.2} (see \eqref{eq.pSs} below). In view of \eqref{eq.C0}, it thus remains to study the continuity of   the functional $x\mapsto e^{  - \int_0^t  \mathbf V_{\mathbf S}(X_s(x))ds}$, this is the purpose of \textbf{Step A.3}.
   \medskip

 \noindent
 \textbf{Step A.1.} In this step we prove that   for all   sequence $(x_n)_n$ in $\mathscr R^d$ such that $x_n\to x\in \mathscr R^d$ as $n\to +\infty$, it holds  for all $f\in b\mathcal B(\mathscr R^d)$,
 \begin{equation}\label{eq.C1}
 \mathbb E_{x_n}[f(X_t)]\to  \mathbb E_{x}[f(X_t)],
 \end{equation}
 namely   $U^{oL}_t$  is strongly Feller over $\mathscr R^d$, where $U^{oL}_tf(x):=\mathbb E_{x}[f(X_t)]$. 
Note that if $\mathbf b_{\mathbf c}$ is smooth, \eqref{eq.C1} follows from the standard elliptic theory. Since it is not \textit{a priori} the case, we argue differently. They are probably many ways to prove  \eqref{eq.C1} in the weak setting where the drift $\mathbf b_{\mathbf c}$ is only locally Lipschitz, and  we use the \textit{energy splitting method} introduced in \cite[Section 2.1]{guillinqsd3} whose starting point is the simple equality
\begin{equation}\label{eq.es}
 \mathbb E_y[f(X_t)]= \mathbb E_y[f(X_t)\mathbf 1_{  \sigma^{oL}_{\mathscr W_R} \le t }]+\mathbb E_y[f(X_t)\mathbf 1_{t<  \sigma^{oL}_{\mathscr W_R} }], \ y \in \mathscr R^d.
 \end{equation}
 
Let $K$ be a compact subset of $\mathscr R^d$ containing the $x_n$'s and $x$ and such that $K\subset \mathscr W_R$   for all $R>R_x$, for some $R_x>0$  (where we recall that $\mathscr W_R=\{y\in \mathscr R^d, \mathbf W(y)< R\}$, see just after \eqref{eq.energy}). By \eqref{eq.energy}, 
\begin{equation}\label{eq.high-k}
\lim_{R\to +\infty} \sup_{x\in K} \mathbb E_x[f(X_t)\mathbf 1_{  \sigma^{oL}_{\mathscr W_R} \le t }]=0.
 \end{equation}
Let us now prove that for any  fixed $R>R_x$, 
\begin{equation}\label{eq.SF-k}
 x \in \mathscr W_R\mapsto \mathbb E_x[f(X_t)\mathbf 1_{t<\sigma^{oL}_{\mathscr W_R}}]  \text{ is continuous}.
 \end{equation}
Let us now consider a fixed $R>R_x$ and a globally Lipschitz and bounded vector field $\boldsymbol a:\mathscr R^d\to \mathscr R$ such that $\boldsymbol a=\mathbf b_{\mathbf c}$ on the closure of $\mathscr W_R$. Consider the solution $(Z_t,t\ge 0)$ to  $dZ_t= \boldsymbol a(Z_t)+ dB_t$. Since the   two processes $(X_t,t\ge 0)$ and $(Z_t,t\ge 0)$  coincide up to their first exit time from  $\mathscr W_R$, we then have for all $y\in \mathscr W_R$, 
\begin{equation}\label{eq.SF==}
\mathbb E_y[f(X_t)\mathbf 1_{t<\sigma^{oL}_{\mathscr W_R}}]=\mathbb E_y[f(Z_t)\mathbf 1_{t<\sigma^Z_{\mathscr W_R}}],
 \end{equation}
where $\sigma^Z_{\mathscr W_R}=\inf\{t\ge 0, Z_t\notin \mathscr W_R\}$. 
Note that since $\boldsymbol a$ is globally Lipschitz, Proposition \ref{pr.c2} also holds for the process $(Z_t,t\ge 0)$. Moreover, using again that $\boldsymbol a$ is globally Lipschitz,   the transition density $p^*_t(x,y) dy$ of the process  $(Z_t,t\ge 0)$ satisfies a  gaussian upper bound ~\cite{baldi} and hence,    the sequence  $(p^*_t(x_n,y))_n$ is  uniformly integrable  w.r.t.   a fixed probability measure. Using    \cite[Lemma 3.2]{wu1999}, we finally deduce that for all $f\in b\mathcal B(\mathscr R^d)$:
\begin{equation}\label{eq.SFY}
 f(Z_t(x_n))\to  f(Z_t(x)) \text{ in $\mathbb P$-probability}.
 \end{equation}
 In addition, using the same arguments as those used in the proof of \cite[Lemma 2.5]{guillinqsd3} (see also its note there), 
 it holds  for all  $\delta>0$ and all compact subset $ K$ of $\mathscr R^d$,  
 $$\lim_{s\to 0^+} \sup_{y\in K}\mathbb P_{y}[ \sigma^Z_{B(y,\delta)}\le s]=0,$$
 where   $\sigma^Z_{B(x,\delta)}$ is the first exit time from the open ball $B(x,\delta)$ for the process $(Z_t,t\ge 0)$. Together with \eqref{eq.SFY}, one therefore finally gets, with the same arguments as those used at the end of the proof of \cite[Theorem 2.6]{guillinqsd3}, that the mapping 
 $$y \in \mathscr R^d\mapsto \mathbb E_y[f(Z_t)\mathbf 1_{t<\sigma^Z_{\mathscr W_R}}] \text{  is continuous.}$$ Thanks to \eqref{eq.SF==},  we deduce that \eqref{eq.SF-k} holds. Finally, Equations \eqref{eq.es}, \eqref{eq.high-k},  and \eqref{eq.SF-k} imply \eqref{eq.C1}. 
  \medskip

 \noindent
 \textbf{Step A.2.}  In this step we prove that for all $f\in b\mathcal B(\mathscr R^d)$, and all  sequence $(x_n)_n\subset \mathscr R^d$ such that $\lim_nx_n=x$, it holds:
 \begin{equation}\label{eq.pSs}
Y_n:=f(X_t(x_n))\to Y:=f(X_t(x)) \text{ in $\mathbb P$-probability.}
  \end{equation}
 By the Girsanov formula \eqref{eq.Girsanov}, for every $z\in \mathscr R^d$, $X_t(z)$ admits for every $t>0$ a density $p_t(z,y)$ w.r.t. the Lebesgue measure $dy$ over $\mathscr R^d$. Let us now consider an arbitrary  fixed smooth function $\psi:\mathscr R^d\to \mathscr R_+^*$ such that $\int_{\mathscr R^d} \psi(y)dy=1$. Define the probability measure  $\rho (dy):= \psi(y)dy$ and set $q_t(z,y):= p_t(z,y)\psi^{-1}(y)$ the density of  $X_t(z)$ w.r.t. the  probability measure $\rho (dy)$.    We have thanks to \eqref{eq.C1}, for all $\phi\in L^\infty(\rho)$, 
 $$\lim_{n}\int_{\mathscr R^d}\phi(y)q_t(x_n,y)\rho (dy)=\int_{\mathscr R^d}\phi(y)q_t(x,y)\rho (dy).$$
Hence, $\lim_n q_t(x_n,\cdot )= q_t(x,\cdot )$ in $\sigma(L^1(\rho ),L^\infty(\rho ))$. This implies that for all $\mathcal A\in \mathcal B(\mathscr R^d)$, $\int_{\mathcal A}q_t(x_n,y)\rho (dy)$ has a finite limit. Consequently  the sequence  $( q_t(x_n,y))_n$ is uniformly integrable in $L^1(\rho )$ (see \cite[Theorem 4.5.6]{bogachev2007measure}). Therefore, \eqref{eq.pSs} follows from \cite[Lemma 3.2]{wu1999} and the fact that $X_t(x_n) \to X_t(x)$ in $\mathbb P$-probability (see Proposition \ref{pr.c2}).
   \medskip

 \noindent
 \textbf{Step A.3.} We now conclude the proof of \eqref{eq.C0}.  To this end, consider a sequence $(x_n)_n\subset \mathscr R^d_0$ such that $\lim_nx_n=x\in \mathscr R^d_0$, where we recall that $\mathscr R^d_0=\mathscr R^d\setminus \{0\}$. 
  Introduce the following function
 $$F: \gamma\in \mathcal C([0,t],\mathscr R^d)\mapsto \int_0^t  \mathbf V_{\mathbf S}(\gamma_s)ds,$$
 where $\mathcal C([0,t],\mathscr R^d)$ is endowed with the topology of the uniform convergence. For this topology,  the function $F$ is continuous at any $\gamma \in \mathcal C([0,t],\mathscr R^d_0)$. On the other hand, recall that by Lemma~\ref{le.finiteLsur}, $\mathbb P_x[ X_{[0,t]} \in \mathcal C([0,t],\mathscr R^d_0)]=1$.  %see indeed Proposition \ref{pr.Pre} and \eqref{eq.00}, one deduces that  $\mathbb P_x[ F \text{ discontinuous at } X]=0$.  
Hence, by the continuous mapping theorem and Proposition \ref{pr.c2}, it holds  $F(X_n)\to F(X)$ in $\mathbb P$-probability (where $X_n:=X_{[0,t]}(x_n)$ and $X:=X_{[0,t]}(x)$).  Hence as $n\to +\infty$ and in $\mathbb P$-probability, 
 $$Z_n=\exp\Big[   - \int_0^t  \mathbf V_{\mathbf S}(X_s(x_n))ds\Big]\to Z=\exp\Big[   - \int_0^t  \mathbf V_{\mathbf S}(X_s(x))ds\Big].$$
 In conclusion the bounded sequence of random variables $(Z_nY_n)_n$ converges towards $ZY$ as $n\to +\infty$ in $\mathbb P$-probability, and hence, also in $L^1$. Note here that  $b\mathcal B(\mathscr R^d_0)\subset b\mathcal B(\mathscr R^d)$ if we extend any element of $b\mathcal B(\mathscr R^d_0)$ by $0$ at the point $0$. 
 This shows \eqref{eq.C0} and concludes the proof of \eqref{eq.C0}.
 \medskip

 \noindent
 \textbf{Step B}.  We now conclude the proof of \eqref{eq.C0-killed}. Let $t>0$   and $f\in b\mathcal B(\mathscr O)$. 
 By the Markov property we have for $0<s<t$ and $z\in \mathscr R^d_0$,  
$$ Q^{oL}_tf(z)=   \mathbb E_z\Big[ \mathbf 1_{s<\sigma^{oL}_{\mathscr O}} e^{  - \int_0^{s}  \mathbf V_{\mathbf S}(X_u)du}  \mathbb E_{X_s}\big[f(X_{t-s})   \, e^{  - \int_0^{t-s}  \mathbf V_{\mathbf S}(X_u)du} \mathbf 1_{t-s<\sigma^{oL}_{\mathscr O}}\big]  \Big].$$
Set  for $0<s<t$ and $x\in \mathscr R^d_0$, 
$$\varphi_s(x)=\mathbb E_{x}\big[f(X_{t-s})   \, e^{  - \int_0^{t-s}  \mathbf V_{\mathbf S}(X_u)du} \mathbf 1_{t-s<\sigma^{oL}_{\mathscr O}}\big].$$
 The function $ \varphi_s$ is bounded over $\mathscr R^d_0$  (by $\Vert f\Vert_{b\mathcal B(\mathscr R^d_0)} e^{k_{\mathbf S}t}$) and thus, by \eqref{eq.C0},  $T^{oL}_s \varphi_s$ is continuous over $\mathscr R^d_0$.  We then have for $z\in \mathscr O$:
\begin{align*}
|T^{oL}_s\varphi_s(z)-Q^{oL}_tf(z)|&\le \Vert f\Vert_{b\mathcal B(\mathscr R^d_0)} e^{k_{\mathbf S}t}  \mathbb E_{z}\big[ \mathbf 1_{\sigma^{oL}_{\mathscr O}\le s} e^{  - \int_0^{s}  \mathbf V_{\mathbf S}(X_s)ds} \big]\\
&\le\Vert f\Vert_{b\mathcal B(\mathscr R^d_0)} e^{k_{\mathbf S}2t}\mathbb P_z[\sigma^{oL}_{\mathscr O}\le s].
\end{align*}
We now claim that for all compact subset $K$ of $\mathscr O$, setting $\delta:= {\rm dist}(K, \partial \mathscr  O)/2$, 
\begin{equation}\label{eq.contu}
\sup_{z\in K} \mathbb P_z[\sigma^{oL}_{\mathscr O}\le s]\le \sup_{z\in K} \mathbb P_z[\sigma^{oL}_{B(z,\delta)}\le s] \to 0,
\end{equation}
 as $s\to 0^+$. Note that together with \eqref{eq.contu},  the proof of \eqref{eq.C0-killed}, and thus also the proof of the theorem, are complete.  
Equation \eqref{eq.contu} can be  proved with the same arguments as those used   in the proof of \cite[Lemma 2.5]{guillinqsd3} which is based on  the strong Markov property of the process and  fact that the coefficients in \eqref{eq.Lsur} are locally Lipschitz.   We give an alternative   proof of \eqref{eq.contu}  which is  based on Itô calculus.  
 Consider  a compact subset   $K$ of $\mathscr R^d$ and   $\delta>0$. Consider also 
a smooth   function  $\theta:\mathscr R^d\to [0,1]$  such that  $\theta=0$ on $B(0,\delta/2)$ and $\theta=1$ on $\mathscr R^d\setminus B(0,\delta)$. Set  $\theta_z(y)= \theta(z-y)$.  Let $K_\delta$ be the (closed)  $\delta$-neighborhood of $K$.  We then have since $\mathbf b_c$ is locally bounded:
$$\sup_{z\in K, y\in K_\delta}|\mathscr L^{oL}\theta_z(y)|<+\infty.$$
Moreover, by Itô formula and the optional stopping theorem, 
% Le Gall Corollary 3.23 appliqué avec t\wedge  \sigma^{oL}_{B(z,\delta)} qui est un tps d'arret ps borné
% 
one gets  because $\theta_z(z)=0$, 
\begin{align*}
\mathbb E_{z}[\theta_z(X_{s\wedge\sigma^{oL}_{B(z,\delta)}})]&\le   \mathbb E_z\Big[\int_0^{s\wedge \sigma^{oL}_{B(z,\delta)}}| \mathscr L^{oL}\theta_z(X_{u} ) |du\Big]  \le s \sup_{z\in K, y\in K_\delta}|\mathscr L^{oL}\theta_z(y)|,
\end{align*}
as $X_u (z)\in K_\delta$ when $u\le   \sigma^{oL}_{B(z,\delta)}$ and $z\in K$. 
Since $\theta_z(X_{ \sigma^{oL}_{B(z,\delta)}}(z))=1$, we get that:
$$ \mathbb P_z[\sigma^{oL}_{B(z,\delta)}\le s]= \mathbb E_{z}[\mathbf 1_{\sigma^{oL}_{B(z,\delta)}\le s} \theta_z(X_{ \sigma^{oL}_{B(z,\delta)}})]\le s \sup_{z\in K, y\in K_\delta}|\mathscr L^{oL}\theta_z(y)|,$$
when $z\in K$. 
This proves \eqref{eq.contu}. 
 \end{proof}
 
 \noindent
 \textbf{Note.} 
Let us mention that there is an alternative proof of \eqref{eq.C1} based on the Girsanov formula~\eqref{eq.Girsanov}. Indeed, using e.g.~\cite[Proposition 5.8]{le2016brownian}, one can show that as $x_n\to x\in \mathscr R^d$,
\begin{equation}\label{eq.mm}
m_t^0(x_n)\to m_t^0(x) \text{ in } \mathbb P\text{-probability}.
\end{equation}
Since $\mathbb E[m_t^0(x_n)]=\mathbb E[m_t^0(x)]=1$, one then deduces (by 
Vitali convergence theorem~\cite[Theorem 6.16]{schillingBook})
 that $m_t^0(x_n)\to m_t^0(x)$ in $L^1(\mathbb P)$. 
 On the other hand, for any $f\in b\mathcal B(\mathscr R^d)$, $f(x_n+B_t)\to f(x+B_t)$ in $\mathbb P$-probability (one can use for instance \cite[Lemma 3.2]{wu1999} to prove it). Using~\eqref{eq.Girsanov} then proves \eqref{eq.C1}.

   \begin{prop}\label{pr.TI}
Assume  {\rm \textbf{[c1]}} or {\rm \textbf{[c2]}} and assume also    {\rm \textbf{[S1]}} or {\rm \textbf{[S2]}}.    Then, for all $t>0$, $Q^{oL}_t$ is topologically irreducible on $\mathscr O$, i.e. for all $x\in \mathscr O$ and all non-empty open subset $O$ of $\mathscr O$,  $Q^{oL}_t(x,O)>0$.
 \end{prop}

 %%%

  \begin{proof}
Since $\mathbf b_{\mathbf c}$ is locally Lipschitz and $\mathscr O$ is a subdomain of $\mathscr R^d$, with the same arguments as those used to prove~\cite[Proposition 2.10]{guillinqsd3}, which are based on the knowledge of the support of the law of the trajectories of  the Brownian motion, one shows that  for all $t>0$, all $x\in \mathscr O$ and all non-empty open subset $  O$ of $\mathscr O$,  $\mathbb E_x[\mathbf 1_{  O}(X_t) \mathbf 1_{t<\sigma^{oL}_{\mathscr O}}]>0$. Note that this property can also be obtained with the Girsanov formula \eqref{eq.Girsanov}. The desired result follows from the fact that for all   $x\in \mathscr O$ and $t>0$, $ \exp [   - \int_0^t  \mathbf V_{\mathbf S}(X_s(x))ds]>0$ a.s. by Lemma~\ref{le.finiteLsur}. 
 \end{proof}

 %%%

% \begin{proof}
% Conditions \eqref{eq.St1} and \eqref{eq.St2} implies (see indeed the proof of the second step in the proof of~\cite[Theorem 3]{ferre2}) that there exist two positive sequences $(\gamma_n)$ and $(b_n)$ such that  $\gamma_n\to 0^+$ and an increasing sequence of compact sets $(K_n)$ of $\mathscr R^d_0$ such that on $\mathscr R^d_0$:
% $$Q^{oL}_t\mathbf W\le \gamma_n \mathbf W+ b_n \mathbf 1_{K_n}.$$
%This second statement is a consequence of the fact that $Q^{oL}_t$ is strongly Feller for $t>0$ together with the analysis led in the first  step in the proof of~\cite[Lemma 2]{ferre2}. 
% \end{proof}
 %%%%%

 \begin{prop}\label{th.infty}
Assume    {\rm \textbf{[c1]}}   and assume also    {\rm \textbf{[S1]}} or {\rm \textbf{[S2]}}.   Let $\mathbf W$ be defined by \eqref{eq.W}.  Then, when $|x|\to  +\infty$,  $ \mathscr L^{oL}   \mathbf W(x) /{\mathbf W(x) } \to -\infty$. 
 In particular, for any $c>0$, when 
  $x\to \partial \mathscr R^d_0\cup \{\infty\}$ and $x\in \mathscr R^d_0$  (where $\partial \mathscr R^d_0=\{0\}$), 
\begin{equation}\label{eq.Lyy}
\frac{\mathscr L^{oL}   \mathbf W(x) }{\mathbf W(x) } - c \mathbf V_{\mathbf S}(x)\to -\infty.
\end{equation}
Assume that  {\rm \textbf{[S2]}} holds and that either   {\rm \textbf{[c1]}} or {\rm \textbf{[c2]}} holds. Then, \eqref{eq.Lyy} is satisfied  with the Lyapunov function $\mathbf 1$. 
 \end{prop}
%The previous  proposition will imply  that, when {\rm \textbf{[c1]}} (resp. {\rm \textbf{[c2]}}) holds,  the operator $Q^{oL}_t|_{b\mathcal B_{\mathbf W}(\mathscr R^d_0)}$  (resp. $Q_t|_{b\mathcal B(\mathscr R^d_0)}$) has a spectral gap.
 \begin{proof}
 See the proof of Proposition \ref{pr.Pre}.  
  \end{proof}
  Note that when there exists  $k>2$ and $c_1,c_2>0$ 
 such that $ - \mathbf b_{\mathbf c}(x) \cdot x \ge c_2 |x|^{k}-c_1$ for all $x\in \mathscr R^d$, one can choose the smaller Lyapunov function than \eqref{eq.W} defined by: 
 \begin{equation}\label{eq.alternativeW}
 \mathbf W= \boldsymbol \ell+1, \ \boldsymbol \ell = |\cdot |^k \times (1-\chi).
  \end{equation}
   We then indeed have  for $|x|$ sufficiently large  and some $c>0$ that $\mathscr L^{oL} \mathbf W(x)/\mathbf W(x)\le - c|x|^{2(k-1)-k}\to -\infty$ as $|x|\to \infty$.

  \subsubsection{A Perron-Frobenius type theorem from \cite{guillinqsd}}
 \label{sec.thm41}
 In this section, we recall~\cite[Theorem~4.1]{guillinqsd}. To this end, let us consider    a positive bounded kernel $Q=Q(x,dy)$ on the Polish space $\mathscr S$. The typical   state spaces $\mathscr S$ we consider in this work are open   subsets $\mathscr M$  of $\mathscr R^k$.  Let us also consider a continuous   function  $\mathbf W_\star: \mathscr S\to [1,+\infty)$.
We assume that $Q$ satisfies the following three regularity assumptions: 
 \begin{enumerate}
 \item[$\mathfrak 1$.] There exists $N_1\ge 1$ such that $Q^k$ is Feller for all $k\ge N_1$, i.e. $Q^kf\in \mathcal C_b(\mathscr S)$ for all $f\in \mathcal C_b(\mathscr  S)$. 
 \item[$\mathfrak 2$.] There exists $N_2\ge 1$, $Q^{N_2}$ is topologically irreducible over $\mathscr S$, i.e.  for any $x\in \mathscr S$ and  for every nonempty open subset $O$ of $\mathscr S$, $
 Q^{N_2} (x,O)>0$. 
 \item[$\mathfrak 3$.] For some $p>1$ and some constant $C>0$,  
 $Q\mathbf W_\star^p \le C \mathbf W_\star^p$.  
 \end{enumerate}
  \begin{thm}\label{thm41}(\cite[Theorem 4.1]{guillinqsd}) Assume, in addition to the three previous assumptions $\mathfrak 1$, $\mathfrak 2$, and $\mathfrak 3$ above, that $Q$ has a spectral gap in $b\BB_{\mathbf W_\star} (\mathscr S)$, i.e. 
\begin{equation}\label{eq.SGap}
 \mathsf r_{ess}(Q|_{b\BB_{\mathbf W_\star} (\mathscr S)})<\mathsf r_{sp}(Q|_{b\BB_{\mathbf W_\star} (\mathscr S)}).
\end{equation}
Then, there exist  a unique couple $(\mu, \varphi) \in \mathcal P_{\mathbf W_\star}(\mathscr S)\times \mathcal C_{b\mathbf W_\star}(\mathscr S)$, where  $\varphi>0$ over $\mathscr S$ and  $\mu(\varphi)=1$, and constants $r\in ]0,1[, C\ge 1$, such that
$$
 \mu Q=\mathsf r_{sp}(Q|_{b\BB_{\mathbf W_\star}}(\mathscr S))\mu, \ Q\varphi= \mathsf r_{sp}(Q|_{b\BB_{\mathbf W_\star} (\mathscr S)}) \varphi
$$ 
 and
\bbeq\label{thm52c}
\Big \|\frac {1}{\mathsf r_{sp}(Q|_{b\BB_{\mathbf W_\star}(\mathscr S)})^n}Q^n f - \varphi \mu(f)\Big \|_{b\BB_{\mathbf W_\star}(\mathscr S)} \le C r^n \|f \|_{b\BB_{\mathbf W_\star} (\mathscr S)}, \ \forall f\in b\BB_{\mathbf W_\star}(\mathscr S).
\nneq
Moreover, $\mu(O)>0$ for all nonempty open subset $O$ of $\mathscr S$. 
In particular, we have: 
 \benu[(a)]
\item[$\mathfrak a$.] If  $\nu\in \MM_{b\mathbf W_\star}(\mathscr S) $  satisfies for some $\lambda \in \mathscr R$, $\nu Q= \lambda\nu$ and $\nu(\varphi)\ne0$, then $\lambda=\mathsf r_{sp}(Q|_{b\BB_{\mathbf W_\star} (\mathscr S)})$ and $\nu=c\mu$ for some constant $c$.
 \item[$\mathfrak b$.] If $f\in b\BB_{\mathbf W_\star}(\mathscr S)$ satisfies for some  $\lambda \in \mathscr R$,  $Qf=\lambda f$  and  $\mu(f)\ne 0$, then $\lambda=\mathsf r_{sp}(Q|_{b\BB_{\mathbf W_\star}(\mathscr S)})$ and $f=c\varphi$ for some constant $c$.
 \nenu 
   \end{thm}

  \subsubsection{Measures of non  compactness}
 \label{sec.bw}
 For a bounded non negative transition kernel $Q$ on a Polish space $\mathscr S$, set
\begin{equation}\label{eq.betaw}
\beta_w(Q  ):=\inf_{K\subset \mathscr S}\, \sup_{ x\in \mathscr S}\,  Q( x, \mathscr S\setminus K)\, \text{ and } \, \beta_\tau(Q ):=\sup_{(A_n)_n}\, \lim_{n\to +\infty} \,  \sup_{ x\in \mathscr S}\, Q( x, A_n) ,
\end{equation}
where the infimum is taken over all compact subset  $K$ of $\mathscr S$ and supremum is taken over all the sequences  $(A_n)_n$ of elements of  $  \mathcal B(\mathscr S)$ decreasing towards $\emptyset$. The kernel $Q$ satisfies  \cite[\textbf{(A1)}]{Wu2004} if by definition, for any compact subset $K$ of $\mathscr S$, 
$$
\text{\textbf{(A1)} }\ \ \beta_w(\mathbf 1_K Q)=0 \text{ and  } \beta_\tau(\mathbf 1_KQ^N)=0,
$$
for some  $N\ge 1$  independent of $K$. 
Under this previous assumption, the last author proves in \cite[Theorem 3.5]{Wu2004} the following Gelfand-Nussbaum type formula  \eqref{eq.rr} 
 \begin{equation}\label{eq.Gelfand}
 \mathsf r_{ess}(Q|_{b\mathcal B(\mathscr S)})=\lim_{n\to +\infty}[ \beta_w((Q|_{b\mathcal B(\mathscr S)})^n)]^{1/n}.
 \end{equation}
 for the essential spectral radius of the transition kernel $Q$.

%%%

\subsubsection{On the essential spectral radius of the non-killed Feynman-Kac semigroup}
 \label{sec.L3}

Recall \eqref{eq.C0} for the definition of the  non-killed Feynman-Kac semigroup  $(T^{oL}_t,t\ge 0)$. 
 
  \begin{thm}\label{th.Ress}
 Assume  {\rm \textbf{[c1]}} or {\rm \textbf{[c2]}} and assume also    {\rm \textbf{[S1]}} or {\rm \textbf{[S2]}}.   
Moreover, assume  that there exist   constants ${\mathfrak m_0},{\mathfrak r_0},{\mathfrak b_0}>0$,  a compact subset $\mathscr K_0$ of $\mathscr R^d_0=\mathscr R^d\setminus \{0\}$, and   a $\mathcal C^2$  function $\mathbf W_\star:\mathscr R^d  \to [1,+\infty)$  such  that 
%$x\to \partial \mathscr R^d_0\cup \{+\infty\}$ ($x\in \mathscr R^d_0$), 
\begin{equation}\label{eq.cc1}
\frac{\mathscr L^{oL}  \mathbf W_\star }{\mathbf W_\star }-   \mathbf V_{\mathbf S}\le -{\mathfrak r_0} \mathbf 1_{\mathscr R^d_0\setminus \mathscr K_0} + {\mathfrak b_0}\mathbf 1_{\mathscr K_0} \text{ on } \mathscr R^d_0, 
\end{equation}
and for some $p>1$, 
\begin{equation}\label{eq.cc2}
( \mathscr L^{oL} - \mathbf V_{\mathbf S})\mathbf W_\star^p\le {\mathfrak m_0} \mathbf W_\star^p \text{ on } \mathscr R^d_0.
\end{equation}
Then, for all $t\ge 0$, $T^{oL}_t$ is a bounded operator on $b\mathcal B_{\mathbf W_\star}(\mathscr R^d_0)$. In addition (see \eqref{eq.betaw} and \eqref{eq.Ttt} for the definitions of $\beta_w$ and $T^{oL}_{t,\mathbf W_\star}$), for $t>0$, 
$$\beta_w(T^{oL}_{t,\mathbf W_\star}|_{b\mathcal B (\mathscr R^d_0)})\le e^{-{\mathfrak r_0}t}\,  \text{ and  } \, \mathsf r_{ess}(T^{oL}_t|_{b\mathcal B_{\mathbf W_\star}(\mathscr R^d_0)})\le e^{-{\mathfrak r_0}t}.$$ 
  \end{thm}
  %%%%
  
%  \begin{remarks}
% In this work, the Lyapunov  function $\mathbf W_\star$ will actually be defined over the whole space $\mathscr R^d$. 
%   \end{remarks}

 %%%

 \begin{proof}
 In all this proof,  $k_{\mathbf S}\ge 0$ is  such that $\mathbf V_{\mathbf S}\ge - k_{\mathbf S}$.  
   The proof of Theorem \ref{th.Ress} is divided into three steps. 
  We first prove that $T^{oL}_t:b\mathcal B_{\mathbf W_\star}(\mathscr R^d_0) \to b\mathcal B_{\mathbf W_\star}(\mathscr R^d_0)$ is bounded (\textbf{Step 1}). Then, in order to use the formula \eqref{eq.Gelfand} with this operator, we check that it satisfies assumption  \textbf{(A1)}, this is the purpose of the second step (see \textbf{Step 2} below).  In the last step (\textbf{Step 3}), we finally give an upper bound on $\beta_w(T^{oL}_{t,\mathbf W_\star})$ making use in particular of \eqref{eq.cc1}.    
\medskip

\noindent
\textbf{Step 1.} In this step we prove that  $T^{oL}_t$ is a bounded operator on $b\mathcal B_{\mathbf W_\star}(\mathscr R^d_0)$. The proof is  rather standard, we make it for sake of completeness.    Differentiating the function $t\in \mathscr R_+\mapsto e^{-\mathfrak m_0t}T^{oL}_t\mathbf W_\star^p(x)$ for $x\in \mathscr R^d_0$ and using \eqref{eq.cc2} yields:
 $$\frac{d}{dt}\big(e^{-\mathfrak m_0t}T^{oL}_t\mathbf W_\star^p\big )=e^{-\mathfrak m_0t} T^{oL}_t\big (-\mathfrak m_0\mathbf W_\star^p +  (\mathscr L^{oL} -   \mathbf V_{\mathbf S})\mathbf W_\star^p\big ) \le 0.$$
Therefore, for all $t\ge 0$, 
\begin{equation} \label{eq.=p-}
T^{oL}_t\mathbf W_\star^p\le e^{\mathfrak m_0t}\mathbf W_\star^p \text{ on } \mathscr R^d_0.
\end{equation} 
In addition, by Hölder’s inequality, we have on $\mathscr R^d_0$, 
$$T^{oL}_t\mathbf W_\star\le [T^{oL}_t\mathbf 1]^{1/q}[T^{oL}_t\mathbf W_\star^p]^{1/p},$$ where $q=p/(p-1)$. On the other hand,     it holds for all $x\in \mathscr R^d_0$, $T^{oL}_t\mathbf 1(x)= \mathbb E_x [  \, e^{  - \int_0^t  \mathbf V_{\mathbf S}(X_s)ds}   ]\le e^{k_{\mathbf S}t}$. Consequently, on $\mathscr R^d_0$, 
\begin{equation}\label{eq.Q1}
T^{oL}_t\mathbf W_\star\le e^{(k_{\mathbf S}/q+\mathfrak m_0/p)t}\mathbf W_\star.
\end{equation}
 Hence, $T^{oL}_t:b\mathcal B_{\mathbf W_\star}(\mathscr R^d_0)\to b\mathcal B_{\mathbf W_\star}(\mathscr R^d_0)$ is well defined and is  a bounded operator. 
We can then define the operator $T^{oL}_{t,\mathbf W_\star}$ on $b\mathcal B(\mathscr R^d_0)$ by
\begin{equation}\label{eq.Ttt}
T^{oL}_{t,\mathbf W_\star}(x,dy)= \frac{\mathbf W_\star(y)}{\mathbf W_\star(x)} T_{t}^{oL}(x,dy),
\end{equation}
 which is similar to $T^{oL}_t|_{b\mathcal B_{\mathbf W_\star}(\mathscr R^d_0)}$. In particular,  $T^{oL}_{t,\mathbf W_\star}|_{b\mathcal B(\mathscr R^d_0)}$ and $T^{oL}_t|_{b\mathcal B_{\mathbf W_\star}(\mathscr R^d_0)}$  have the same essential spectral radius. 
 \medskip

\noindent
\textbf{Step 2.}
 We now prove that $T^{oL}_{t,\mathbf W_\star}|_{b\mathcal B(\mathscr R^d_0)}$  satisfies  the assumption \textbf{(A1)}. Since $T^{oL}_t$ is strongly Feller for all $t>0$, see Theorem \ref{th.SF1},  $T^{oL}_t$ satisfies   \textbf{(A1)} with $N=1$ (see \cite[Remark 3.3]{guillinqsd}). Let $K_1$ be a compact subset of $\mathscr R^d_0$. Applying  Hölder’s inequality, we have that since $\beta_w(\mathbf 1_{K_1}T^{oL}_{t})=0$, 
\begin{align*}
\beta_w(\mathbf 1_{K_1} T^{oL}_{t,\mathbf W_\star})&= \inf_{K\subset \mathscr R^d_0}\  \sup_{x\in  K_1} \, T^{oL}_{t,\mathbf W_\star}(x,\mathscr R^d_0\setminus K)\\
&= \inf_{K\subset \mathscr R^d_0}\  \sup_{x\in  K_1} \, \frac{1}{\mathbf W_\star(x)} \int_{\mathscr R^d_0\setminus K} \mathbf W_\star(y) T^{oL}_t(x,dy) \le e^{{\mathfrak m_0}t/p}[\beta_w(\mathbf 1_{K_1}T^{oL}_{t})]^{1/q}=0,
\end{align*} 
 by the Feller property of $T^{oL}_{t}$. 
With the same arguments, for every $t>0$,
$$\beta_\tau (\mathbf 1_{K_1} T^{oL}_{t,\mathbf W_\star})= \sup_{(A_n)}\lim_{n\to +\infty} \sup_{x\in  K_1} T^{oL}_{t,\mathbf W_\star}(x,A_n) \le e^{{\mathfrak m_0}t/p}[\beta_\tau(\mathbf 1_{K_1} T^{oL}_{t})]^{1/q}=0,$$
 by the strong Feller property of $T^{oL}_{t}$ (see Theorem \ref{th.SF1}). 
 Hence, for every $t>0$,  $T^{oL}_{t,\mathbf W_\star}|_{b\mathcal B(\mathscr R^d_0)}$  satisfies  \textbf{(A1)}\footnote{This can also be obtained noticing that $T^{oL}_{t,\mathbf W_\star}$ is   strongly Feller for $t>0$.}.   This allows us to use \eqref{eq.Gelfand} to get that:
\begin{equation}\label{eq.rr}
\mathsf r_{ess}(T^{oL}_{t,\mathbf W_\star}|_{b\mathcal B(\mathscr R^d_0)})=\lim_{n\to +\infty}[ \beta_w\big ([T^{oL}_{t,\mathbf W_\star}]^n\big )]^{1/n}.
\end{equation}
 Let us now prove that 
\begin{equation}\label{eq.Kw}
\beta_w(T^{oL}_{t,\mathbf W_\star})\le e^{-{\mathfrak r_0}t} .
\end{equation}
Note that \eqref{eq.Kw} implies that    $\mathsf r_{ess}(T^{oL}_t|_{b\mathcal B_{\mathbf W_\star}(\mathscr R^d_0)})\le e^{-{\mathfrak r_0} t}$. Indeed, by~\cite[Proposition 3.2.(e)]{Wu2004}, we have $\beta_w([T^{oL}_{t,\mathbf W_\star}]^n)\le \beta_w(T^{oL}_{t,\mathbf W_\star})^n$. 
 \medskip

\noindent
\textbf{Step 3.}
 Let us now prove \eqref{eq.Kw}. 
 We have by definition that 
\begin{equation}\label{eq.p1}
\beta_w(T^{oL}_{t,\mathbf W_\star})= \inf_{K\subset \mathscr R^d_0}\  \sup_{x\in  \mathscr R^d_0}\ \  \frac{1}{\mathbf W_\star(x)} \mathbb E_x\Big[ \mathbf W_\star(X_t)\mathbf 1_{\mathscr R^d_0\setminus K} (X_t) \, e^{  - \int_0^t  \mathbf V_{\mathbf S}(X_s)ds}   \Big].
\end{equation}
 Denote by $\tau_{\mathscr K_0}$ the first hitting time of the compact $\mathscr K_0$ for the process (where $\mathscr K_0$ is the compact set such that  \eqref{eq.cc1} holds), that is to say $\tau_{\mathscr K_0}:=\inf\{s\ge 0, X_s\in \mathscr K_0\}$. We have for all $x\in \mathscr R^d_0$ and all compact subset $K$ of $\mathscr R^d_0$:
 \begin{align}
 \nonumber
&  \frac{1}{\mathbf W_\star(x)}  \mathbb E_x\Big[ \mathbf W_\star(X_t)\mathbf 1_{\mathscr R^d_0\setminus K} (X_t) \, e^{  - \int_0^t  \mathbf V_{\mathbf S}(X_s)ds}   \Big]\\
 \nonumber
&=   \frac{1}{\mathbf W_\star(x)}  \mathbb E_x\Big[\mathbf W_\star(X_t)\mathbf 1_{\tau_{\mathscr K_0}\le t} \mathbf 1_{\mathscr R^d_0\setminus K} (X_t) \, e^{  - \int_0^t  \mathbf V_{\mathbf S}(X_s)ds}   \Big]\\
\label{eq.=1}
 &\quad + \frac{1}{\mathbf W_\star(x)}  \mathbb E_x\Big[\mathbf W_\star(X_t)\mathbf 1_{t<\tau_{\mathscr K_0} } \mathbf 1_{\mathscr R^d_0\setminus K} (X_t) \, e^{  - \int_0^t  \mathbf V_{\mathbf S}(X_s)ds}   \Big].
 \end{align}
 We now treat separately each term in the r.h.s. of \eqref{eq.=1}. 
  \medskip

\noindent
\textbf{Step 3a.}
In this step, we deal with the first term in the r.h.s. of \eqref{eq.=1}. Let  $K$ be a compact subset  of $\mathscr R^d_0$.  
%Note that 
%for all $t>0$ and $x\in \mathscr R^d$, $\mathbf 1_{ \mathscr R^d_0 \setminus K} (X_t(x))=\mathbf 1_{ \mathscr R^d\setminus K} (X_t(x))$ a.s. since $\mathbb P_x[X_t=0]=0$ due to the fact that $X_t(x)$ has a density w.r.t. the Lebesgue measure $dy$. 
Using    Hölder's inequality and the strong Markov property  of the process $(X_t,t\ge 0)$, for all $x\in \mathscr R^d_0$:
 \begin{align*}
& \frac{1}{\mathbf W_\star(x)}  \mathbb E_x\Big[\mathbf W_\star(X_t)\mathbf 1_{\tau_{\mathscr K_0}\le t} \mathbf 1_{\mathscr R^d_0\setminus K} (X_t) \, e^{  - \int_0^t  \mathbf V_{\mathbf S}(X_s)ds}   \Big] \\
 &\le \frac{\mathbb E_x\big [\mathbf W_\star^p(X_t)   e^{  - \int_0^t  \mathbf V_{\mathbf S}(X_s)ds}   \big ]^{1/p}  }{\mathbf W_\star(x)} \  \mathbb E_x\big [ \mathbf 1_{\tau_{\mathscr K_0}\le t} \mathbf 1_{ \mathscr R^d_0\setminus K} (X_t)  e^{  - \int_0^t  \mathbf V_{\mathbf S}(X_s)ds}  \big  ] ^{1/q}\\
 &\le e^{(k_{\mathbf S}/q+\mathfrak m_0/p)t}\,  \mathbb E_x\big [ \mathbf 1_{\tau_{\mathscr K_0}\le t} \mathbf 1_{ \mathscr R^d_0\setminus K} (X_t)  \big  ] ^{1/q} \\
 &=e^{(k_{\mathbf S}/q+\mathfrak m_0/p)t}\, \mathbb E_x\big [\mathbf 1_{\tau_{\mathscr K_0}\le t}\,    \mathbb P_{X_{\tau_{\mathscr K_0}}} [X_{t-\tau_{\mathscr K_0}}\in  \mathscr R^d_0\setminus K] \big ]^{1/q}.
   \end{align*}
 In addition, we have for all $x\in \mathscr R^d_0$,
  \begin{align*}
 \mathbb E_x\big [\mathbf 1_{\tau_{\mathscr K_0}\le t}\,    \mathbb P_{X_{\tau_{\mathscr K_0}}} [X_{t-\tau_{\mathscr K_0}}\in  \mathscr R^d_0\setminus K] \big ] &\le \mathbb E_x\big [\mathbf 1_{\tau_{\mathscr K_0}\le t}\,    \mathbb P_{X_{\tau_{\mathscr K_0}}} [\exists s\in [0,t], X_s \in  \mathscr R^d_0\setminus K] \big ] 
 \\
 &\le    \sup_{z\in \mathscr K_0}\mathbb P_{z}  \big [\exists s\in [0,t], X_s \in  \mathscr R^d_0\setminus K \big ] .
  \end{align*}
 %Using Proposition \ref{pr.c2} and the same arguments as those used in the second step of the proof of~\cite[Theorem 3.5]{guillinqsd}, 
 Let us now prove that 
\begin{equation}\label{eq.k=0}
\inf_{K\subset \mathscr R^d_0}\, \sup_{z\in \mathscr K_0}\mathbb P_{z} \big [\exists s\in [0,t], X_s \in  \mathscr R^d_0\setminus K \big ] =0.
\end{equation}
First of all, note that by \eqref{eq.Cc} and  Proposition~\ref{pr.c2}, the mapping
$$z\in \mathscr R^d_0\mapsto \mathbb P_z[X_{[0,t]}\in \cdot ]\in \mathcal P(\mathcal C([0,t],\mathscr R^d_0))$$
is continuous for the weak topology. 
Note indeed that \eqref{eq.Cc} and  Proposition~\ref{pr.c2} imply that 
$$\mathbb P  [\sup_{s\in [0,t]}\mathsf d_{\mathscr R^d_0}( X_s( x_n),X_s( x))\ge \epsilon   ]\to 0,$$
% use teh continuous mapping thm pour la deuxième
as $x_n\to x\in \mathscr R^d_0$ and for all $\epsilon >0$ (see \eqref{eq.DD}). 
Hence, the family of probability measures   $\{z\in \mathscr K_0, \mathbb P_z[X_{[0,t]}\in \cdot ]\}$ over $\mathcal C([0,t],\mathscr R^d_0)$  is relatively compact for the weak convergence topology and thus, is tight. 
%We can thus  conclude this step using the same arguments as those used at the end of the second step of the proof of~\cite[Theorem 3.5]{guillinqsd}. 
Consequently, for all $\epsilon>0$, there exists a compact set $A_\epsilon$ of $ \mathcal C([0,t],\mathscr R^d_0)$  such that $\sup_{z\in \mathscr K_0} \mathbb P_z[X_{[0,t]}\notin A_\epsilon]<\epsilon$. Then, introduce $K_\epsilon=\{\gamma_s, s\in [0,t], \gamma\in A_\epsilon\}$ which is a compact subset of $\mathscr R^d_0$. Therefore, we have 
\begin{align*}
\inf_{K\subset \mathscr R^d_0}\, \sup_{z\in \mathscr K_0}\mathbb P_{z} [\exists s\in [0,t], X_s \in  \mathscr R^d_0\setminus K]&\le   \sup_{z\in \mathscr K_0}\mathbb P_{z} [\exists s\in [0,t], X_s \in  \mathscr R^d_0\setminus  K_\epsilon]\\
&\le \sup_{z\in \mathscr K_0} \mathbb P_z[X_{[0,t]}\notin A_\epsilon]<\epsilon.
\end{align*} This ends the proof of~\eqref{eq.k=0} and shows that  the infimum over all compact set $K\subset \mathscr R^d$ of the  first term in the r.h.s. of \eqref{eq.=1} vanishes. 
\medskip

\noindent
\textbf{Step 3b.}
In this step, we deal with the second term in the r.h.s. of \eqref{eq.=1}.
We have that 
 \begin{align*}
    \mathbb E_x\Big[\mathbf W_\star(X_t)\mathbf 1_{t<\tau_{\mathscr K_0} } \mathbf 1_{\mathscr R^d_0\setminus K} (X_t) \, e^{  - \int_0^t  \mathbf V_{\mathbf S}(X_s)ds}   \Big] &\le   \mathbb E_x\Big[\mathbf W_\star(X_t)\mathbf 1_{t<\tau_{\mathscr K_0} } \, e^{  - \int_0^t  \mathbf V_{\mathbf S}(X_s)ds}   \Big].
 \end{align*}
 Let $(\mathsf m_t, t\ge 0)$ be defined by:
 $$\mathsf m_t:= \frac{\mathbf W_\star(X_t)}{\mathbf W_\star(X_0)} \exp \big (-\int_0^t \frac{\mathscr L^{oL}\mathbf W_\star}{\mathbf W_\star} (X_s) ds\big ).$$
This is a local martingale $(\mathsf m_t, t\ge 0)$, which turns out to be  actually a supermartingale. Hence,   it holds: 
 $$\text{ for all $x\in \mathscr R^d$, }\,    \mathbb E_x [\mathsf m_t ]\le 1.
 $$
 Thus,    using  \eqref{eq.cc1}, we have for $x\in \mathscr R^d_0$, 
 \begin{align*}
 &\frac{1}{\mathbf W_\star(x)}  \mathbb E_x\Big[\mathbf W_\star(X_t)\mathbf 1_{t<\tau_{\mathscr K_0} }   \, e^{ -  \int_0^t   \mathbf V_{\mathbf S}(X_s) ds}   \Big] \le e^{-{\mathfrak r_0}t}   \mathbb E_x [\mathbf 1_{t<\tau_{\mathscr K_0} }\mathsf m_t]\le  e^{-{\mathfrak r_0}t}.
 \end{align*}
Consequently, the second term in the r.h.s. of \eqref{eq.=1} is upper bounded by $e^{-{\mathfrak r_0}t}$. 
 In conclusion, coming back to \eqref{eq.p1} and \eqref{eq.=1}, we have proved \eqref{eq.Kw}.  
This achieves the proof of Theorem~\ref{th.Ress}. 
 \end{proof}

%%%
  \subsubsection{Proof of Theorem \ref{th.oL}}
  \label{sec.thm1-p}
  We will only prove Theorem \ref{th.oL}  when   {\rm \textbf{[c1]}}   holds and  either    {\rm \textbf{[S1]}} or {\rm \textbf{[S2]}}  holds (the other case is proved  with the same arguments but with   the Lyapunov function $\mathbf 1$). 
Let $\mathscr O$ be a subdomain of $\mathscr R^d_0$. 
 As already explained, the proof of Theorem \ref{th.oL} consists in  using \cite[Theorem 4.1]{guillinqsd} (namely  Theorem~\ref{thm41} above). It is divided into two steps. In the first step, we check that $Q^{oL}_1$ satisfies all the required conditions to use   Theorem \ref{thm41}. We then apply Theorem \ref{thm41} to deduce all the assertions of Theorem~\ref{th.oL} except three points which remain to be proved (see Items \textbf{i}, \textbf{ii}, and \textbf{iii} in \textbf{Step 2}). \textbf{Step 2} is dedicated to the proofs of these  three points.  
\medskip

\noindent
 \textbf{Step 1.} 
 We start by checking that $Q^{oL}_1$ satisfies all the fourth  conditions of   Theorem \ref{thm41} with the Lyapunov function $\mathbf W_\star$ (see \eqref{eq.Ww}) on the state space $\mathscr O$.

  First of all, the transition kernel    $Q^{oL}_1$ is Feller over $\mathscr O$  and $Q^{oL}_1(x, O)>0$ for all $x\in \mathscr O$ and all non-empty open subset $  O$ of $\mathscr O$ (see Theorem \ref{th.SF1} and Proposition \ref{pr.TI}). 
We now prove that $Q^{oL}_1$ has a spectral gap on $b\mathcal B_{\mathbf W_*}(\mathscr O)$.

  \medskip

\noindent
 \textbf{Step 1.a.} In this step we check that the non-killed Feynman-Kac operator $T_1^{oL}$  is a bounded operator on $b\mathcal B_{\mathbf W_\star}(\mathscr R^d_0)$ and that  $T^{oL}_{1,\mathbf W_\star}$ (see \eqref{eq.C0} and \eqref{eq.Ttt}), which is then well defined,  satisfies  (see  \eqref{eq.betaw}) 
\begin{equation}\label{eq.Tb}
\beta_w(T^{oL}_{1,\mathbf W_\star}|_{b\mathcal B (\mathscr R^d_0)})=0.
\end{equation}
 To this end, we apply Theorem \ref{th.Ress}. 
By \eqref{eq.Lyy}, 
 there exist two sequences of positive constants $(\mathfrak r_n)$ and $(\mathfrak b_n)$ where $\mathfrak r_n\to +\infty$, and an increasing sequence of compact subsets $(\mathscr K_n)$ of $\mathscr R^d_0$, such that for all $x\in \mathscr R^d_0$, 
$$
  \frac{\mathscr L^{oL} \mathbf W (x) }{ \mathbf W (x) }- p \mathbf V_{\mathbf S}(x)\le - \mathfrak r_n + \mathfrak b_n \mathbf 1_{\mathscr K_n}(x).
$$
Hence,  over $\mathscr R^d_0$, it holds (see \eqref{eq.Ww}): 
\begin{align*}
 (\mathscr L^{oL}-\mathbf V_{\mathbf S})  {\mathbf W_\star}  \le \frac 1p {\mathbf W}^{1/p-1}  (\mathscr L^{oL}-p\mathbf V_{\mathbf S})  {\mathbf W} &\le \big  [-\frac{\mathfrak r_n}{p}   +\frac{\mathfrak b_n}{p} \mathsf 1_{\mathscr K_n} \big ]{\mathbf W}^{1/p}\\
 &= \big [-\frac{\mathfrak r_n}{p}   +\frac{\mathfrak b_n}{p} \mathsf 1_{\mathscr K_n} \big ]{\mathbf W_\star}.
\end{align*}
Note that the first inequality above can be obtained by a direct computations, but it is actually a consequence of a more general result~\cite[Proposition 5.1]{guillinqsd}.
On the other hand, by \eqref{eq.Lyy}, there exists $\mathfrak m_0>0$ such that  on $\mathscr R^d_0$,  
$$(\mathscr L^{oL}-\mathbf V_{\mathbf S})  {\mathbf W}^p_\star=(\mathscr L^{oL}-\mathbf V_{\mathbf S})  \mathbf W\le \mathfrak m_0    \mathbf W=\mathfrak m_0    \mathbf W^p_\star,$$
which is precisely \eqref{eq.cc2}. 
Moreover, applying  Theorem \ref{th.Ress}, we then have that $T^{oL}_1$  is a bounded operator on $b\mathcal B_{\mathbf W_\star}(\mathscr R^d_0)$  (see \eqref{eq.Q1}), and in addition, we also have that \eqref{eq.Tb} holds and $\mathsf r_{ess}(T^{oL}_1|_{b\mathcal B_{\mathbf W_\star}(\mathscr R^d_0)})=0$.  Note that \eqref{eq.Tb} actually holds not only for $t=1$ but also  for all $t>0$.  
  \medskip

\noindent
 \textbf{Step 1.b.} In this step we prove the following spectral gap: 
\begin{equation}\label{eq.spga}
0=\mathsf r_{ess}(Q^{oL}_1|_{b\mathcal B_{\mathbf W_\star}(\mathscr O)})<\mathsf r_{sp}(Q^{oL}_1|_{b\mathcal B_{\mathbf W_\star}(\mathscr O)}).
\end{equation}
 First of all, recall that \eqref{eq.cc2}  (which is proved just above in \textbf{Step 1.a}) implies  \eqref{eq.=p-} and  \eqref{eq.Q1}. Then, because $Q^{oL}_1\le T^{oL}_1$,  Assumption $\mathfrak 3$ in Section \ref{sec.thm41} is satisfied  for $Q^{oL}_1$ and moreover, $Q^{oL}_1$ is a bounded kernel over $b\mathcal B_{\mathbf W_\star}(\mathscr O)$. 
Then, one can consider  the bounded kernel $Q^{oL}_{1,\mathbf W_\star}$ on $b\mathcal B(\mathscr O)$ defined by
 $$Q^{oL}_{1,\mathbf W_\star}(x,dy)= \frac{\mathbf W_\star(y)}{\mathbf W_\star(x)} Q^{oL}_1(x,dy),$$
which has the  same essential spectral radius as $Q^{oL}_1|_{b\mathcal B_{\mathbf W_\star}(\mathscr O)}$.  Because  
$$Q^{oL}_{1,\mathbf W_\star}\le T^{oL}_{1,\mathbf W_\star},$$ it follows that, when considering $Q^{oL}_{1,\mathbf W_\star}$ as a kernel operator over $\mathscr R_0^d$: 
\begin{enumerate}
\item[-]  $Q^{oL}_{1,\mathbf W_\star}$   satisfies  \textbf{(A1)} since  $T^{oL}_{1,\mathbf W_\star}$ satisfies this condition (see the second step of the proof of Theorem \ref{th.Ress}).

\item[-] By \eqref{eq.Tb}, $\beta_w(Q^{oL}_{1,\mathbf W_\star})=0$. 
 
\end{enumerate}
Hence, we have  thanks to \eqref{eq.Gelfand}, $\mathsf r_{ess}(Q^{oL}_{1,\mathbf W_\star}|_{b\mathcal B(\mathscr R_0^d)})=0$ and  consequently:
$$  \mathsf r_{ess}(Q^{oL}_1|_{b\mathcal B_{\mathbf W_\star}(\mathscr O)})=\mathsf r_{ess}(Q^{oL}_1|_{b\mathcal B_{\mathbf W_\star}(\mathscr R_0^d)})=0.$$
% ress(P_t^D|bB(O))= ress(P_t^D|bB(S)), où S= R^{d}_0 est l'espace total. Facile à voir en revenant à la définition du spectre essentiel (Fredholm). On montre en effet que pour lambda\neq 0, \lambda \in  \sigma_ess(P_t^D|bB(O)) ssi  \lambda \in  \sigma_ess(P_t^D|bB(S))
\begin{sloppypar}
Note that  $Q^{oL}_{1,\mathbf W_\star}$ is  strongly Feller  as   $Q^{oL}_1$.  Indeed, to see it, one can argue as follows. Recall $\mathbf W_\star$ is continuous. For all compact set $K$ of $\mathscr O$, $f\in b\mathcal B(\mathscr O)$, we have, setting  $f_n=   ( [\mathbf W_\star\wedge n] /\mathbf W_\star)f\in  b\mathcal B(\mathscr O)$ ($n\ge 1$), $\sup_{x\in K}|Q_1^{oL}(\mathbf W_\star f)(x)-Q_1^{oL}(\mathbf W_\star f_n)(x)|\le \sup_{x\in K}|T_1^{oL}\mathbf W_\star^p(x)|^{1/p} \sup_{x\in K}|T_1^{oL}|f-f_n|^q(x)|^{1/q}$ which goes to $0$ as $n\to +\infty$ by \eqref{eq.=p-} and Dini's theorem.  Hence $Q_1^{oL}(\mathbf W_\star f)$ is continuous over $\mathscr O$. This proves that  $Q^{oL}_{1,\mathbf W_\star}$ is  strongly Feller.    

 Finally, we claim that  
\begin{equation}\label{eq.RSP}
\mathsf r_{sp}(Q^{oL}_1|_{b\mathcal B_{\mathbf W_\star}(\mathscr O)})>0.
\end{equation}
 Let us prove  \eqref{eq.RSP}.
    Let  $r>k_{\mathbf S}$  and   $\phi(\cdot):=\int_{0}^{+\infty} e^{-rs} Q_s^{oL}(x_*,\cdot)ds$ where $x_*\in \mathscr O$ is arbitrary. If $\phi(\mathcal A)>0$,  $\exists s_{\mathcal A}>0$ such that $Q_{s_{\mathcal A}}^{oL}(x_*,\mathcal A)>0$. By Theorem~\ref{th.SF1}, there exists a  neighborhood $U_*\neq \emptyset$ of $x_*$ in $\mathscr O$ and  $c_*>0$ such that $Q_{s_{\mathcal A}}^{oL}(y ,\mathcal A)\ge c_*$ for all $y\in U_*$. Let $n>  s_{\mathcal A}$ and write $n =t+{s_{\mathcal A}}$ with $t>0$. Let $x\in \mathscr O$. We have $Q_n^{oL}(x,\mathcal A)= \int_{\mathscr O} Q_t^{oL}(x,dy)Q_{s_{\mathcal A}}^{oL}(y,\mathcal A)\ge c_* Q_t^{oL}(x, U_*) >0$ by  Proposition \ref{pr.TI}. In particular,   $Q_1^{oL}$ is $\phi$-irreducible~\cite[Definition 2.2]{Nummelin}. 
 Let $\psi$ be    a maximal irreducibility measure for  $Q_1^{oL}$. By~\cite[Theorem 2.1]{Nummelin}, there exist $m\ge 1$, $\beta>0$, $\mathbf s: \mathscr O\to \mathscr R_+$ measurable  with $\psi(\mathbf s)>0$ and a   measure $\nu\ge 0$ over $(\mathscr O,\mathcal B(\mathscr O))$ with  $\nu(\mathscr O)>0$ such that    
\begin{equation}\label{eq.Lpetite}
Q_m^{oL} (x, \mathcal V)\ge \beta \, \mathbf s(x)\,  \nu(\mathcal V),\ \ \forall x\in \mathscr O,\, \forall  \, \mathcal V \in \mathcal B(\mathscr O).
\end{equation}
Note that in particular $\nu(\mathscr O)<+\infty$ and $\mathbf s\in b\mathcal B(\mathscr O)$. 
%Up to changing, for some $n\ge 1$ large enough, $\mathbf s$ into $Q_n^{oL} \mathbf s$ and $m$ into $m+n$ in \eqref{eq.Lpetite}, 
One can assume    $\nu(\mathbf s)>0$. Indeed, $Q_n^{oL} \mathbf s$ is small by  multiplying  both side of \eqref{eq.Lpetite} by $Q_n^{oL}$. Let $\mathcal A=\{\mathbf s>0\}$. Since $\psi(\mathcal A)>0$, it holds by the previous analysis, for all $x\in \mathscr O$ and $n> s_{\mathcal A}$,  $Q_n^{oL}(x,\mathcal A)>0$  and thus $Q_n^{oL}\mathbf s(x)>0$. In particular,  $\nu(Q_n^{oL}\mathbf s)>0$.  
 
%The rest of the proof of \eqref{eq.RSP} is then rather standard. 
Set   $Q=Q_m^{oL} $. 
 We  have by \eqref{eq.Lpetite}, for every $f\in b\mathcal B_{\mathbf W_\star}(\mathscr O)$, $f\ge 0$,  $\nu (Qf)\ge \beta \nu(\mathbf s)   \nu(f) $ (hence, $\nu (Q^qf)\ge \nu(f)\beta^q\nu (\mathbf s)^q $, $\forall q\ge 1$). Let $x\in \mathscr O$ such that $\mathbf s(x)>0$. Since  $\Vert \mathbf 1\Vert_{b\mathcal B_{\mathbf W_\star}(\mathscr O)} \le 1$:
\begin{align*}
\Vert Q^k\Vert_{\mathcal L(b\mathcal B_{\mathbf W_\star}(\mathscr O))}\ge \frac{( Q^k \mathbf 1)(x)}{\mathbf W_\star(x)} \ge  \beta \mathbf s(x) \frac{ \nu( Q^{k-1} \mathbf 1)}{\mathbf W_\star(x)}\ge \beta^k \mathbf s(x) \,  \frac{ \nu(\mathbf 1)\nu (\mathbf s)^{k-1}}{\mathbf W_\star(x)}, \ \forall k\ge 1.
\end{align*} 
By Gelfand's formula, $\mathsf r_{sp}(Q|_{b\mathcal B_{\mathbf W_\star}(\mathscr O)}) \ge \beta \nu(\mathbf s)>0$. Equation \eqref{eq.RSP} follows from the fact that $\mathsf r_{sp}(Q^{oL}_1|_{b\mathcal B_{\mathbf W_\star}(\mathscr O)})^m=\mathsf r_{sp}(Q|_{b\mathcal B_{\mathbf W_\star}(\mathscr O)}) $.  
 The  spectral  gap \eqref{eq.spga}   is thus proved.
 \end{sloppypar} 
  \medskip

\noindent
 \textbf{Step 1.c.}
 Thanks to the previous steps, we can  apply  Theorem~\ref{thm41}  with the operator  $Q^{oL}_1$ on $b\mathcal B_{\mathbf W_\star}(\mathscr O)$. 
We then deduce, setting 
$$\lambda_p:=-\log \mathsf r_{sp}(Q^{oL}_1 |_{b\mathcal B_{\mathbf W_\star}(\mathscr O)})\in \mathscr R,$$ that there is a unique couple $(\rho_p, \varphi_p)$ such that all the following conditions   hold:
  \begin{enumerate}
 \item[ -]   The measure $\rho_p$ is a probability measure on $\mathscr O$, $\rho_p(\mathbf W_\star)<+\infty$, $\varphi_p\in \mathcal C_{b{\mathbf W_\star}}(\mathscr O)$, and $\rho_p(\varphi_p)=1$.
\item[-] It holds $
\rho_p   Q^{oL}_1    = e^{-\lambda_p}\rho_p $ and $  Q^{oL}_1  \varphi_p= e^{-\lambda_p}\varphi_p$.  
\item[-] There exist $C\ge 1$ and $\delta>0$, such that for  all $f\in b\mathcal B_{\mathbf W_\star}(\mathscr O)$  and $n\ge 1$, 
\begin{equation}\label{eq.cvbl}
 \big \|e^{n\lambda_p }Q^{oL}_n f -\rho_p(f) \varphi_p\big \| _{b\mathcal B_{\mathbf W_\star}(\mathscr O)} \le C e^{-\delta n}\|f\| _{b\mathcal B_{\mathbf W_\star}(\mathscr O)}.
\end{equation}
\end{enumerate}
In addition $\rho  _p(O)>0$ for all nonempty open subsets $O$ of $\mathscr O$ and $\varphi_p$ is positive everywhere  on $\mathscr O$. In particular, one has: 
\begin{enumerate}
\item[-] If  $\nu\in \mathcal M_{b\mathbf W_\star}(\mathscr O) $  satisfies for some $\eta \in \mathscr R$, $\nu Q^{oL}_1= \eta \nu$ and $\nu(\varphi_p)\ne0$, then $\eta=e^{-\lambda_p}$ and $\nu=c\rho_p$ for some constant $c$.
 \item[-] If $g\in b\mathcal B_{\mathbf W_\star}(\mathscr O)$ satisfies for some $\eta \in \mathscr R$,  $Q^{oL}_1g=\eta g$  and  $\rho_p(g)\ne 0$, then $\eta=e^{-\lambda_p}$ and $g=c\varphi_p$ for some constant $c$.
 \end{enumerate}
The space $\mathcal M_{b\mathbf W_\star}(\mathscr O)$ is defined by $\mathcal M_{b\mathbf W_\star}(\mathscr O)=  \{\nu\in \mathcal M_{b}(\mathscr O), \mathbf W_\star( x)\nu(d x) \in \mathcal M_{b}(\mathscr O)  \}$. 
\medskip

\noindent
 \textbf{Step 2.} End of the proof of Theorem \ref{th.oL}. 
   We claim that for all $p>1$:
  \begin{enumerate}
 \item[-] The measure $\rho_p$ is the unique   q.s.d. in $b \BB_{\mathbf W_*}(\mathscr O)$ for  $(P_t^{oL},t\ge 0)$ (see \eqref{eq.FK-sur2} and Definition \ref{def.QSD}).   
 In addition, 
 \begin{equation}\label{eq.rspOl}
 \mathsf r_{sp}(Q_t^{oL}|_{b\BB_{\mathbf W_*}(\mathscr O)})=\lambda(t)=e^{-\lambda_{p} t}, \ \forall t\ge 0.
 \end{equation} 
 \item [-] 
 There exists $\delta,C>0$ such that for all $\nu\in \mathcal P_{\mathbf W_\star}(\mathscr O) $ and $f\in  b \BB_{\mathbf W_*}(\mathscr O)$:
 $$
\big | \nu P_t^{oL}f-\rho_{p}(f)\big |\le C e^{-\delta t} \frac{\nu({\mathbf W_*})}{\nu(\varphi_p)}\|f\|_{b\BB_{\mathbf W_*}(\mathscr O)},\ \  \forall t>0.
$$ 
   \end{enumerate}
We sketch the proofs of  the previous claims (see the second step of the proof of \cite[Theorem 5.3]{guillinqsd} for more details). To show  the first claim, one first proves that for all $t\ge 0$, there exists $\lambda(t)$ such that $\rho_{p}Q_t^{oL} = \lambda(t) \rho_{p}$. By the semigroup property, $\lambda(t+s)=\lambda(t)   \cdot\lambda(s)$ for all $s,t\ge 0$. Since $\lambda(1)=e^{-\lambda_{p}}$,  one deduces that $
\lambda(t)=e^{-\lambda_p t}$, $t\ge 0$. Because $\rho_p Q_t^{oL}(\mathscr O) = e^{-\lambda_p t}$, one also deduces that $\rho_{p}$ is a q.s.d.  (namely that $\rho_p P_t^{oL}(\cdot)= \rho_p(\cdot)$, $t\ge 0$). 
Then, one proves~\eqref{eq.rspOl} and the uniqueness of the q.s.d. in $b \BB_{\mathbf W_*}(\mathscr O)$ using   item $\mathfrak a$ in Theorem~\ref{thm41} (applied to $Q_t^{oL}$). 
% appliqué à Q^oL_t, où t>0 fixé
% On check les HP du THM 7 (Perron-Frob):
%  Q^oL_{t} est TI et SF pour t>0: prouvé. 
%  r_sp(Q^oL_t) >0 car Q^oL_t est TI et SF: vu avec Nummelin
% r_ess(P^D_{t}) =0: vu un peu plus haut
% Donc trou spectral -> on peut appliquer le THM 4.1
The last claim follows easily from \eqref{eq.cvbl}.

 It thus remains to show: 
 
 \begin{enumerate}
 \item[\textbf{i}.] The compactness  of $Q^{oL}_t:b\mathcal B_{\mathbf W^{1/p}}(\mathscr O)\to b\mathcal B_{\mathbf W^{1/p}}(\mathscr O)$, for $t>0$.
 \item[\textbf{ii}.] The fact that  $\lambda_p \in [i_{\mathbf V_{\mathbf S},\mathscr O},+\infty)$ 
 (where we set $ i_{\mathbf V_{\mathbf S},\mathscr O}:=\inf_{\mathscr O} \mathbf V_{\mathbf S}$).
  \item[\textbf{iii}.] If moreover $\mathscr R^d_0\setminus \overline{\mathscr O}$ is non-empty or  if  $\mathscr O=\mathscr R^d_0$, then  $\lambda_p>i_{\mathbf V_{\mathbf S},\mathscr O}$. 
  \end{enumerate}
  %%%
  
 Note that Item \textbf{ii} is obvious  (i.e. $\lambda_p\ge  i_{\mathbf V_{\mathbf S},\mathscr O} $) since 
  $$ \mathbb E_{\rho_p}\Big[ e^{  - \int_0^t  \mathbf V_{\mathbf S}(X_s)ds}   \mathbf 1_{t<\sigma_{\mathscr O}}\Big] = \rho_p (Q_t \mathbf 1) = e^{-\lambda_p t}  \rho_p(\mathbf 1)= e^{-\lambda_p t}.$$

\noindent
 \textbf{Step 2.a}.  Proof of Item \textbf i.  Let us prove that  for all $t>0$,    $Q^{oL}_t:b\mathcal B_{\mathbf W^{1/p}}(\mathscr O)\to b\mathcal B_{\mathbf W^{1/p}}(\mathscr O)$ (or equivalently,  $Q^{oL}_{t,\mathbf W_\star}$ over $b\mathcal B(\mathscr O)$) is compact.  Recall that for all  $t\ge 0$, $Q^{oL}_{t,\mathbf W_\star}\le T^{oL}_{t,\mathbf W_\star}$ over $\mathscr R^d_0$.  By \eqref{eq.Tb} (which actually holds for all $t>0$), $\beta_w(Q^{oL}_{t,\mathbf W_\star})=0$ for all $t>0$. We also have that (see the lines just above \eqref{eq.rr}), $\beta_\tau (\mathbf 1_{K} Q^{oL}_{t,\mathbf W_\star})=0$ for all compact subset $K$ of $\mathscr R^d_0$, $t>0$. For any $t>0$, write $t=3s$,  with $s>0$. Then, using~\cite[Proposition 3.2 (f)]{Wu2004}, we have
$\beta_\tau (Q^{oL}_{2s,\mathbf W_\star})=0$. Consequently, thanks to~\cite[Proposition 3.2 (g)]{Wu2004}, $Q^{oL}_{3s,\mathbf W_\star}=Q^{oL}_{s,\mathbf W_\star}Q^{oL}_{2s,\mathbf W_\star}$ is compact over $b\mathcal B_{\mathbf W^{1/p}}(\mathscr R_0^d)$ and thus clearly also on $b\mathcal B_{\mathbf W^{1/p}}(\mathscr O)$. 
  \medskip

\noindent
 \textbf{Step 2.b}.  Proof of Item \textbf{iii}.  
  Let us  prove  that $\lambda_p>i_{\mathbf V_{\mathbf S},\mathscr O}$ when $\mathscr R^d_0\setminus \overline{\mathscr O}$ is non-empty  or $\mathscr O=\mathscr R^d_0$.  Assume that it is not the case, i.e.  that $\lambda_p=i_{\mathbf V_{\mathbf S},\mathscr O}$. Then, since for all $t\ge 0$,  $\rho _p(Q_t^{oL}\mathbf 1)=e^{-\lambda_p t}\mathbf 1 =e^{-i_{\mathbf V_{\mathbf S},\mathscr O} t}\mathbf 1 $, we get that $\rho _p( e^{-i_{\mathbf V_{\mathbf S},\mathscr O} t}\mathbf 1-Q^{oL}_t\mathbf 1)=0$. Since the function 
 $e^{-i_{\mathbf V_{\mathbf S},\mathscr O} t}\mathbf 1-Q^{oL}_t\mathbf 1$ is non negative and continuous over $\mathscr O$
 together with the fact that $\rho _p$ charges all non-empty open subsets of $\mathscr O$, we deduce that $e^{-i_{\mathbf V_{\mathbf S},\mathscr O} t}\mathbf 1=Q^{oL}_t\mathbf 1$ on $\mathscr O$, i.e. for all $t\ge 0$ and $x\in \mathscr O$, 
\begin{equation}\label{eq.=I}
\mathbb E_x\big [e^{-\int_0^t\mathbf V_{\mathbf S}(X_s)ds}\mathbf 1_{t<\sigma^{oL}_{\mathscr O}}\big ]=e^{-i_{\mathbf V_{\mathbf S},\mathscr O} t}.
 \end{equation}
 
Let us first consider the case when $\mathscr R^d_0\setminus \overline{\mathscr O}$ is non-empty. Then, using \eqref{eq.=I}, we have  $e^{-i_{\mathbf V_{\mathbf S},\mathscr O} t} \le e^{-i_{\mathbf V_{\mathbf S},\mathscr O} t} \mathbb P_x[t<\sigma^{oL}_{\mathscr O}]$ so that for all $t>0$ and $x\in \mathscr O$, $ \mathbb P_x[t<\sigma^{oL}_{\mathscr O}]=1$ 
and thus  
$$\mathbb P_x[\sigma^{oL}_{\mathscr O}=+\infty]=1, \forall x\in \mathscr O.$$ 
However, as $\mathscr R^d_0\setminus \overline{\mathscr O}$ is open, there exists a non-empty open ball $B\subset \mathscr R^d_0\setminus \overline{\mathscr O}$. Moreover,  there exist   $x_0\in \mathscr O$ and   $t_0>0$ (actually this is true for all $x_0\in \mathscr O$ and $t_0>0$) such that 
$$\mathbb P_{x_0}[X_{t_0}\in B]>0.$$
Hence,  $0<\mathbb P_{x_0}[X_{t_0}\in B]\le\mathbb P_{x_0}[ \sigma^{oL}_{\mathscr O}\le t_0]$. A contradiction. 
 
Assume now that   $\mathscr O=\mathscr R^d_0$ so that by Lemma \ref{le.finiteLsur}, for all $x\in \mathscr R^d_0$ a.s. $$\sigma^{oL}_{\mathscr R^d_0}(x)=+\infty.$$
Using \eqref{eq.=I}, we then have for all $t\ge 0$, 
$$\mathbb E_x  [e^{-\int_0^t\mathbf V_*(X_s)ds}  ]=1,$$ 
 where $\mathbf V_*=\mathbf V_{\mathbf S}-i_{\mathbf V_{\mathbf S},\mathscr R^d_0}\ge 0$. Consequently, for all $t\ge 0$ and $x\in \mathscr O$, $\mathbb P_x  [\int_0^t\mathbf V_*(X_s)ds=0  ]=1$. Then,     $
 \mathbb P_x  [ \mathbf V_*(X_s) =0,  \text{for almost every }s\in [0,t] ]=1$. 
 Due to the fact that  $\mathbf V_*$ is continuous over $\mathscr R^d_0$ and since a.s.  the trajectories of the process $(X_s,s\ge 0)$ are also continuous, for all $t\ge 0$ and all $x\in\mathscr O$ (actually \textit{càdlàg} sample paths would be enough), it holds $\mathbb P_x  [ \mathbf V_*(X_s) =0,  \text{for all }s\in [0,t) ]=1$. 
   In particular, for all $s>0$ and all $x\in\mathscr O$, 
     \begin{equation}\label{eq.Co}
  \mathbb P_x  [ \mathbf V_*(X_s) =0 ]=1.
   \end{equation} 
On the other hand,   since $\mathbf V_{\mathbf S}\neq i_{\mathbf V_{\mathbf S},\mathscr R^d_0}$ and $ \mathbf V_{\mathbf S}$ is continuous,  there exist $c>0$ and a  nonempty   ball $ B \subset \mathscr R^d_0$ such that $ \mathbf V_{\mathbf S}\ge i_{\mathbf V_{\mathbf S},\mathscr R^d_0}+c$ over $ B $. Furthermore,     there exist     $x_0\in \mathscr O$ and  $s_0>0$ (actually this is true for all $x_0$ and $s_0>0$) such that,  
$$\mathbb P_{x_0}[X_{s_0}\in   B] >0.$$  This contradicts \eqref{eq.Co} and shows that $\lambda_p>i_{\mathbf V_{\mathbf S},\mathscr O}$.  
The proof of Theorem \ref{th.oL} is thus complete.

 %%%%

 %%%

 \subsection{A general result and proof of Theorem \ref{th.L}}
 \label{sec.alp}

 \subsubsection{A general result}
 In view of the proof of Theorem~\ref{th.oL}, we have the following result, where we use the notation and the assumptions of Section \ref{sec.Xx} (recall in particular  \eqref{eq.C--1} and \eqref{eq.C--2}).  Assume that the process 
 $(\mathfrak X_t, t\ge 0)$ satisfies the following conditions (below $(U^{\mathfrak X}_t,t\ge 0)$ denotes the semigroup\footnote{Defined by $U^{\mathfrak X}_tf(\mathfrak z)=\mathbb E_{\mathfrak z}[f(\mathfrak X_t)]$, $\mathfrak z\in \mathscr M$ and $f\in b\mathcal B(\mathscr M)$.} of the process  $(\mathfrak X_t, t\ge 0)$):
 
\begin{enumerate}
 \item[]\textbf{[SF$_1$]} For all $t>0$ and $f\in b\mathcal B(\mathscr M)$, the mapping  $\mathfrak z\in\mathscr M \mapsto T^{\mathfrak X}_tf(\mathfrak z):=\mathbb E_{\mathfrak z} [f(\mathfrak X_t)   \, e^{  - \int_0^t  \mathbf V (\mathfrak X_s)ds}  ]$ is continuous.  
 
 \item[]\textbf{[SF$_2$]} For all $t>0$ and $f\in b\mathcal B(\mathscr V)$,  $\mathfrak z\in \mathscr V \mapsto Q^{\mathfrak X}_tf(\mathfrak z)$ is continuous.  
 
     \item[] \textbf{[Ti]} For all $t>0$, $\mathfrak z \in  \mathscr V$ and all nonempty open subset $O$ of $\mathscr V$, $Q^{\mathfrak X}_t(\mathfrak z,O)>0$.

      \item[]  \textbf{[L$_{{\rm yap}}$]}  There exists a continuous function $\mathbf W:\mathscr M\to [1,+\infty)$ in the extended domain $\mathbb D_e(\mathscr L)$ of the generator\footnote{See~\cite[Chapter VII]{revuz2013continuous},~\cite{davis1993markov}, or \cite{guillinqsd} and references therein for a definition.}   $\mathscr L $ of $(U^{\mathfrak X}_t,t\ge 0)$  such that for all $p>1$, there exist 
        positive constants $(\mathfrak r_n)$ and $(\mathfrak b_n)$ where $\mathfrak r_n\to +\infty$, and an increasing sequence of compact subsets $(\mathscr K_n)$ of $\mathscr M$, such that for all $\mathfrak z \in \mathscr M$, 
$$\frac{ \mathscr L  \mathbf W(\mathfrak z) }{\mathbf W(\mathfrak z) } - p\mathbf V (\mathfrak z)  \le - \mathfrak r_n + \mathfrak b_n \mathbf 1_{\mathscr K_n}(\mathfrak z).$$
      
%      such that for all $\mathfrak z\in \mathscr M $ satisfying $\mathfrak z\to \partial \mathscr M$ and, if $\mathscr M$ is unbounded, $|\mathfrak z|\to +\infty$, it holds for all $c>0$:  
      \item[] \textbf{[C$_{{\rm traj}}$]} For all $t\ge 0$, the mapping  $\mathfrak z \in \mathscr M  \mapsto \mathbb P_{\mathfrak z}[\mathfrak X_{[0,t]} \in \cdot ]\in \mathcal P(\mathcal D([0,t],\mathscr M))$ is continuous for the weak topology. 
       \end{enumerate}
       
       %%%
        Note also that thanks to \eqref{eq.C--1}, {\rm\textbf{[Ti]}} is equivalent to the fact that $(\mathfrak X_t,t\in [0,\sigma^\mathfrak X_{\mathscr V}))$ is topologically irreducible over $\mathscr V$. 
 Then, under the five assumptions above, we have the following result whose proof is left to the reader since it is a direct generalization of the proof of Theorem~\ref{th.oL}.   
 %%%%

 \begin{thm}\label{th.G}
 Assume {\rm \textbf{[SF$_1$]}}, {\rm \textbf{[SF$_2$]}}, {\rm \textbf{[Ti]}}, {\rm \textbf{[L$_{{\rm yap}}$]}}, and {\rm \textbf{[C$_{{\rm traj}}$]}}. 
 Then, for all $p>1$, the killed Feynman-Kac semigroup $(Q^{\mathfrak X}_t,t\ge 0)$   is $\mathbf W^{1/p}$ compact-ergodic over $\mathscr V$. In addition, if:
 \begin{itemize}
 \item {\rm \textbf{[O1]}}  $\mathbb P_{\mathfrak z^*}[\sigma^\mathfrak X_{\mathscr V}<\infty]>0$ for some  $\mathfrak z^*\in \mathscr V$, or, 
   \item {\rm \textbf{[O2]}}  if  $\mathscr V=\mathscr M$ and $\mathbf V$ is a non constant  function  over $\mathscr M$,
    \end{itemize}
    then $\lambda_p>  \inf_{\mathscr V} \mathbf V$  where $\lambda$
 is the principal eigenvalue of $(Q^{\mathfrak X}_t,t\ge 0)$ in   $ b\mathcal B_{\mathbf W^{1/p} }(\mathscr V)$.
  \end{thm}

Let us just give some  indications and remarks  on the proof of Theorem~\ref{th.oL}. 
On the one hand,  for $p>1$,  $\mathbf W^{1/p}\in \mathbb D_e(\mathscr L)$, and hence
  $$t\mapsto \mathbf M_t:= \frac{\mathbf W^{1/p}(\mathfrak X_t)}{\mathbf W^{1/p}(\mathfrak X_0)} \exp \big (-\int_0^t \frac{\mathscr L \mathbf W^{1/p}}{\mathbf W^{1/p}} (\mathfrak X_s) ds\big )$$
  is a local martingale, which   is a supermartingale by  Fatou’s lemma. This is a key argument in the proof of  Theorem \ref{th.G}. Note also that {\rm  \textbf{[L$_{{\rm yap}}$]}} implies that 
  $$ { \mathscr L  \mathbf W^{1/p}  }/{\mathbf W^{1/p} } - \mathbf V  \le \mathfrak m=\mathfrak b_0/p.$$ Hence,  using $(\mathbf M_t,t\ge 0)$, $ T_t^{\mathfrak X}\mathbf W^{1/p}(\mathfrak z)=  \mathbb E_\mathfrak z [\mathbf W^{1/p}(\mathfrak X_t)    \, e^{ -  \int_0^t   \mathbf V (\mathfrak X_s) ds}  ] \le e^{\mathfrak mt} \mathbf W^{1/p}(\mathfrak z)$, which proves \eqref{eq.Q1} in this general setting. Moreover, let us recall that  {\rm \textbf{[C$_{{\rm traj}}$]}} was indeed used in the third step of the proof of Theorem~\ref{th.Ress}.  
 
  %%%
  
In view of Theorem \ref{th.SF1}, one way to check {\rm \textbf{[SF$_1$]}} and {\rm \textbf{[SF$_2$]}}   is the following. 

%%%

  \begin{thm}\label{th.SFG}
  Assume that:
\begin{itemize}
\item[] {\rm \textbf{[SF$_0$]}} For every $t>0$ and  $f\in b\mathcal B(\mathscr M)$,  $\mathfrak z\in \mathscr M \mapsto U^{\mathfrak X}_t f(\mathfrak z)$ is continuous.     
\item[]     {\rm \textbf{[P$_{t=s}$]}} The mapping $\mathfrak z \in \mathscr M\mapsto \mathfrak X_t(\mathfrak z)$ is continuous in $\mathbb P$-probability.  
    \item[] {\rm \textbf{[D$_{\rm e}$]}} For all $t>0$ and $\mathfrak z\in \mathscr M$,  $\mathfrak X_t(\mathfrak  z)$ has a density w.r.t. the Lebesgue measure. 
 
\item[] {\rm \textbf{[$I\mathbf V_{\mathbf S}$]}} For all $t\ge 0$ and all $\mathfrak  z_n \to \mathfrak  z\in \mathscr M$, $ \int_0^t \mathbf V (\mathfrak X_s(\mathfrak  z_n))ds\to \int_0^t \mathbf V (\mathfrak X_s(\mathfrak  z))ds$ in $\mathbb P$-probability. 
 %%%
  \end{itemize}
%%%
  Then,  {\rm \textbf{[SF$_1$]}} is satisfied. Assume in addition that 
  \begin{itemize}
  \item [] {\rm \textbf{[B]}} For all compact   subset $K$ of $\mathscr V$,  $\lim_{s\to 0^+}\sup_{\mathfrak  z\in K}  \mathbb P_{\mathfrak z}[\sigma^\mathfrak X_{\mathscr V}\le s]=0$.
  \end{itemize}
Then {\rm \textbf{[SF$_2$]}}   holds.   
  \end{thm}
 
 \begin{proof} 
  Using \textbf{[SF$_0$]} together with  \textbf{[P$_{t=s}$]} and  \textbf{[D$_{\rm e}$]}, 
    we deduce as in Step A.2 in  the proof of Theorem \ref{th.SF1} that for every $f\in b\mathcal B(\mathscr M)$, and for all  $\mathfrak z_n\to \mathfrak z\in \mathscr M$, it holds as $n\to +\infty$:
  $$   f(\mathfrak X_s(\mathfrak  z_n))\to f(\mathfrak X_s(\mathfrak z)) \text{ in $\mathbb P$-probability}.$$
 Using {\rm \textbf{[$I\mathbf V_{\mathbf S}$]}} and the fact that $\mathbf V$ is lower bounded, we deduce as in Step A.3 in  the proof of Theorem \ref{th.SF1} that for every $f\in b\mathcal B(\mathscr M)$, and for all  $\mathfrak z_n\to \mathfrak z\in \mathscr M$, one has:
 $$\mathbb E_{\mathfrak z_n} [f(\mathfrak X_t)   \, e^{  - \int_0^t  \mathbf V (\mathfrak X_s)ds}   ]\to \mathbb E_{\mathfrak z} [f(\mathfrak X_t)   \, e^{  - \int_0^t  \mathbf V (\mathfrak X_s)ds}  ],$$
 which proves \textbf{[SF$_1$]}.
 We then conclude the proof of {\rm \textbf{[SF$_2$]}} using {\rm \textbf{[B]}} and the same arguments as those used in Step B in  the proof of Theorem \ref{th.SF1}. 
 \end{proof}

 %%% 

 \subsubsection{Proof of Theorem \ref{th.L}}
 %%%%
 Let $(L_t,t\ge 0)$ be a Lévy process over $\mathscr R^d$  satisfying {\rm \textbf{[L1]}} and recall Lemma \ref{le.finitealpha-s}. Let $\mathscr O$ be a subdomain  of $\mathscr R^d_0$ (see \eqref{eq.RD0}). Assume that  $(L_t,t\ge 0)$  satisfies  {\rm \textbf{[L2]}}, and  assume also {\rm \textbf{[S2]}}. 
 To prove Theorem \ref{th.L} we apply Theorems \ref{th.G} and \ref{th.SFG} to the Lévy process $(L_t,t\ge 0)$ over $\mathscr R^d_0$. 
 Let us denote by $\mathscr L^{Le}$ the generator of the Lévy process $(L_s,s\ge 0)$, see~\cite[Theorem 6.8]{schilling2016introduction}.    
The  non-killed Feynman-Kac semigroup $(T^{Le}_t, t\ge 0)$ over the   space  $\mathscr R^d_0$ associated  with the potential $ \mathbf V_{\mathbf S}$  and the  process    $(L_s,s\ge 0)$      is defined by 
\begin{equation}\label{eq.FK-B1-NK}
T_t^{Le}f(x)=\mathbb E_x\Big[f(L_t)   \, e^{  - \int_0^t  \mathbf V_{\mathbf S} (L_s)ds}   \Big], \ \text{$t\ge 0$, $x\in \mathscr R^d_0$, and $f\in b\mathcal B(\mathscr R^d_0)$}.
 \end{equation}
 Let us first check Assumptions  \textbf{[SF$_1$]} and \textbf{[SF$_2$]} using  
 Theorem \ref{th.SFG} for the process $(L_t,t\ge 0)$ over $\mathscr M=\mathscr R^d_0$ and with $\mathbf V=\mathbf V_{\mathbf S}$. 
By assumption,  for all $x\in \mathscr R^d$, $L_s(x)$ has a density w.r.t. the Lebesgue measure $dy$ (and more precisely,  $\mathbb E_z[f(L_s)]=\int p_s(z-y)f(y)dy$, for every $z\in \mathscr R^d$ and $f\in b\mathcal B(\mathscr R^d$)). In particular it is well-known that this implies that the semigroup of $(L_s,s\ge 0)$ is strongly Feller (see e.g.~\cite[Lemma 4.9]{schilling2016introduction}). Moreover, for all $y,z\in \mathscr R^d$, 
  \begin{equation}\label{eq.coo}
 \sup_{s\ge 0}| L_s(y)-L_s(z)|=|y-z|.
 \end{equation}
Thus, for all $t\ge 0$ and all $(x_n)_n\subset  \mathscr R^d_0$ such that $x_n\to x\in \mathscr R^d_0$, we have that almost surely:
  $$ \int_0^t \mathbf V_{\mathbf S}(L_s(x_n))ds\to \int_0^t \mathbf V_{\mathbf S}(L_s(x))ds.$$
 This follows from a dominated convergence theorem together with   \eqref{eq.coo}. 
Indeed, recall that by  \eqref{eq.Ty},  $\inf_{s\in [0,t]}|L_s(x)|>0$ almost surely. Then, for $\epsilon\in (0, \inf_{s\in [0,t]}|L_s(x)|/2)$, there exists $n_\epsilon\ge 1$ such that   for all $n\ge n_\epsilon$ and all $s\in [0,t]$, $|L_s(x_n)|\ge |L_s(x)|-\epsilon\ge \inf_{s\in [0,t]}|L_s(x)|/2>0$. Hence there exists $c>0$ such that for all $n\ge n_\epsilon$ and all $s\in [0,t]$,  $|\mathbf V_{\mathbf S}(L_s(x_n))|\le c$. This proves \textbf{[$I\mathbf V_{\mathbf S}$]}. Moreover, \textbf{[SF$_1$]}. By Theorem \ref{th.SFG}, $T_t^{Le}$ is strongly Feller for $t>0$. On the other hand, 
 for all compact subset $K$ of $\mathscr R^d$ and $\delta>0$ it holds 
 $$
 \lim_{s\to 0^+}\sup_{x\in K}  \mathbb P_x[\sigma^{Le}_{B(x,\delta)}\le s]=0,$$
where $\sigma^{Le}_{B(x,\delta)}(x):=\inf\{t\ge 0, L_t(x)\notin {B(x,\delta)}\}$. 
Note indeed that   as $s\to 0^+$
$$\mathbb P_x[\sigma^{Le}_{B(x,\delta)}\le s]=\mathbb P[\sigma^{Le}_{B(0,\delta)}\le s]\to 0$$  since a.s.   $L_t\to L_0=0$ when $t\to 0^+$.  This proves \textbf{[B]} and then,     \textbf{[SF$_2$]} by  Theorem \ref{th.SFG}. 

We now check the other assumptions in Theorem \ref{th.G}. 
First the fact that \textbf{[Ti]} is satisfied (i.e.  for all $t>0$,   all $x\in \mathscr O$,  and all non-empty open subset $O$ of $\mathscr O$, $Q_t^{Le}(x,O)>0$)  follows from \textbf{[L2]}  together with the fact that  $ \int_0^t  \mathbf V_{\mathbf S}(L_s(x))ds$ is  a.s.  finite for all $x\neq 0$ (see Lemma \ref{le.finitealpha-s}).  In addition, the following Lyapunov condition is satisfied thanks to {\rm \textbf{[S2]}}. When  
  $x\to \partial \mathscr R^d_0\cup \{\infty\}$ ($x\in \mathscr R^d_0$), where we recall that $\partial \mathscr R^d_0=\{0\}$, 
 $$\frac{\mathscr L^{Le}   \mathbf 1(x) }{\mathbf 1(x) } -p\mathbf V_{\mathbf S}(x)=-p\mathbf V_{\mathbf S}(x)\to -\infty, \ \forall p>1.$$ 
Finally, the mapping $z\in \mathscr R^d_0\mapsto \mathbb P[L_{[0,t]}(z)\in \cdot ]\in \mathcal P(\mathcal D([0,t],\mathscr R^d_0))$
 is continuous for the weak topology (this is \textbf{[C$_{{\rm traj}}$]}). 
 This function is well-defined and continuous because we first have for all $z\in  \mathscr R^d_0$ and $t\ge 0$, $\mathbb P_x[ L_{[0,t]} \in \mathcal D([0,t],\mathscr R^d_0)]=1$ and we also have, thanks to \eqref{eq.coo}, that  when $x_n\to x\in \mathscr R^d_0 $ (see \eqref{eq.DD}):
 $$\mathbb P  [\sup_{s\in [0,t]}\mathsf d_{\mathscr R^d_0}( L_s( x_n),L_s( x))\ge \epsilon   ]\to 0.$$   
 It thus remains to check {\rm \textbf{[O1]}}  and {\rm \textbf{[O2]}}. First assume that $\mathscr R^d_0\setminus \overline{\mathscr O}$ is non-empty. Then, we assume {\rm \textbf{[L3]}}  which is precisely {\rm \textbf{[O1]}}. Moreover  when $\mathscr O=\mathscr R^d_0$, {\rm \textbf{[O2]}} is clearly satisfied since, by {\rm \textbf{[S2]}},  $\mathbf V_{\mathbf S}$ is non constant over $\mathscr R^d_0$. 
 
% %billingsley2013
%  

 %%%%

 %%%

  \subsection{Proofs of Lemma \ref{le.finiteLcin} and of Theorem \ref{th.kL}}
  \label{sec.th.3}
  
  In this section, we prove Lemma \ref{le.finiteLcin} and  Theorem \ref{th.kL}.
  
  %%%
  
 \subsubsection{Proof of Lemma \ref{le.finiteLcin}} 
 \label{sec.prp}
  Recall that because  $\mathbf V_{\mathbf c}$  is lower bounded over $\mathscr R^d$, thanks to Proposition \ref{pr.PreKin} we have a unique strong solution $(Y_t=(x_t,v_t),t\ge0)$ to \eqref{eq.Lcin} over $\mathscr R^d\times \mathscr R^d$ which satisfies the Girsanov formula \eqref{eq.Girsanov2}. 
   Let us prove \eqref{eq.nn} for all $\mathsf y\in \mathscr R^{2d}_0$ where we recall that $\mathscr R^{2d}_0=(\mathscr R^d\setminus\{0\})\times \mathscr R^d$, see \eqref{eq.r2d0}. Note that \eqref{eq.nn} is equivalent to the fact that    for all $T\ge 0$ and  $\mathsf y\in \mathscr R^{2d}_0$,  
 \begin{equation}\label{eq.nnT}
  \mathbb P_{\mathsf y}[\forall t\in [0,T],  |x_t|>0]=1.
 \end{equation}
   By the Girsanov formula (see Proposition \ref{pr.PreKin}), \eqref{eq.nnT} is equivalent to: for all $T\ge 0$ and  $\mathsf y\in \mathscr R^{2d}_0$,
 \begin{equation}\label{eq.nnT0}
 \mathbb P_{\mathsf y}[\forall t\in [0,T],  |x_t^0|>0]=1, 
 \end{equation}
  where  $(Y_t^0=(x_t^0,v_t^0),t\ge 0)$ is  the solution to the stochastic differential equation \eqref{eq.Lcin0}. Let us prove \eqref{eq.nnT0}. Introduce to this end 
  $$\tau^0_{\{0\}}:=\inf\{t\ge 0, x_t^0=0\}.$$
  Let us show that $ \mathbb P_{\mathsf y}[\tau^0_{\{0\}}\le T]=0$ (this is exactly \eqref{eq.nnT0}). 
  \medskip
  
  \noindent
 \textbf{Step 1.} Let $T>0$ (otherwise the result is obvious). First of all we claim that for all $\eta\in (0,T]$, 
 \begin{equation}\label{eq.eta-}
  \mathbb P_{\mathsf y}[\forall t\in  [\eta, T], |x_t^0|>0]=1.
   \end{equation}
To prove \eqref{eq.eta-}, we   use~\cite[Theorem 1]{sohl2010polar} with $I= [\eta, T]$ and $N=1$ there and for the Gaussian process  $x_t(\mathsf y)= x+vt+ \int_0^t B_udu$,  $\mathsf y=(x,v)\in \mathscr R^{2d}$. 
To this end we check that \textbf{Conditions 1} and \textbf 2 in~\cite{sohl2010polar} are satisfied for the   process $(x_t(\mathsf y),t\ge 0)$ on $I$. These assumptions are easy to check and  we check them for sake of completeness. On the one hand, we have for all $0\le s\le t$, 
$$|x_t(\mathsf y)-x_s(\mathsf y)|\le |v| |t-s| + |\int_s^t B_udu|\le |v| |t-s| + |t-s|^{1/2}  (\int_s^t |B_u|^2du)^{1/2} ,$$ so that for some $C>0$, 
$$|\mathbb E\big [|x_t(\mathsf y)-x_s(\mathsf y)|^2\big ]|^{1/2}\le C |t-s|.$$ This is  \textbf{Condition 1} in~\cite{sohl2010polar}. Adopting the terminology of~\cite[p. 845]{sohl2010polar}, this implies that  $H_1=1$ and the parameter $Q$ is equal to $1$.  Introduce now $e=(e_1,\ldots,e_d)\in \mathscr R^d$ such that $\sum_j e_j^2=1$. Write 
$x=(x_1,\ldots,x_d)$, $v=(v_1,\ldots,v_d)$, and $B_u=(B^1_u,\ldots,B^d_u)$ where the $B^j$'s are independent standard real Brownian motions. 
Then, one has for all $t\in I=[\eta,T]$:
\begin{align*}
\mathbb E\Big [\Big (\sum_{j=1}^d e_j (x_j+v_jt+ \int_0^t B^j_udu\Big )^2\Big ]&= \Big |\sum_{j=1}^d e_j (x_j+v_jt)\Big |^2+ \mathbb E\Big [\Big (\sum_{j=1}^d e_j \int_0^t B^j_udu\Big )^2\Big ]\\
&= | e\cdot (x+vt)|^2+ \sum_{j=1}^d  |e_j|^2  \mathbb E\Big [\big |\int_0^t B^j_udu\big |^2\Big ]\\
&=| e\cdot (x+vt)|^2+  t^3/3 \ge \eta^3/3.
\end{align*}
 This proves \textbf{Condition 2} in~\cite{sohl2010polar} when 
 $ I=[\eta,T]$. 
Since $Q=1<d$, by~\cite[Theorem 1]{sohl2010polar} and since  the  Hausdorff measure $ \dim _{d-1}(\{0\})$ of the set $\{0\}$ in   dimension $d-1$ (also called   the  $d-1$-Hausdorff measure   of the set $\{0\}$. See for instance~\cite[Definition 4.7]{morters} for a definition) is $0$, 
%Wiki: that if d is a positive integer, the d-dimensional Hausdorff measure of Rd is a rescaling of the usual d-dimensional Lebesgue measure
we   have that  for all $\eta\in (0,T]$,
 $$\mathbb P[\exists s\in [\eta,T], x_s(\mathsf y)\in \{0\}]=0.$$ 
 This achieves the proof of \eqref{eq.eta-}. Note that we have not used yet that the initial condition satisfies $x\neq 0$.  
 
   \medskip
  
  \noindent
 \textbf{Step 2.} We now conclude the proof of the fact that $ \mathbb P_{\mathsf y}[\tau^0_{\{0\}}\le T]=0$. Fix $\mathsf y\in \mathscr R^{2d}_0$ and  $T>0$. 
 Let $\epsilon >0$. Consider   $M_\epsilon>0$ large enough such that 
 $$\mathbb P\big [\sup_{u\in [0,T]}|B_u|>M_\epsilon\big ]\le \epsilon.$$  
   Then,   one has for all $\eta\in (0,T]$,  using \eqref{eq.eta-}, $\mathbb P_{\mathsf y} [\tau^0_{\{0\}}\in [\eta,T]]\le\mathbb P_{\mathsf y} [ \exists s\in [\eta,T], |x_s|=0]= 0$. Hence, for  all $\eta\in (0,T]$,
\begin{align*}
 \mathbb P_{\mathsf y}[\tau^0_{\{0\}}\le T]&\le  \mathbb P_{\mathsf y}\big [\tau^0_{\{0\}}\in (0,\eta), \sup_{u\in [0,T]}|B_u|\le M_\epsilon\big ]+\epsilon.
 \end{align*}
 Note that when $\tau^0_{\{0\}}\in (0,\eta)$ and $\sup_{u\in [0,T]}|B_u|\le M_\epsilon$, we have, since  $x_{\tau^0_{\{0\}}}(\mathsf y)=0$, 
\begin{align*}
0<|x|&\le |v| \tau^0_{\{0\}}+ \int_0^{\tau^0_{\{0\}}}|B_u|du \le  \eta (|v|  + M_\epsilon).
 \end{align*} 
Choose $\eta_\epsilon  \in (0,T]$ such that $\eta_\epsilon (|v|  + M_\epsilon)<|x|$. For such a $\eta_\epsilon>0$, one has  
$$ \mathbb P_{\mathsf y}\Big [\tau^0_{\{0\}}\in (0,\eta_\epsilon), \sup_{u\in [0,T]}|B_u|\le M_\epsilon\Big ]=0$$ and thus $\mathbb P_{\mathsf y}[\tau^0_{\{0\}}\le T]\le \epsilon$. This proves that $ \mathbb P_{\mathsf y}[\tau^0_{\{0\}}\le T]=0$. The proof of  \eqref{eq.nnT0}, and hence also the one of  \eqref{eq.nn},  are complete. The last statements in Lemma \ref{le.finiteLcin} are consequences  of \eqref{eq.nn}.

 %%%%

   \subsubsection{Proof of   Theorem \ref{th.kL}} 
   Recall that we work here on the space $\mathscr R^{2d}_0$, see \eqref{eq.r2d0}.  
    Assume  that  $-\nabla \mathbf V_{\mathbf c}$ (resp. $\mathbf V_{\mathbf S}$)  satisfies {\rm \textbf{[c1]}}   (resp. {\rm \textbf{[S1]}} or {\rm \textbf{[S2]}}). % 
 Let $(Y_t,t\ge 0)$ be the strong solution to \eqref{eq.Lcin} (see Proposition \ref{pr.PreKin}). Let us  denote by $ 
 \mathscr  L^{kL}=  \frac{1}{2}\Delta_v + v\cdot \nabla_x -\nabla \mathbf V_{\mathbf c}\cdot \nabla_v-\gamma v\cdot \nabla_v$ 
 the infinitesimal generator of the diffusion $(Y_t,t\ge 0)$. 
Since $\mathscr  L^{kL} \mathbf H\le c\mathbf H$ over $\mathscr R^{2d}$,  it holds for all $\mathsf y\in \mathscr H_R$ and $t\ge 0$, 
\begin{equation}\label{eq.energyLcin}
\mathbb P_{\mathsf y}[\sigma^{kL}_{\mathscr H_R}\le t]\le \frac{e^{ct}}{R} \mathbf H(\mathsf y), 
\end{equation}
where for $R>0$, $\mathscr H_R:= \{\mathsf y\in \mathscr R^{2d}, \mathbf H(\mathsf y)< R\}$ is an open bounded subset of $\mathscr R^d$ and $\sigma^{kL}_{\mathscr H_R}:=\inf\{t\ge 0, Y_t\notin \mathscr H_R\}$.

To prove Theorem \ref{th.kL}, and in view of Theorem \ref{th.G} and  Theorem \ref{th.SFG}, it is enough to check the  conditions \textbf{[SF$_0$]}, \textbf{[P$_{t=s}$]}, \textbf{[D$_{\rm e}$]}, \textbf{[$I\mathbf V_{\mathbf S}$]}, \textbf{[B]}, \textbf{[Ti]},\textbf{[L$_{{\rm yap}}$]}, and  \textbf{[C$_{{\rm traj}}$]} that we rewrite for the kinetic Langevin process \eqref{eq.Lcin}: 
 \begin{enumerate}
 \item[]\textbf{[SF$_0$]} The semigroup $(U_t^{kL},t\ge 0)$ of the process $(Y_t,t\ge 0)$ is strongly Feller. 
   \item[] \textbf{[P$_{t=s}$]}  For all $t\ge 0$, the mapping $\mathsf y\in \mathscr R^{2d} \mapsto Y_t(\mathsf y)$ is continuous in probability. 
    \item[]\textbf{[D$_{\rm e}$]} For all $t> 0$ and $\mathsf y\in \mathscr R^{2d}$, $ Y_t(\mathsf y)$ has density w.r.t. the Lebesgue measure $d\mathsf y$ over $\mathscr R^{2d}$.

     \item[] \textbf{[$I\mathbf V_{\mathbf S}$]} As $\mathsf y_n \to \mathsf y \in \mathscr R^{2d}_0$,  for all $t\ge 0$,  $\int_0^t \mathbf V_{\mathbf S}(x_s(\mathsf y_n))ds\overset{\mathbb P}{\to} \int_0^t \mathbf V_{\mathbf S}(x_s(\mathsf y))ds$.

     \item[] \textbf{[B]}  For all $\delta>0$ and any compact subset $K$ of $\mathscr R^{2d}$, 
     $$\lim_{s\to 0^+}\sup_{\mathsf y=(x,v)\in K}  \mathbb P_{\mathsf y}[\sigma^{kL}_{B(x,\delta)}\le s]=0,$$ 
 where $\sigma^{kL}_{B(x,\delta)} (\mathsf y) :=\inf\{t\ge 0, x_t(\mathsf y)\notin {B(x,\delta)}\}$.  
     \item[] \textbf{[Ti]} The semigroup $(Q_t^{kL},t\ge 0)$ is topologically irreducible over $\mathscr D=\mathscr O\times \mathscr R^d$.  
 
      \item[]  \textbf{[L$_{{\rm yap}}$]}  There exists a  $\mathcal C^{1,2}$ function $\mathbf W:\mathscr R^{2d}\to [1,+\infty)$ (which thus belongs to the  extended domain of the generator of $(U_t^{kL},t\ge 0)$)    such that  for all $\mathsf y=(x,v)\in  \mathscr R^{2d}_0$ satisfying   either  $|\mathsf y|\to +\infty$ or $\mathsf y\to  \partial \mathscr R^{2d}_0=\{0\}\times \mathscr R^d$ (i.e. when $x\to 0$), 
      $$ { \mathscr L^{kL}  \mathbf W(\mathsf y) }/{\mathbf W(\mathsf y) } -p\mathbf V_{\mathbf S}(x) \to -\infty, \ \forall p>1.$$ 
      \item[] \textbf{[C$_{{\rm traj}}$]} The mapping  $\mathsf y \in \mathscr R^{2d}_0\mapsto \mathbb P[Y_{[0,t]}(\mathsf y)\in \cdot ]\in \mathcal P(\mathcal C([0,t],\mathscr R^{2d}_0))$ is continuous for the weak topology. 
       \end{enumerate}
       
      We check these conditions. First of all, using the same arguments as those  used in the proof of Theorem \ref{th.SF1} (see more precisely Step A.1 there), one proves  \textbf{[SF$_0$]} for the kinetic Langevin process \eqref{eq.Lcin}. Note that an alternative proof consists in using the Girsanov formula \eqref{eq.Girsanov2} as we did   for the overdamped Langevin process, see \eqref{eq.mm} and the lines below. 
%in   \cite{Wu2001} (see also \cite{guillinqsd3}). 

       Condition \textbf{[C$_{{\rm traj}}$]}  is a consequence of the fact that 
      for all $\mathsf y\in \mathscr R^{2d}_0$, we have $\mathbb P_{\mathsf y}[Y_{[0,t]} \in \mathcal C([0,t], \mathscr R^{2d}_0)]=1$ (see Lemma~\ref{le.finiteLcin}), together with the fact that for all   $t\ge 0$ and  all sequence $(\mathsf y_n)_n$ in $\mathscr R^{2d}$ such that $\mathsf y_n\to \mathsf y\in \mathscr R^{2d}$ as $n\to +\infty$,  for all $\epsilon>0$,  
 \begin{equation}\label{eq.pr=}
 \mathbb P  [\sup_{s\in [0,t]}| Y_s( \mathsf y_n)-Y_s( \mathsf y)|\ge \epsilon   ]\to 0 \ \text{ as $n\to +\infty$}.
 \end{equation} 
Equation \eqref{eq.pr=} is a consequence of \eqref{eq.energyLcin} together with the fact that the coefficients of the stochastic differential equation \eqref{eq.Lcin} are locally Lipschitz (this is proved e.g. by adapting the proof of  \cite[Proposition 2.2]{guillinqsd3}). Note that \eqref{eq.pr=}  implies \textbf{[P$_{t=s}$]}. Condition \textbf{[D$_{\rm e}$]} is a consequence of the Girsanov formula stated in Proposition \ref{pr.PreKin} (see \eqref{eq.Girsanov2}). Condition \textbf{[$I\mathbf V_{\mathbf S}$]} is proved as in   Step A.3 in  the proof of Theorem \ref{th.SF1} using that $\mathbb P_{\mathsf y}[Y_{[0,T]}\in \mathcal C([0,T], \mathscr R^{2d}_0)]=1$,  \eqref{eq.pr=}, and  the continuous mapping theorem. 
%Let $t\ge 0$.  
%Define the  mapping 
%  $$G: \gamma=(x,v)\in \mathcal C([0,t],\mathscr R^{2d})\mapsto \int_0^t \mathbf V_{\mathbf S}(x_s)ds.$$  Using the dominated convergence theorem, the function $G$ is continuous at any $\gamma\in \mathcal C([0,t],\mathscr R^{2d}_0)$.
%  Let $(\mathsf y_n)\subset  \mathscr R^{2d}_0$   such that $\mathsf y_n\to \mathsf y\in \mathscr R^{2d}_0$. Set $Y_n:=Y_{[0,t]}(\mathsf y_n)$ and $Y:=Y_{[0,t]}(\mathsf y )$.   Then, since $\mathbb P_{\mathsf y}[Y_{[0,T]}\in \mathcal C([0,T], \mathscr R^{2d}_0)]=1$, using \eqref{eq.pr=} and  the continuous mapping theorem that 
%  $G(Y_n)\to G(Y)$ in $\mathbb P$-probability, which proves \textbf{[$I\mathbf V_{\mathbf S}$]}.  
Since the coefficients of \eqref{eq.Lcin} are locally Lipschitz, the condition   \textbf{[B]}    is proved as in Lemma 2.4 in \cite{guillinqsd3},  see also its note there for other ways to prove it.  Let us mention that one can also prove \textbf{[B]}   with the same  arguments we used  to show \eqref{eq.contu} above and which we recall  are based on Itô calculus with a suitable Lyapunov function. 
  Condition \textbf{[Ti]} is a consequence of the fact that the killed semigroup of the  process $(Y_s,s\in [0,\sigma^{kL}_{\mathscr D}))$ is topologically irreducible over $\mathscr D$ (to see this, use the Girsanov formula and the fact that the killed semigroup  $(Y^0_s,s\in [0,\sigma^{Y^0}_{\mathscr D}))$  is  topologically irreducible over $\mathscr D$, see e.g.~\cite{guillinqsd}) together with the fact that    $\int_0^t \mathbf V_{\mathbf S}(x_s(\mathsf y))ds<+\infty$ a.s. for all $\mathsf y\in \mathscr R^{2d}_0$ (see Lemma~\ref{le.finiteLcin}). 
 Let us now check condition  \textbf{[L$_{{\rm yap}}$]}.  
  We choose the Lyapunov function constructed in~\cite{Wu2001}. Define  $
\mathsf F (x,v)=a  \, \mathbf H  (x,v)+b\, v\cdot     \mathsf G(x)$ ($a,b>0$), 
where $\mathsf G$ is the $\mathcal C^1$ function
$$
 x\in \mathscr R^d\mapsto \mathsf G(x)=\frac{x}{|x|}(1-\chi(x)), 
$$
where $\chi$ is defined just before \eqref{eq.W}.
 Note that both $\mathsf G$ and its gradient are bounded over $\mathscr R^d$. In particular $\mathsf F$ is lower bounded. 
For all $\mathsf y=(x,v)\in \mathscr R^{2d}$, we then define as in~\cite{Wu2001}:
 \begin{equation}\label{eq.Wkl}
\mathbf W  (x,v)=\exp\big[{\mathsf F(x,v) -\inf_{\mathscr R^{2d}} \mathsf F  }\big].
\end{equation} 
We then have  for all $\mathsf y\in \mathscr R^{2d}$,
\begin{align*}
&-{ \mathscr L^{kL}  \mathbf W(\mathsf y) }/{\mathbf W(\mathsf y) } \\
&\quad =- \mathscr L^{kL} \mathsf F(\mathsf y) - \frac 12 |\nabla _v \mathsf F(\mathsf y)|^2 \\
&\quad =-\frac{a d}2+ a\,  \gamma |v|^2- b \,  v\cdot  \nabla \mathsf G(x) v +b   \, \gamma \,  v \cdot \mathsf G(x)+ b\,  \mathsf G(x)\cdot \nabla  \mathbf V_{\mathbf c}(x)-\frac 12\,  |av + b\mathsf G(x)|^2.
\end{align*}
Hence, for some constants $c_1,c_2>0$,  such that $-{ \mathscr L^{kL}  \mathbf W }/{\mathbf W } $ is lower bounded by 
$$    |v|^2\Big[a\big[\gamma- \frac a2\big] - b |\nabla \mathsf G|_\infty\Big]-c_1   b (a+1) \, |v| + b\,  \mathsf G(x)\cdot \nabla  \mathbf V_{\mathbf c}(x)-c_2 .  
$$
 Choose $a>0$ such that $a<2\gamma$ and then $b>0$ such that  $ a\big[\gamma- \frac a2\big] - b>0$. Then, when  
  $|\mathsf y|\to   +\infty$,   $- { \mathscr L^{kL}  \mathbf W(\mathsf y) }/{\mathbf W(\mathsf y) } \to +\infty$. Hence, for $\mathsf y=(x,v)\in \mathscr R^{2d}_0$, when $|\mathsf y|\to +\infty$ or $\mathsf y\to  \partial \mathscr R^{2d}_0=\{0\}\times \mathscr R^d$ (i.e. when $|x|\to 0$), $ { \mathscr L^{kL}  \mathbf W(\mathsf y) }/{\mathbf W(\mathsf y) } - p\mathbf V_{\mathbf S}(x) \to -\infty$, for all $p>1$. \\

 Theorem \ref{th.kL} is  a consequence of  Theorems~\ref{th.G} and~\ref{th.SFG}.  Note that {\rm \textbf{[O1]}} is a consequence of the fact that when   $\mathscr R^d_0\setminus \overline{\mathscr O}$ is non-empty, there exist  a non empty open ball $B\subset \mathscr R^d_0\setminus \overline{\mathscr O}$,    $\mathsf y_0\in \mathscr D= \mathscr O\times \mathscr R^d$, and   $t_0>0$ (actually this is true for all $\mathsf y_0\in \mathscr D$ and $t_0>0$) such that 
$$\mathbb P_{\mathsf y_0}[Y_{t_0}\in B \times \mathscr R^d]>0.$$
Indeed, this implies that   $0<\mathbb P_{\mathsf y_0}[Y_{t_0}\in B\times \mathscr R^d]\le \mathbb P_{\mathsf y_0}[ \sigma^{kL}_{\mathscr D}\le t_0]$.

%%%
\begin{remarks}\label{re.th3a} As already stated in the note \ref{no.kL},  the assertions of Theorem \ref{th.kL} are still valid in the case when   $ \mathbf V_{\mathbf c}$ is only  $\mathcal C^1$. To show this, one applies again Theorem \ref{th.G} over $\mathscr R^{2d}_0$.  We give    the proof of this claim in broad terms, the details are left to the reader since they rely on similar arguments as those used so far in this work.   First, for all $\mathsf y=(x,v)\in \mathscr R^{2d}$, there is a unique weak solution  to \eqref{eq.Lcin} starting at $\mathsf y$,  and  the Girsanov formula \eqref{eq.Girsanov2}  still holds true, see~\cite[Lemma 1.1]{Wu2001}. Note that \eqref{eq.energyLcin} and \eqref{eq.nn} still hold. In addition, with the same arguments as those used in \cite[Proposition 2.9]{guillinqsd2}, one shows {\rm \textbf{[C$_{{\rm traj}}$]}} which actually holds over $\mathscr R^{2d}$. Using the Girsanov formula and auxiliary results on the process $(Y_t^0,t\ge 0)$ shown in the proof of~\cite[Proposition 1.2]{Wu2001}, one proves {\rm \textbf{[SF$_0$]}} (which actually holds over $\mathscr R^{2d}$)   and {\rm \textbf{[SF$_1$]}} over   $\mathscr D$, a subdomain of  $\mathscr R^{2d}_0$.   The proof of {\rm \textbf{[B]}} is made adapting the arguments used in the proof of~\cite[Lemma 1]{guillin2024large} or those given in  the note of Lemmata 2.4 and 2.5 in \cite{guillinqsd3}. 
\end{remarks}

%%%%

 \subsection{Proof of Theorem \ref{th.Bn}}
 \label{sec.4} 
Recall the definition of the state space $\mathscr E$ and the Lévy process $(\Theta_t,t\ge 0)$ over $\mathscr E$, see \eqref{eq.eD} and \eqref{eq.theta}.
Recall also \eqref{eq.HS2} and  Lemma \ref{le.Scc}. To prove Theorem \ref{th.Bn} we use Theorems~\ref{th.G} and~\ref{th.SFG}. 
Recall that $\mathbf U_{\mathbf S}(\mathsf x)\to +\infty$ if and only if $|\mathsf x|\to +\infty$ or $\mathsf x\to \partial \mathscr E$ ($\mathsf x\in \mathscr E$).
By Assumption \textbf{[L1]},  for any $t>0$,  $\Theta_t$ has density w.r.t. the Lebesgue measure over $\mathscr R^{dn}$  (this is \textbf{[D$_e$]}) and its semigroup is thus strong Feller (this is \textbf{[SF$_0$]}). Let $t\ge 0$. Moreover, as $\mathsf x_n\to \mathsf x\in \mathscr E$, it holds in probability (and actually a.s.)
$$\int_0^t \mathbf U_{\mathbf S}(\Theta_s(\mathsf x_n))ds\to \int_0^t \mathbf U_{\mathbf S}(\Theta_s(\mathsf x))ds \ \text{ (this is \textbf{[$I\mathbf V_{\mathbf S}$]})}.$$
Indeed, on the one hand, it holds a.s. $ \sup_{s\ge 0} |\Theta_s(\mathsf x_n)-\Theta_s(\mathsf x)|= |\mathsf x_n-\mathsf x|$. On the other hand, recall (see the proof of Lemma \ref{le.Scc}),   that there exist a.s. $\epsilon>0$ and $M>0$ such that $ |\Theta_s(\mathsf x)|\le M$  and ${\rm dist}\, (\Theta_s(\mathsf x), \partial \mathscr E )\ge \epsilon$, for all $s\in [0,t]$. Hence, there exists a.s. $n_0\ge 1$, for all $n\ge n_0$ and $s\ge 0$, $|\Theta_s(\mathsf x_n)|\le 2M$ and ${\rm dist}\, (\Theta_s(\mathsf x_n), \partial \mathscr E )\ge \epsilon/2$. The result follows from   the dominated convergence theorem.  Note also that the previous analysis implies  \textbf{[C$_{{\rm traj}}$]}. 
Consequently, using Theorem \ref{th.SFG}, the non-killed semigroup is strongly Feller, i.e. for all $t>0$ and $f\in b\mathcal B(\mathscr E)$, the function 
$$\mathsf x\in \mathscr E \mapsto T^B_tf(\mathsf x)=\mathbb E_{\mathsf x}\Big[f(\Theta_t)   \, e^{  - \int_0^t  \mathbf U_{\mathbf S}(\Theta_s)ds}    \Big] \text{ is  continuous (this is \textbf{[SF$_1$]})}.$$
On the other hand, for all compact subset $K$ of $\mathscr E$ and all $\delta>0$, we have that (recalling that $B(\mathsf x,\delta)$  is the open ball in $\mathscr R^{dn}$ of radius $\delta>0$ and centered at $\mathsf x$):
$$  \lim_{s\to 0^+}\sup_{\mathsf x\in K}  \mathbb P_\mathsf x[\sigma^\Theta_{B(\mathsf x,\delta)}\le s]= \lim_{s\to 0^+}   \mathbb P_0[\sigma^\Theta_{B(0,\delta)}\le s]=0,$$ which implies {\rm \textbf{[B]}}. We thus deduce \textbf{[SF$_2$]} using  Theorem \ref{th.SFG}.   Furthermore \textbf{[L$_{{\rm yap}}$]} is satisfied with the constant function  $\mathbf 1$. Condition \textbf{[Ti]} is a consequence of \textbf{[L4]} together with the fact that for all $\mathsf x\in \mathscr E$, $\int_0^t \mathbf U_{\mathbf S}(\Theta_s(\mathsf x))ds$ is finite almost surely (see Lemma \ref{le.Scc}). Finally, note that {\rm \textbf{[O1]}} is precisely \textbf{[L5]} and that $\mathbf U_{\mathbf S}$ is non constant over $\mathscr E$. The proof of Theorem \ref{th.Bn} is complete using Theorem \ref{th.G}.

%%%%

 \medskip
 
  \medskip
 
\begin{center}
\textbf{Appendix: on Assumptions \textbf{[L1]}$\to$\textbf{[L5]}}
\end{center}
\medskip

\noindent
In this section we give examples of Lévy processes satisfying \textbf{[L1]}, \textbf{[L2]}, and    \textbf{[L3]} (see the second model). We also give examples of such processes such that the process $\Theta$ (see \eqref{eq.theta}) satisfies  \textbf{[L4]} and \textbf{[L5]} (see the fourth model).  
\medskip

\noindent
\textbf{On Assumptions \textbf{[L1]}}.
In the mathematical  literature, several  conditions exist  ensuring \textbf{[L1]} for a Lévy process, see e.g.~\cite{hartman1942,tucker1962absolute,sato,knopova2013note,kuhn2019transition} and references therein. For instance, the following simple well-known conditions imply that the Lévy process has a density w.r.t. the Lebesgue measure for all $t>0$:
\begin{enumerate}
\item[\textbf -] Its Gaussian covariance matrix has full rank, or if its Lévy measure $\nu$ is absolutely continuous w.r.t. the Lebesgue measure  and $\nu(\mathscr R^d\setminus \{0\})=+\infty$, see e.g.~\cite[Theorem 27.7]{sato}.  
\item[\textbf -] An isotropic Lévy process in $\mathscr R^d$ ($d\ge 2$) which is  not a compound Poisson process, see~\cite[Eq. (4.6)]{zabczyk1970theorie}. 
\item[\textbf -] A Lévy process without Gaussian component  such that
 $
 \lim_{|\xi|\to +\infty}\frac{\Re (\Psi(\xi))}{\ln(1+|\xi|)}=+\infty$, 
 where  $\Psi$ is the   characteristic exponent and $\Re(z)$ the real part of a complex number $z$, see~\cite{hartman1942,knopova2013note}. 
 \item[\textbf -] A  subordinate Brownian motion $(B_{\ell_t},t\ge 0)$ where     $( \ell_t,t\ge 0)$ is a subordinator 
with infinite lifetime (independent from $(B_t,t\ge 0)$), see~\cite[Lemma 3.1]{kusuoka2014smoothing}. 
\end{enumerate}

%In particular it is well known that if  the Brownian coefficient is positive, or when the mass of the absolutely continuous part of the Lévy measure is infinite, the Lévy process has a density w.r.t. the Lebesgue measure. 
 
 %In the literature, several  conditions exist  ensuring \textbf{[L2]} (resp.  \textbf{[L3]}) for a Lévy process, see e.g.~\cite{tucker1962absolute,knopova2013note,kaleta2015pointwise} (resp.~\cite{simon1999petites,simon2000support,kulyk22})  and references therein for explicit conditions. 

 \medskip
 
 \noindent
\textbf{On Assumptions \textbf{[L2]} and \textbf{[L3]}}.
The following proposition is a way to check \textbf{[L2]} and \textbf{[L3]} for a Lévy process. 
 
 \begin{prop}\label{pr.Topo}
 Let $(L_t,t\ge 0)$ be a purely jump Lévy process over $\mathscr R^d$ such that its  Lévy measure $\nu$ has full topological support over $\mathscr R_0^d$, i.e. for all $y\in \mathscr R^d_0$ and all $r\in (0,|y|)$, $\nu(B(y,r))>0$, and such that for some $\beta >0$, 
 $$\int_{|u|\le 1}|u|^\beta \nu(du)<+\infty.$$  Let $\mathscr O$ be a subdomain of $\mathscr R^d$. 
 Let  $T>0$ and $x,z\in \mathscr O$. Then,  for all $\epsilon>0$, 
 \begin{equation}\label{eq.Topo}
 \mathbb P_x[|L_T-z|<\epsilon, T<\sigma _{\mathscr O}]>0,
 \end{equation}
 where $\sigma_{\mathscr O}:=\inf\{t\ge 0, L_t\notin \mathscr O\}$. 
 In addition, if $\mathscr R^d\setminus \overline{\mathscr O}$ is nonempty, for all $x\in \mathscr O$, $\mathbb P_x[\sigma_{\mathscr O}<+\infty ]>0$.  
  \end{prop}

\begin{proof}  
We will use~\cite[Theorem 2.1]{kulyk22} and for this reason, we adopt the notation of~\cite[Sections 2.1 and  2.2]{kulyk22}.
Note that Assumptions $\mathbf H_1$ and $\mathbf H_2$ there are satisfied here for the process  $(L_t,t\ge 0)$ with $r\equiv 0$, $b\equiv 0$, $\sigma\equiv 1$, and $c(x,u)\equiv u$.  Fix $\epsilon, T>0$ and $x,z\in \mathscr O $.   In view of~\cite[Theorem 2.1]{kulyk22}, we aim at constructing  $\phi\in  \mathbf S^{\text{const}}_{0,T,x}$ such that  $\overline{\rm Ran}\, \phi \subset \mathscr O$ with $\phi_0=x$, $|\phi_T-z|<\epsilon/2$. By~\cite[Theorem 2.1]{kulyk22}, one then has in particular that: 
 \begin{equation}\label{eq.Kulyk}
 \mathbb P_x[ d_T(L,\phi)<\epsilon']>0, \ \forall \epsilon'>0,
 \end{equation}
  where $ d_T$ is the Skorokhod metric  of  $ \mathcal D([0,T], \mathscr  R^d)$, see e.g.~\cite[Section 12]{billingsley2013}. 
   First note that for any $r_0>0$ and $a\neq b\in \mathscr R^d$, $J(a,B(b,r_0))= \nu(B(b-a,r_0))\in (0,+\infty)$ if $r_0<|b-a|$. Therefore $b\neq a\Rightarrow b\in \text{supp}(J(a,\cdot))$ (i.e. the jump from $a$ to $b$ is admissible). 
In our setting, we have  $\tilde b=- \mathbf b$  is a constant 
 where $\mathbf b=\int_{|u|\le 1}u_L \nu(du)$ and where $u_L$ is the orthogonal projection of $u$ on the \textit{integrability subvector}  space $$L=\big \{\ell \in \mathscr R^d, \int_{|u|\le 1}|u\cdot \ell| \nu(du)<+\infty\big \}.$$ 
It is then  not difficult\footnote{There are many ways to do it, choosing e.g. $f_t\equiv 0$ in~\cite[Eq. (7)]{kulyk22}.} to construct a   curve  $\phi\in  \mathbf S^{\text{const}}_{0,T,x}$ such that  $\overline{\rm Ran}\, \phi \subset \mathscr O$,  $\phi_0=x$, and $|\phi_T-z|<\epsilon/2$ ($\phi$ is usually called a \textit{control} curve).    
 Assume now that $\epsilon'>0$ is small enough (say $\epsilon'\in (0,\epsilon_\phi)$, $\epsilon_\phi\in (0,\epsilon)$) such that $ d_T(f,\phi)<\epsilon'/2$ implies 
 that $\overline{\rm Ran} \, f\subset \mathscr O$\footnote{Indeed,  use  the distance $d(A,B)=\sup_{a\in A} d(a,B)=d(\bar A,\bar B)$, $A,B\subset \mathscr R^d$. Since $d_T(f,\phi)<\epsilon'/2$, there exists  a strictly increasing and  continuous curve $\lambda$ from  $[0,T]$  onto itself such that 
 $$\sup_{u\in [0,T]}|f_u-\phi_{\lambda_u}|<\epsilon'/2.$$  Thus,   we have $d({\rm Ran} \, f,{\rm Ran} \, \phi)= \sup_{t\in [0,T]} \inf_{s\in [0,T]}|f_t-\phi_s|\le \sup_{t\in [0,T]}|f_t-\phi_{\lambda_t}|<\epsilon'/2$.} (in particular $|f_T-z|\le |f_T-\phi_T| + |\phi_T-z|\le  d_T(f,\phi)+\epsilon/2<\epsilon$).  Then, using  \eqref{eq.Kulyk} with such a small  $\epsilon'>0$,  $\mathbb P_x[\{|L_T-z|<\epsilon\} \cap \{ \overline{\rm Ran}\, L_{[0,T]}\subset  \mathscr O\}]>0$. 
 Therefore,  \eqref{eq.Topo} holds.     The second statement in Proposition \ref{pr.Topo} is  easy to obtain with the same arguments.  
\end{proof}

 \noindent
\textbf{Examples}. 
As a conclusion, due to their importance both in theory and in applications, one can easily check with the discussion above that  the following examples of (isotropic) Lévy processes over $\mathscr R^d$ ($d\ge 2$) satisfy \textbf{[L1]}, and also  \textbf{[L2]}-\textbf{[L3]} for any subdomain $\mathscr O$ of $\mathscr R^d$:
 \begin{enumerate}
 \item[\textbf -] The standard  Brownian motion.
\item[\textbf -] The rotationally invariant $\alpha$-stable processes,  $ \alpha\in (0,2)$.   If $\alpha\in (0,2)$, these are  purely jump Lévy processes 
%(i.e. its Lévy triplet is given by $(0,0,\nu_0)$) 
with characteristic exponent   $\Psi(u)=|u|^\alpha$, see~\cite{Ryznar,ascione2024bulk,bogdan-book}.
%  
%and Lévy measure  $\nu_{0,\alpha}(dx)=F_{0,\alpha}(x)dx$, $F_{0,\alpha}(x)= C_\alpha dx/|x|^{\alpha+d}$, $C_\alpha>0$, 
\item[\textbf -] The rotationally invariant relativistic $\alpha$-stable processes,  $ \alpha\in (0,2)$, see~\cite{Ryznar,ascione2024bulk,bogdan-book}. For $m>0$, these are purely jump Lévy processes with characteristic exponent   $\Psi(u)=(|u|^\alpha+m^{2/\alpha})^{\alpha/2}-m$. 
% whose \textcolor{red}{Lévy measure $\nu_{m,\alpha}(dx)=F_{m,\alpha}(x)dx$ behaves like $F_{0,\alpha}(x)dx$ for small $|x|$}, 
\item[\textbf -]   The rotationally invariant geometric $\alpha$-stable processes,   $ \alpha\in (0,2)$. These are purely jump Lévy processes with characteristic exponent   $\Psi(u)=\log(1+|u|^\alpha)$, see~\cite{vsikic2006potential,bogdan-book}.
% and   \textcolor{red}{Lévy measure $\nu_\alpha(dx)= C|x|^{-d} (1+|x|)^{-\alpha}dx$}.
\item[-]  The rotationally symmetric geometric $2$-stable  process (also called the \textit{variance gamma process} in some finance literature). Its characteristic exponent   is $\Psi(u)=\log(1+|u|^2)$, see~\cite[Example 4.7]{kaleta2015pointwise}. 
% and   \textcolor{red}{Lévy measure $\nu_\alpha(dx)= C|x|^{-d} e^{-|x|}(1+|x|)^{(d-1)/2}dx$}.
\item[\textbf -] The jump-diffusion processes. These are   Lévy processes with characteristic exponent   $\Psi(u)= |u|^2+|u|^\alpha$, $\alpha\in (0,2)$, see~\cite{chen2011heat}. 
% and   \textcolor{red}{ Jump measure = celle du alpha stable}
\end{enumerate}  
Note that the first  five  processes above are subordinate Brownian motions.   
This list is clearly  a non exhaustive list and  there are   many other Lévy processes satisfying Assumptions \textbf{[L1]}, \textbf{[L2]}, and \textbf{[L3]}. 
 \medskip
 
 \noindent 
 \textbf{On Assumptions \textbf{[L4]}-\textbf{[L5]}}.  We consider here the fourth model (see \eqref{eq.FK-Ll} and Theorem \ref{th.Bn}).  Conditions   \textbf{[L4]}- \textbf{[L5]} hold when $(L^{ j}_t,t\ge 0)$'s are independent copies of a standard Brownian motion since in this case $\Theta(0)$ is also a standard Brownian motion (see \eqref{eq.theta}). Let us mention that  when $(L^{ j}_t,t\ge 0)$'s are independent copies of one of the examples  given just above in the appendix,    \textbf{[L4]}- \textbf{[L5]} are satisfied (one can  use e.g.~\cite{kulyk22} and explicit constructions of controls).  Let us for instance prove that \textbf{[L4]}- \textbf{[L5]} hold when $(L^{ j}_t,t\ge 0)$'s are independent copies of a jump-diffusion process. More precisely, let $(B^{\mathfrak 1},\ldots,B^{\mathfrak n},Z^{\mathfrak 1},\ldots,Z^{\mathfrak n})$ be $2n$ ($n\ge 2$)  independent $\mathscr R^d$-processes such that each $B^i$ is a $\mathscr R^d$-standard Brownian motion and each $Z^{\mathfrak i}$ is a $\mathscr R^d$-rotationally invariant $\alpha$-stable processes ($\alpha\in (0,2)$).  Then $\Theta=(B^{\mathfrak 1}+Z^{\mathfrak 1},  \ldots,B^{\mathfrak n}+ Z^{\mathfrak n})=  \mathfrak B+ \mathfrak Z $ where $\mathfrak  B=(B^{\mathfrak 1},  \ldots,B^{\mathfrak n})$ is a $\mathscr R^{dn}$-standard Brownian motion  and $\mathfrak Z =( Z^{\mathfrak 1},  \ldots, Z^{\mathfrak n})$. Let $\mathscr U$ be a subdomain of $\mathscr E$ (see \eqref{eq.eD}), $T>0$, and $\mathsf x,\mathsf z\in \mathscr U$. Pick a smooth curve $\Phi:[0,T]\to \mathscr U$ joining $\mathsf x$ to $\mathsf z$. Note that for any $\epsilon>0$,
 $$\mathbb P\big [\sup_{t\in [0,T]}|\mathsf x+\mathfrak B_t-\Phi_t|\le \epsilon\big ]>0.$$  Let now  $\epsilon_0 \in(0,1)$ be such that:  for any $\epsilon\in (0,\epsilon_0)$ and   for any càdlàg curve $f:[0,T]\to \mathscr R^{dn}$, $\sup_{t\in [0,T]}|f_t-\Phi_t|\le \epsilon \Rightarrow \overline{\rm Ran}\, f\subset \mathscr U$.  
  On the other hand, for any $\epsilon>0$, $\mathbb P[\sup_{t\in [0,T]}|Z^{\mathfrak 1}_t|^2<\epsilon]>0$. 
  % utiliser le thm de  Kulyk, facile de faire une courbe qui saute plein de fois tous les t_0 petits (mais qui dep que de la taille de la boule) et qui reste dans une boule. 
   Thus,    one has:
  \begin{align*}
  0&< \mathbb P\big [\sup_{t\in [0,T]}|\mathsf x+\mathfrak B_t-\Phi_t|\le \epsilon_0/2 \big ] \cdot \Pi_{i=1}^n \mathbb P [\sup_{t\in [0,T]}|Z^{\mathfrak i}_t|^2<\epsilon_0^2/(16n)  ]\\
  &= \mathbb P\big [\forall i, \sup_{t\in [0,T]}|Z^{\mathfrak i}_t|^2<\epsilon_0^2/(16n) , \sup_{t\in [0,T]}|\mathsf x+\mathfrak B_t-\Phi_t|\le \epsilon_0/2 \big ] \\
  &\le \mathbb P_0[\sup_{t\in [0,T]}|\underbrace{\mathsf x+\mathfrak B_t+\mathfrak Z_t}_{\Theta_t(\mathsf x)} -\Phi_t|\le 3\epsilon_0/4 ] .
  \end{align*}
 In conclusion this shows that \textbf{[L4]} is satisfied. A similar argument proves that \textbf{[L5]} is also satisfied for this process.

 \medskip

\noindent
 \textbf{Acknowledgement.}\\
 {\small A. Guillin   has benefited from a government grant managed by the Agence Nationale
de la
Recherche under the France 2030 investment plan ANR-23-EXMA-0001 and is supported by the ANR-23-CE-40003, Conviviality. 
B.N. is  supported by  the grant  IA20Nectoux from the Projet I-SITE Clermont CAP 20-25 and   by  the ANR-19-CE40-0010, Analyse Quantitative de Processus M\'etastables (QuAMProcs).}

 %%%%%

 %%%
 
 %%
 
% \section{Degenerate processes with singular Schrödinger  potential}

 {\small
 %%%%
 \bibliography{FKac} %You need to replace "rsc" on this
%line with the name of your .bib file
 
\bibliographystyle{plain}%the RSC's .bst file
 
}

 \end{document}